\documentclass[twoside, 11pt]{article}
\usepackage{amsmath,amssymb,amsthm,mathrsfs}
\usepackage[margin=2cm]{geometry} 
\usepackage{lipsum}
\usepackage{titlesec,hyperref}
\usepackage{fancyhdr}
\usepackage[numbers,sort&compress]{natbib}
\usepackage{color}

\fancypagestyle{plain}
{
\fancyhf{}

}

\titleformat{\subsection}{\it}{\thesubsection.\enspace}{1.5pt}{}
\titleformat{\subsubsection}{\it}{\thesubsubsection.\enspace}{1.5pt}{}

\newtheorem{theorem}{Theorem}[section]
\newtheorem{proposition}[theorem]{Proposition}

\newtheorem{remark}{Remark}[section]
\newtheorem{lemma}[theorem]{Lemma}
\numberwithin{equation}{section}

\allowdisplaybreaks

\def\th2{\frac{\theta}{2}}

\begin{document}
\title{Local-in-time Well-posedness of  Boundary Layer System for the
Full Incompressible MHD Equations by Energy Methods \hspace{-4mm}}
\author{Jincheng Gao \quad Boling Guo  \quad Daiwen Huang\\[10pt]
\small {Institute of Applied Physics and Computational Mathematics,}\\
\small {100088, Beijing, P. R. China}\\[5pt]
}

\footnotetext{Email: \it gaojc1998@163.com(J.C. Gao),
                     \it gbl@iapcm.ac.cn(B.L. Guo)
                     \it hdw55@tom.com(D.W.Huang).}
\date{}

\maketitle

\begin{abstract}
In this paper, we investigate the well-posedness theory for the
MHD boundary layer system in two-dimensional space.
The boundary layer equations are governed by the Prandtl type equations that
are derived from the full incompressible MHD system
with non-slip boundary condition on the velocity,
perfectly conducting condition on the magnetic field,
and Dirichlet boundary condition on the temperature
when the viscosity coefficient depends on the temperature.
To derive the Prandtl type boundary layer system, we require all the
hydrodynamic Reynolds numbers, magnetic Reynolds numbers
and Nusselt numbers tend to infinity at the same rate.
Under the assumption that the initial tangential magnetic field is not zero,
one applies the energy methods to establish the local-in-time existence
and uniqueness of solution for the MHD boundary
layer equations without the necessity of monotonicity condition.

\vspace*{5pt}

\vspace*{5pt}

\end{abstract}


\section{Introduction}

\quad The dynamics of an electrically conducting liquid near a wall
has been a topic of constant interest since the pioneering work
of Hartmann \cite{Hartmann}. An appropriate starting point to describe
such dynamics is the classical incompressible magnetohydrodynamics(MHD) system.
One important problem about MHD is to understand the high Reynolds and Nusselt
numbers limit in a domain with boundary.
In this paper, we investigate the following initial boundary value problem
for the two dimensional full incompressible  MHD system
in a periodic domain $\Omega=\left\{(x, y): x\in \mathbb{T}, y\in \mathbb{R}^+\right\}$:
\begin{equation}\label{eq1}
\left\{
\begin{aligned}
&\partial_t u^\varepsilon+(u^\varepsilon\cdot \nabla)u^\varepsilon
-\varepsilon {\rm div}(2\mu(\vartheta^\varepsilon)D(u^\varepsilon))
+\nabla p^\varepsilon=(H^\varepsilon \cdot \nabla)H^\varepsilon,\\
&c_\upsilon [\partial_t \vartheta^\varepsilon+(u^\varepsilon\cdot \nabla) \vartheta^\varepsilon]
-\varepsilon \kappa \Delta \vartheta^\varepsilon
=2 \varepsilon{\mu(\vartheta^\varepsilon)}|D(u^\varepsilon)|^2
+\varepsilon \nu  |\nabla \times H^\varepsilon|^2,\\
&\partial_t H^\varepsilon-\nabla \times (u^\varepsilon\times H^\varepsilon)
-\varepsilon \nu \Delta H^\varepsilon =0,\\
&{\rm div}u^\varepsilon =0,\quad {\rm div}H^\varepsilon=0,
\end{aligned}
\right.
\end{equation}
The unknown function $u^\varepsilon=(u_1^\varepsilon, u_2^\varepsilon)$
denotes the velocity vector, $H^\varepsilon=(h_1^\varepsilon, h_2^\varepsilon)$
denotes the magnetic field,
$\vartheta^\varepsilon$ denotes the absolute temperature,
and $p^\varepsilon=\widetilde{p}^\varepsilon+\frac{1}{2}|H^\varepsilon|^2$
represents the total pressure with $\widetilde{p}^\varepsilon$ the pressure of fluid.
Here, $\mu(\vartheta^\varepsilon)$ means that $\mu$ is a smooth function of temperature
$\vartheta^\varepsilon$, and $\varepsilon \mu(\vartheta^\varepsilon), \varepsilon \kappa$
and $\varepsilon \nu$ represent the viscosity, heat conductivity and resistivity coefficients
respectively.
To obtain the same boundary layer thickness,
we assume the viscosity, heat conductivity and resistivity coefficients have the
the same order of a small parameter $\varepsilon$.
The positive constant $c_v$ is the heat capacity coefficient,
and the deformation tensor $D(u^\varepsilon)$ is defined by
$$
D(u^\varepsilon)=\frac{1}{2}\left[\nabla u^\varepsilon+(\nabla u^\varepsilon)^{tr}\right].
$$
To complete the system \eqref{eq1}, the boundary conditions are given by
\begin{equation}\label{bc1}
\left.u^\varepsilon\right|_{y=0}=0,\quad
\left.\partial_y h_1^\varepsilon\right|_{y=0}=\left.h_2^\varepsilon\right|_{y=0}=0,\quad
\left. \vartheta^\varepsilon\right|_{y=0}=0.
\end{equation}
As the parameter $\varepsilon$ tends to zero in the systems \eqref{eq1},
we obtain the following systems formally
\begin{equation}\label{ideal}
\left\{
\begin{aligned}
&\partial_t u^0+(u^0\cdot \nabla)u^0
+\nabla P^0=(H^0 \cdot \nabla)H^0,\\
&\partial_t \vartheta^0+(u^0\cdot \nabla) \vartheta^0=0,\\
&\partial_t H^0-\nabla \times (u^0\times H^0)=0,\\
&{\rm div}u^0 =0,\quad {\rm div}H^0=0,
\end{aligned}
\right.
\end{equation}
which are the ideal MHD systems with energy equation.
Then, it is easy to check that there is a mismatch of boundary condition between
the equations \eqref{eq1} and \eqref{ideal} on the boundary $y=0$,
which will form the boundary layer as in the vanishing viscosity,
heat conductivity and resistivity limit process.
To find out the terms in \eqref{eq1} whose contributions is essential
for the boundary layer, we use the same scaling as the one used in \cite{Oleinik2},
$$
t=t,\quad x=x,\quad \widetilde{y}=\varepsilon^{-\frac{1}{2}}y,
$$
then  set
$$
\begin{aligned}
&u_1(t, x, \widetilde{y})=u_1^\varepsilon (t, x, y),\quad
u_2(t, x, \widetilde{y})=\varepsilon^{-\frac{1}{2}}u_2^\varepsilon (t, x, y),\\
&h_1(t, x, \widetilde{y})=h_1^\varepsilon (t, x, y),\quad
h_2(t, x, \widetilde{y})=\varepsilon^{-\frac{1}{2}}h_2^\varepsilon (t, x, y),
\end{aligned}
$$
and
$$
\theta(t, x, \widetilde{y})=\vartheta^\varepsilon(t, x, y),\quad
p(t, x, \widetilde{y})=p^\varepsilon(t, x, y).
$$
In this paper, we assume the viscosity function $\mu(\vartheta^\varepsilon)$
has the following form
\begin{equation}\label{assumption}
\mu(\vartheta^\varepsilon)\triangleq \mu \vartheta^\varepsilon+\mu.
\end{equation}
Then by taking the leading order, we deduce from the equations \eqref{eq1} that
\begin{equation}\label{eq2}
\left\{
\begin{aligned}
&\partial_t u_1+u_1\partial_x u_1+u_2\partial_y u_1
-\mu \partial_y[(\vartheta+1)\partial_y u_1]+\partial_x p
=h_1 \partial_x h_1+h_2 \partial_y h_1,\\
&\partial_y p=0,\\
&c_\upsilon(\partial_t \vartheta+u_1 \partial_x \vartheta+u_2 \partial_y \vartheta)
-\kappa \partial^2_y \vartheta
=\mu(\vartheta+1)(\partial_y u_1)^2+\nu(\partial_y h_1)^2,\\
&\partial_t h_1+\partial_y(u_2 h_1-u_1 h_2)
-\nu \partial^2_y h_1=0,\\
&\partial_t h_2-\partial_x(u_2 h_1-u_1 h_2)
-\nu \partial^2_y h_2=0,\\
&\partial_x u_1+\partial_y u_2=0,\quad
\partial_x h_1+\partial_y h_2=0,
\end{aligned}
\right.
\end{equation}
where $(t, x, y)\in [0, T]\times \Omega$, here we have replaced
$\widetilde{y}$ by $y$ for simplicity of notations.
Indeed, the nonlinear boundary layer systems \eqref{eq2}
become the classical well-known unsteady boundary layer systems
if the magnetic field vanishes, refer to \cite{Schlitchting}.

The second equation of equations \eqref{eq2}$_2$ implies that the leading order of
boundary layers for the total pressure $p^\varepsilon(t, x, y)$ is
invariant across the boundary layer, and should be matched to the outflow pressure
$P(t, x)$ on top of boundary layer, that is, the trace of pressure of idea MHD flow.
Hence, we obtain
$$
p(t, x, y)\equiv P(t, x).
$$
Furthermore, the tangential component $u_1(t, x, y)$ of velocity flied, $h_1(t, x, y)$
of magnetic field, temperature $\vartheta(t, x, y)$, should match the outflow
tangential velocity $U(t, x)$, outflow tangential magnetic field $H(t, x)$
and the outflow temperature $\Theta(t, x)$, on the top of boundary layer, that is
\begin{equation}\label{infinity}
u_1(t, x, y)\rightarrow  U(t, x), \quad
h_1(t, x, y) \rightarrow  H(t, x),\quad
\vartheta(t, x, y) \rightarrow \Theta(t, x),
~{\rm as}~y \rightarrow +\infty,
\end{equation}
where $U(x, t), H(x, t)$ and $\Theta(x, t)$ are the trace of tangential velocity,
tangential magnetic field and temperature respectively.
Then, we have the following matching conditions:
\begin{equation}\label{eqinf}
\partial_t U+U\partial_x U+\partial_x P=H\partial_x H,\quad
\partial_t \Theta+U\partial_x \Theta=0,\quad
\partial_t H+U\partial_x H-H\partial_x U=0.
\end{equation}
Moreover, by virtue of \eqref{bc1}, one attains  the following boundary conditions
\begin{equation}\label{bc2}
\left.u_1\right|_{y=0}=
\left.u_2\right|_{y=0}=
\left.\vartheta\right|_{y=0}=
\left.\partial_y h_1\right|_{y=0}=
\left.h_2\right|_{y=0}=0.
\end{equation}
On the other hand, it is noted that the equation \eqref{eq2}$_5$
is a direct consequences of equations \eqref{eq2}$_4$, \eqref{eq2}$_6$
and the boundary conditions \eqref{bc2}. Hence, we only need to study the
following initial boundary value problem for the nonlinear MHD boundary layer equations
\begin{equation}\label{eq3}
\left\{
\begin{aligned}
&\partial_t u_1+u_1\partial_x u_1+u_2\partial_y u_1
-\mu \partial_y[(\vartheta+1)\partial_y u_1]+ P_x
=h_1 \partial_x h_1+h_2 \partial_y h_1,\\
&c_\upsilon(\partial_t \vartheta+u_1 \partial_x \vartheta+u_2 \partial_y \vartheta)
-\kappa \partial^2_y \vartheta
=\mu(\vartheta+1)(\partial_y u_1)^2+\nu(\partial_y h_1)^2,\\
&\partial_t h_1+\partial_y(u_2 h_1-u_1 h_2)
-\nu \partial^2_y h_1=0,\\
&\partial_x u_1+\partial_y u_2=0,\quad
\partial_x h_1+\partial_y h_2=0,
\end{aligned}
\right.
\end{equation}
with the boundary conditions
\begin{equation}\label{bc3}
\left.(u_1, u_2, \vartheta, \partial_y h_1, h_2)(t, x, y)\right|_{y=0}=0,
\quad \underset{y\rightarrow +\infty}{\lim}
(u_1, \vartheta, h_1)(t, x, y)=(U, \Theta, H)(t,x).
\end{equation}
and the initial data
\begin{equation}\label{id}
\left.u_1(t, x, y)\right|_{t=0}=u_{10}(x,y),\quad
\left.\vartheta(t, x, y)\right|_{t=0}=\vartheta_0(x,y),\quad
\left.h_1(t, x, y)\right|_{t=0}=h_{10}(x,y).
\end{equation}

Let us first introduce some weighted Sobolev spaces for later use. Denote
$$
\Omega \triangleq \{(x, y): x\in \mathbb{T}, y \in \mathbb{R}^+\}.
$$
For any $l\in \mathbb{R}$, denote by $L^2_l(\Omega)$ the weighted Lebesgue space
with respect to the spatial variables:
$$
L_l^2(\Omega)\triangleq \{f(x, y):\Omega \rightarrow \mathbb{R},
\|f\|_{L_l^2(\Omega)}\triangleq (\int_\Omega {\langle y\rangle}^{2l}|f(x,y)|^2 dxdy)^{\frac{1}{2}}<+\infty\},\quad {\langle y\rangle} \triangleq 1+y,
$$
and then, for any given $m \in \mathbb{N}$, denote by $H_l^m(\Omega)$
the weighted Sobolev space:
$$
H^m_l(\Omega)\triangleq\{f(x,y):\Omega \rightarrow \mathbb{R},
\|f\|_{H_l^m(\Omega)}\triangleq (
\underset{m_1+m_2\le m}{\sum}\|{\langle y\rangle}^{l+m_2}\partial_x^{m_1}\partial_y^{m_2}f\|_{L^2(\Omega)}^2
)^{\frac{1}{2}}<+\infty\}.
$$

Now, we can state the main results with respect to the well-posedness theory
for the nonlinear MHD boundary layer sytems \eqref{eq3} in this paper as follows.
\begin{theorem}\label{theo1}
Let $m \ge 5$ be an integer, and $l \ge 0$ be a real number. Assume that the outer flow
$(U, \Theta, H, P_x)(t, x)$ satifies that for some $T>0$,
\begin{equation}\label{111}
M_0\triangleq\sum_{i=0}^{2m+2}
\underset{0\le t \le T}{\sup}\|\partial_t^i(U,\Theta,H,P)(t)\|_{H^{2m+2-i}(\mathbb{T}_x)}
<+\infty,
\end{equation}
and
$ \Theta(t, x)\ge 0$ for all $(t, x)\in [0, T]\times \mathbb{T}_x$.
Also, we suppose the initial data $(u_{10}, \vartheta_0, h_{10})(x, y)$ satisfies
\begin{equation}
\vartheta_0(x, y)\ge 0, \quad \left(u_{10}(x,y)-U(0,x), \vartheta_0(x,y)-\Theta(0,x),
h_{10}(x,y)-H(0,x)\right)\in H_{l}^{3m+2}(\Omega),
\end{equation}
and the compatibility conditions up to $m-$th order.
Moreover, there exists a sufficiently small constant $\delta_0>0$ such that
\begin{equation}
h_{10}(x,y)\ge 2\delta_0,\quad
|{\langle y\rangle}^{l+1}\partial_y^i (u_{10},\vartheta_0, h_{10})(x,y)|\le (2\delta_0)^{-1},
~{\rm for}~i=1,2,~(x, y)\in \Omega.
\end{equation}
Then, there exist a positive time $0<T^*\le T$ and a unique solution
$(u_1, u_2, \vartheta, h_1, h_2)$ to the initial boundary value problem
\eqref{eq3},  such that
\begin{equation}
(u_1-U, \vartheta-\Theta, h_1-H)\in \underset{i=0}{\overset{m}{\bigcap}}
W^{i,\infty}(0,T_*; H_l^{m-i}(\Omega)),
\end{equation}
and
\begin{equation}
\begin{aligned}
&(u_2+U_x y, h_2+H_x y)\in \underset{i=0}{\overset{m-1}{\bigcap}}
W^{i,\infty}(0,T_*; H_{-1}^{m-1-i}(\Omega)),\\
&(\partial_y u_2+U_x, \partial_y h_2+H_x)\in \underset{i=0}{\overset{m-1}{\bigcap}}
W^{i,\infty}(0,T_*; H_{l}^{m-1-i}(\Omega)).
\end{aligned}
\end{equation}
\end{theorem}

We now review some related works to the problem studied in this paper.
The vanishing viscosity limit of the incompressible Navier-Stokes equations
that, in a bounded domain with Dirichlet boundary condition,
is an important problem in both physics and mathematics.
This is due to the formation of a boundary layer, where the solution
undergoes a sharp transition from a solution of
the Euler system to the zero non-slip boundary condition on boundary of the
Navier-Stokes system.
This boundary layer satisfies the Prandtl system formally.
Indeed, Prandtl \cite{Prandtl} derived the Prandtl equations for boundary layers
from the incompressible Navier-Stokes equations with non-slip boundary condition.
The first systematic work in rigorous mathematics was obtained Oleinik \cite{Oleinik1},
in which she established the local in time well-posedness of the Prandtl equations
in dimension two by applying the Crocco transformation under the monotonicity condition on the tangential velocity field in the normal direction to the boundary.
For more extensional mathematical results, the interested readers can refer to
the classical book finished by Oleinik and Samokhin \cite{Oleinik2}.
By taking care of the cancelation in the convection term to overcome the
loss of derivative in the tangential direction of velocity, the researchers
in \cite{Xu-Yang-Xu} and \cite{Masmoudi} independently used the simply
energy method to establish well-posedness theory for the
two-dimensional Prandtl equations in the framework of Sobolev spaces.
Moreover, Xin and Zhang \cite{Xin-Zhang} built the global in time
weak solution by imposing an additional favorable condition on the pressure.
Furthermore, the well-posedness results for both classical and weak solutions
in dimension three were studied by Liu et al.\cite{{Liu-Wang-Yang1},{Liu-Wang-Yang2}}.
On the other hand, Sammartino and Caflisch \cite{{Sammartino-Caflisch1},{Sammartino-Caflisch2}}
obtained the well-posedness in the framework of analytic functions without the
monotonicity condition on the velocity field and justified the boundary layer expansion.
For more results to the Prandtl equations in the framework of analytic functions,
the interested readers can refer to
\cite{{Ignatova-Vicol},{Kukavica-Vicol},
{Lombardo-Cannone-Sammartino},{Maekawa},{Kukavica-Masmoudi-Vicol-Wong},{Zhang-Zhang}}
and the references therein.
And recently, the analyticity condition can be further relaxed to Gevrey regularity,
cf. \cite{{G-V-M-M},{G-V-M},{Li-Wu-Xu},{Li-Yang}}.

When the monotonicity condition is violated, separation of the boundary
layer is expected and observed for classical fluid.
Hence, E and Engquist \cite{E-Engquist} constructed a finite time blowup solution to
the Prandtl system for some special type of initial data.
Recently, G\'erard-Varet and Dormy  \cite{G-D-E} proved ill-posedness for
the linearized Prandtl equations around a nonmonotonic shear flow.
For more interesting ill-posedness(or instability) phenomena of solution to both
the linear and nonlinear Prandtl equations around the shear flow, the readers can
refer to \cite{{G-Nguyen},{Grenier},{Guo-Nguyen},{Liu-Wang-Yang3},{Liu-Yang},
{Grenier-Guo-Nguyen1},{Grenier-Guo-Nguyen2}, {Grenier-Guo-Nguyen3}}
and the references therein.
All these results show that the monotonicity assumption on the tangential
velocity is essential for the well-posedness except in the framework
of analytic functions or Gevrey functions.
On the other hand, as observed by Van Dommnelen and Shen \cite{Van Dommelen-Shen}
and studied mathematically by Hong and Hunter \cite{Hong-Hunter},
the monotonicity condition  is not needed for the well-posedness
of the inviscid Prandtl equations at least locally in time.
Recently, the well-posedness of thermal layer equations,
which was derived from the full compressible Navier-Stokes equations when the viscosity
coefficients vanish or are of higher order with respect to the heat conductivity
coefficient, were obtained by Liu et al.\cite{Liu-Wang-Yang4} without the monotonicity
condition on the velocity field in dimension three.

Under the influence of electro-magnetic field, the system of
magnetohydrodynamics(denoted by MHD) is a fundamental system
to describe the movement of electrically conducting fluid,
for example plasmas and liquid metals, refer to \cite{Alfven}.
For plasma, the boundary layer equations, which can be derived from the fundamental MHD
system, are more complicated than the classical Prandtl system because
of the coupling of the magnetic field with velocity field through the Maxwell equations.
If the magnetic field is transversal to the boundary, there are extensive discussions
on the so-called Hartmann boundary layer, refer to \cite{{Davidson},{Hartmann-Lazarus}}.
In addition, there are works on the stability of boundary layers with minimum Reynolds number
for flow with different structure to reveal the difference from the classical
boundary layers electro-magnetic field,  refer to \cite{{Arkhipov},{Drasin},{Rossow}}.
Under the non-slip boundary condition for the velocity, the well-posedness theory
for the boundary layer systems, which were derived if the hydrodynamic Reynolds numbers
tend to
infinity while the magnetic Reynolds numbers are fixed, was discussed in Oleinik and Samokhin \cite{Oleinik2}, for which the monotonicity condition on the velocity field is needed.
However, if both the hydrodynamic Reynolds numbers and magnetic Reynolds numbers
tend to infinity at the same rate, the local-in-time existence of solution
for the boundary layer system was
obtained by Liu et al.\cite{Liu-Xie-Yang} under the only condition on the initial
tangential magnetic field was not zero.
It should be pointed out that the well-posedness for this boundary layer system
does not need the monotonicity condition of tangential velocity.
At the same time, G\'{e}rard-Varet and Prestipino \cite{G-V-P}
provided a systematic derivation of boundary layer models in magnetohydrodynamics,
through an asymptotic analysis of the incompressible MHD system.
Furthermore, they also performed some stability analysis for the boundary layer
system, and emphasized the stabilizing effect of the magnetic field.

In this paper, we derive the boundary layer systems \eqref{eq3}
by requiring all the hydrodynamic Reynolds numbers,
magnetic Reynolds numbers and Nusselt numbers tend to infinity at the same rate.
On one hand, it is believed that the magnetic field has a stabilizing effect
on the boundary layer that could provide a mechanism for containment of
the high temperature gas in physics.
On the other hand, Liu et al.\cite{Liu-Xie-Yang} established the local
well-posedness theory for the MHD boundary layer systems(without energy equations)
under the only condition on the initial tangential magnetic field was not zero.
Hence, the prime objective of this paper is to prove the local existence and
uniqueness for the two dimensional MHD boundary layer systems with temperature field.
Now, let us explain the main difficulties arising from the appearance of temperature field
as well as the our strategies for overcoming them.
First of all, we should establish the lower bound estimate for the temperature field to
give $L^2(0,T; \mathcal{H}_l^m)-$norm for the quantity $\partial_y u$
because the viscosity coefficient depends on the temperature field.
Due to the lack of viscous term $\partial_x^2 \vartheta$, we can't apply
the minimum principle to attain the lower bound estimate for the temperature field.
Hence, we assume the viscosity function obeys the form(see \eqref{assumption}),
which will help us reach the target by means of energy method.
Secondly, the lack of high-order boundary conditions at $y=0$ prevent
us from applying the integration by parts in the $y-$variable, but
it will be solvable by taking the operator $\partial_t-\partial_y^2$
since the viscosity coefficient has a good form(see \eqref{assumption}).
Thirdly, some higher order nonlinear terms arising in the energy equation
will bring some difficulties when we apply the energy method to establish
local well-posedness theory for the MHD boundary layer systems.
However, we can choose the life span of solutions small suitably to
overcome these difficulties since we only investigate the local existence
of solutions in this paper. Finally, similar to the classical Prandtl equations,
the convective term $u_2\partial_y \vartheta$ in the energy equation \eqref{eq3}$_2$
will create a loss of $x-$derivative estimate.
Indeed, we can take the strategy of cancelation property and
create a quantity $\theta_\beta$(see \eqref{function3})
to avoid the $x-$derivative estimate of temperature field.

The rest of this paper is organized as follows.
Some preliminaries are given in Section \ref{sa}.
In Section \ref{sb}, one establishes the a priori energy estimates for the nonlinear
problem \eqref{eq3}.
The local-in-time existence and uniqueness of the solution to \eqref{eq3}
in Sobolev space are given in Section \ref{sc}.
Finally, some useful inequalities and important equivalent relations
will be stated in Appendixs \ref{appendixA} and \ref{appendixB}.

\section{Preliminaries}\label{sa}

\quad First of all, we introduce some notations which will be used
frequently in this paper. Denote the tangential derivative operator
$$
\partial_\tau^\beta =\partial_t^{\beta_1}\partial_x^{\beta_2},
\quad {\rm for}~\beta=(\beta_1, \beta_2)\in\mathbb{N}^2,~~
|\beta|=\beta_1+\beta_2,
$$
and then denote the derivative operator(in both time and space) by
$$
D^\alpha=\partial_\tau^\beta \partial_y^k,\quad {\rm for}~
\alpha=(\beta_1,\beta_2,k)\in \mathbb{N}^3,~~|\alpha|=|\beta|+k.
$$
Set $e_i \in \mathbb{N}^2, i=1,2$, and $E_j\in \mathbb{N}^3, j=1,2,3$, by
$$
e_1=(1,0),~e_2=(0,1),~E_1=(1,0,0),~E_2=(0,1,0),~E_3=(0,0,1),
$$
and denote by $\partial_y^{-1}$ the inverse of derivative $\partial_y$, i.e.,
$(\partial_y^{-1}f)(y)\triangleq \int_0^y f(z)dz$.
Furthermore, the notation $[\cdot, \cdot]$  denotes the commutator operator,
and $\mathcal{P}(\cdot)$ represents a nondecreasing polynomial function
that may differ from line to line.
For any integer $m$, define the function space $\mathcal{H}_l^m$ of
measurable functions $f(t, x, y): [0,T]\times \Omega \rightarrow \mathbb{R}$,
such that for any $t\in[0, T]$,
\begin{equation}\label{bn}
\|f(t)\|_{\mathcal{H}_l^m}\triangleq (\sum_{|\alpha|\le m}
\|{\langle y\rangle}^{k+l}D^\alpha f(t, \cdot)\|_{L^2(\Omega)}^2)^{\frac{1}{2}}<+\infty.
\end{equation}

Similar to Liu et al.\cite{Liu-Xie-Yang}, we introduce an auxiliary
function $\phi(y)$ satisfying that
\begin{equation}\label{function1}
\phi(y)=
\left\{
\begin{aligned}
&y, &\quad y \ge 2R_0,\\
&0, &\quad 0\le y \le R_0,
\end{aligned}
\right.
\end{equation}
which will help us overcome the technical difficulty originated
from the boundary terms at ${y=+\infty}$.
Then, set the new unknown functions:
\begin{equation}\label{def1}
\begin{aligned}
&u(t, x, y)\triangleq u_1(t, x, y)-U(t,x)\phi'(y),\quad
v(t, x, y)\triangleq u_2(t, x, y)+U_x(t,x)\phi(y),\\
&h(t, x, y)\triangleq h_1(t, x, y)-H(t,x)\phi'(y),\quad
g(t, x, y)\triangleq h_2(t, x, y)+H_x(t,x)\phi(y),\\
&\theta(t, x, y)\triangleq \vartheta(t, x, y)-\Theta(t,x)\phi'(y).
\end{aligned}
\end{equation}
Choose the above construction for $(u, v,\theta, h, g)$ to ensure
the divergence free conditions:
$$
\partial_x u+\partial_y v=0,\quad \partial_x h+\partial_y g=0,
$$
and homogeneous boundary conditions:
$$
\left.(u, v, \theta, \partial_y h, g)\right|_{y=0}=0,\quad
\underset{y\rightarrow \infty}{\lim}(u, \theta, h)(t,x,y)=0.
$$
Then, it is easy to check that
\begin{equation}\label{inc}
v=-\partial_y^{-1}\partial_x u,\quad g=-\partial_y^{-1}\partial_x h.
\end{equation}
At the same time, one can deduce from the relation \eqref{def1} that
$$
u=(u_1-U)+U(1-\phi'(y)),\quad
\theta=(\vartheta-\Theta)+\Theta(1-\phi'(y)),\quad
h=(h_1-H)+H(1-\phi'(y)),
$$
which, together with the construction of $\phi(y)$(see the definition
in \eqref{function1}), yields immediately
\begin{equation}
\|(u, \theta, h)(t)\|_{\mathcal{H}_l^m}-CM_0
\le \|(u_1-U, \vartheta-\Theta, h_1-H)(t)\|_{\mathcal{H}_l^m}
\le \|(u, \theta, h)(t)\|_{\mathcal{H}_l^m}+CM_0,
\end{equation}
where the quantity $M_0$ is defined in \eqref{111}.
By using the new unknown function $(u, v, \theta, h, g)$,
we can reformulate the original problem \eqref{eq3} as the following form
\begin{equation}\label{eq4}
\left\{
\begin{aligned}
&\partial_t u+[(u+U\phi')\partial_x+(v-U_x \phi)\partial_y]u
-[(h+H\phi')\partial_x+(g-H_x \phi)\partial_y]h\\
&~~-\mu{\partial_y}[(\theta+\Theta \phi'(y)+1){\partial_y}u]
+U_x \phi' u+U \phi'' v-H_x \phi' h-H \phi'' g-U\phi^{(3)}\theta-U\phi''\theta_y
=r_1,\\
&c_v\{\partial_t \theta\!+\!(u+U\phi')\partial_x \theta\!+\!(v-U_x \phi)\partial_y \theta\}
\!-\!\kappa \partial_y^2 \theta\!+\!c_v \Theta_x \phi' u\!+\!c_v \Theta \phi'' v
\!-\!\mu \theta (u_y)^2\!-\!\mu (U\phi'')^2 \theta\\
&\quad \!-2\mu U \phi'' \theta u_y\!-\!\mu \Theta \phi'(u_y)^2
\!-\!2\mu \Theta U \phi' \phi'' u_y\!-\!2\mu U\phi'' u_y\!-\!\mu (u_y)^2
\!-\!\nu (h_y)^2\!-\!2\nu H \phi' h_y=r_2,\\
&\partial_t h+[(u+U\phi')\partial_x+(v-U_x \phi)\partial_y]h
-[(h+H\phi')\partial_x+(g-H_x \phi)\partial_y]u
-\nu \partial_y^2 h\\
&\quad +H_x \phi' u+H \phi'' v- U_x \phi' h-U \phi'' g=r_3,\\
&\partial_x u+\partial_y v=0,\quad \partial_x h+\partial_y g=0,\\
\end{aligned}
\right.
\end{equation}
with the boundary and initial conditions
\begin{equation}\label{bc4}
\left.(u, v, \theta, \partial_y h, g)(t, x, y)\right|_{y=0}=0,
~\underset{y\rightarrow +\infty}{\lim}(u, \theta, h)(t, x, y)=0,
~\left.(u,\theta, h)(t, x, y)\right|_{t=0} = (u_0, \theta_0, h_0)(x,y),
\end{equation}
where
\begin{equation}\label{function2}
\left\{
\begin{aligned}
&r_1\triangleq U_t[(\phi')^2-\phi'-\phi \phi'']
+P_x[(\phi')^2-\phi \phi''-1]
+U\Theta[\phi'\phi^{(3)}+(\phi'')^2]+U\phi'\phi^{(3)},\\
&r_2\triangleq c_v \Theta_t [(\phi')^2\!-\!\phi']\!+\!c_v U_x \Theta \phi \phi''
\!+\!\kappa \Theta \phi^{(3)}\!-\!\mu \Theta \phi'(U \phi'')^2
\!+\!\mu(U \phi'')^2\!+\!\kappa(\Theta \phi'')^2\!+\!\nu (H \phi'')^2,\\
&r_3\triangleq H_t [(\phi')^2-\phi'+\phi \phi'']+\nu H \phi^{(3)},\quad
u_0 \triangleq u_{10}(x,y)-U(0,x)\phi'(y),\\
&\theta_0(x,y)\triangleq \vartheta_0(x,y)-\Theta(0, x)\phi'(y),
\quad  h_0 (x,y)\triangleq h_{10}(x,y)-H(0,x)\phi'(y).
\end{aligned}
\right.
\end{equation}
In view of the definition $\phi(y)$, it is easy to check that
\begin{equation}\label{function2}
\begin{aligned}
&r_1(t,x,y),~~r_2(t,x,y),~~r_3(t,x,y)=0, \quad y \ge 2R_0,\\
&r_1(t,x,y)=-P_x(t,x),~~r_2(t,x,y)=r_3(t,x,y)=0, \quad 0\le y \le R_0,\\
\end{aligned}
\end{equation}
and for any $t\in [0, T], \lambda \ge 0$ and $|\alpha|\le m$,
then one gets that
\begin{equation}\label{estimate1}
\|\langle y\rangle^\lambda D^\alpha (r_1, r_2, r_3)(t)\|_{L^2(\Omega)}
\le C \sum_{|\beta|\le |\alpha|+1}
\|\partial_\tau^\beta(U, \Theta, H, P_x)\|_{L^2(\mathbb{T}_x)}^2
\le CM_0^2.
\end{equation}
Furthermore, we have the following relation for the initial data
\begin{equation}\label{estimate2}
\|(u_0, \theta_0, h_0)\|_{\mathcal{H}_l^m}-CM_0
\le \|(u_{10}-U(0,x), \vartheta_0-\Theta(0,x), h_{10}-H(0,x))(t)\|_{\mathcal{H}_l^m}
\le \|(u_0, \theta_0, h_0)\|_{\mathcal{H}_l^m}+CM_0.
\end{equation}

Finally, from the transform \eqref{def1} and the relation \eqref{inc},
it is easy to know that the well-posedness theory in
Theorem \ref{theo1} is a corollary of the following result.

\begin{theorem}\label{theo2}
Let $m \ge 5$ be an integer, $l \ge 0$ be a real number,
and $(U, \Theta, H, P_x)(t,x)$ satisfies the hypotheses given in Theorem \ref{theo1}.
In addition, assume that for the problem \eqref{eq4}, the initial data satisfies
$\theta_0(x,y)+\Theta(0, x)\phi'(y) \ge 0, ~
(u_0(x, y), \theta_0(x, y), h_0(x,y))\in H_l^{3m+2}(\Omega)$,
and the compatibility condition up to $m-$th order.
Moreover, there exists a sufficiently small constant $\delta_0>0$, such that
\begin{equation}
~h_0(x, y)+H(0, x)\phi'(y)\ge 2 \delta_0,~
|\langle y \rangle^{l+1}\partial_y^i(u_0, \theta_0, h_0)(x, y)|\le (2\delta_0)^{-1},
~{\rm for}~i=1,2, (x, y)\in \Omega.
\end{equation}
Then, there exists a time $0<T_* \le T$ and a unique solution
$(u, v, \theta, h, g)(t,x,y)$ to the initial boundary value problem
\eqref{eq4}, such that
\begin{equation}\label{211}
(u, \theta, h)\in \underset{i=0}{\overset{m}{\bigcap}}
W^{i,\infty}(0,T_*; H_l^{m-i}(\Omega)),
\end{equation}
and
\begin{equation}\label{212}
(v, g)\in \underset{i=0}{\overset{m-1}{\bigcap}}
W^{i,\infty}(0,T_*; H_{-1}^{m-1-i}(\Omega)),\quad
(\partial_y v, \partial_y g)\in \underset{i=0}{\overset{m-1}{\bigcap}}
W^{i,\infty}(0,T_*; H_l^{m-1-i}(\Omega)).
\end{equation}
\end{theorem}

Therefore, our main task is to show the local-posedness theory
in the above Theorem  \ref{theo2}, and its proof will be given
in the following two sections.

\section{A priori estimates}\label{sb}
\quad
In this section, we will establish a prior estimates for the
nonlinear MHD boundary layer problem \eqref{eq4}.

\begin{proposition}\label{pro1}[Weighted estimates for $D^m(u, \theta, h)$]\\
Let $m \ge 5$ be an integer, $l \ge 0$ be a real number, and the hypotheses
for $(U, \Theta, H, P_x)(t,x)$ given in Theorem \ref{theo1} hold. Assume
that $(u, v, \theta, h, g)$ is a classical solution to the problem
\eqref{eq4} in $[0 ,T]$, satisfying that
$(u, \theta, h)\in L^\infty(0, T; \mathcal{H}_l^m),
(\partial_y u, \partial_y \theta, \partial_y h)\in L^2(0, T; \mathcal{H}_l^m)$,
and for sufficiently small $\delta_0$:
\begin{equation}\label{p1}
h(t, x, y)+H(t, x)\phi'(y)\ge \delta_0,~
\langle y\rangle^{l+1} \partial_y^i (u, \theta, h)(t, x, y)\le \delta_0^{-1},~
i=1,2, ~(t, x, y)\in [0, T]\times \Omega.
\end{equation}
Then, it holds that for small time,
\begin{equation}\label{p11}
\begin{aligned}
\underset{0\le s \le t}{\sup}\|(u, \theta, h)(s)\|_{\mathcal{H}^m_l}^2
&\le
\delta_0^{-4}
\{\mathcal{P}(M_0+\|(u_0, \theta_0, h_0)\|_{H^{2m}_l(\Omega)})
+C M_0^{10}t\}^{\frac{1}{2}}\\
&\quad \cdot
\left\{1-C\delta_0^{-48}t
(\mathcal{P}(M_0+\|(u_0, \theta_0, h_0)\|_{H^{2m}_l(\Omega)})
+C M_0^{10}t)^5\right\}^{-\frac{1}{10}}.
\end{aligned}
\end{equation}
Also, we have that for $i=1,2$,
\begin{equation}\label{p12}
\begin{aligned}
&\|\langle y\rangle^{l+1}\partial_y^i (u, \theta, h)(t, x, y)\|_{L^\infty(\Omega)}\\
&\le \|\langle y\rangle^{l+1}\partial_y^i (u_0, \theta_0, h_0)(x, y)\|_{L^\infty(\Omega)}\\
&\quad +C \delta_0^{-4}t
\left\{\mathcal{P}(M_0+\|(u_0, \theta_0, h_0)\|_{H^{2m}_l(\Omega)})
+C M_0^{10}t\right\}^{\frac{1}{2}}\\
&\quad \cdot
\left\{1-C\delta_0^{-48}t
(\mathcal{P}(M_0+\|(u_0, \theta_0, h_0)\|_{H^{2m}_l(\Omega)})
+C M_0^{10}t)^5\right\}^{-\frac{1}{10}}.
\end{aligned}
\end{equation}
and
\begin{equation}\label{p13}
\begin{aligned}
h(t, x, y)
&\ge h_0(x, y)
-C \delta_0^{-4}t
\left\{\mathcal{P}(M_0+\|(u_0, \theta_0, h_0)\|_{H^{2m}_l(\Omega)})
+C M_0^{10}t\right\}^{\frac{1}{2}}\\
&\quad \cdot
\left\{1-C\delta_0^{-48}t
(\mathcal{P}(M_0+\|(u_0, \theta_0, h_0)\|_{H^{2m}_l(\Omega)})
+C M_0^{10}t)^5\right\}^{-\frac{1}{10}}.
\end{aligned}
\end{equation}
Here $\mathcal{P}(\cdot)$ denotes a nondecreasing polynomial function.
\end{proposition}

\subsection{Lower bound estimate for temperature}

\quad In this subsection, we obtain lower bound estimate for the temperature field,
which will play an important role in giving
$L^2(0,T; \mathcal{H}_l^m)-$norm for the quantity $\partial_y u$.

\begin{lemma}\label{lemma3.2}
Let $m \ge 5$ be an integer, $l \ge 0$ be a real number, and the hypotheses for
$(U, \Theta, H, P_x)(t, x)$ given in Theorem \ref{theo1} hold on. Assume that
$(u, v, \theta, g, h)$ is a classical solution to the nonlinear problem
\eqref{eq4} in $[0, T]$,
and satisfies $(u, \theta, h)\in L^\infty(0, T; \mathcal{H}^m_l),
(\partial_y u, \partial_y \theta, \partial_y h)\in L^2(0, T; \mathcal{H}_l^m)$.
Then, it holds that
\begin{equation}\label{321}
\theta(t, x, y)+\Theta(t, x)\phi'(y) \ge 0,
\end{equation}
for almost everywhere $(t, x, y)\in [0, T]\times \mathbb{T}_x \times \mathbb{R}^+$.
\end{lemma}
\begin{proof}
By virtue of $\underset{y\rightarrow +\infty}{\lim}\vartheta(t, x, y)=\Theta(t, x)$,
for all $(x, t) \in \mathbb{T}_x \times [0, T]$.
Then there exists a positive constant $R_1$ such that for $y \ge R_1$, we have
$$
\vartheta(t, x, y)-\Theta(t, x)\ge -\frac{1}{2}\underset{(t,x)}{\inf}\Theta(t,x),
$$
which implies
\begin{equation}\label{322}
\vartheta(t, x, y)\ge \Theta(t, x)-\frac{1}{2}\underset{(t,x)}{\inf}\Theta(t,x)\ge
\frac{1}{2}\underset{(t,x)}{\inf}\Theta(t,x),
\end{equation}
for all $(t, x, y)\in [0, T]\times \mathbb{T}_x \times [R_1, +\infty)$.
Let
$l_1\triangleq \underset{(x,y)\in \mathbb{T}_x \times [0, R_1+1]}{\min}\vartheta_0(x, y)
\ge 0,
l_2\triangleq \frac{1}{2}\underset{(t,x)}{\inf}\Theta(t,x)\ge 0$,
$l\triangleq \min\{l_1, l_2, 0\}$,
and  $\Omega_0 \triangleq \mathbb{T}_x \times [0, R_1+1]$.
For any $0<\varepsilon_0<1$, let $k\triangleq l-\varepsilon_0$,
multiplying \eqref{eq3}$_2$ by $(k-\vartheta)_+$ and
integrating the resulting equality over
$[0, t]\times \Omega_{0}$, we find
\begin{equation}\label{323}
\begin{aligned}
&\frac{c_v}{2} \int_0^t \frac{d}{d\tau}\int_{\Omega_{0}} |(k-\vartheta)_+|^2 dxdy d\tau
+\kappa \int_0^t \int_{\Omega_{0}} |\partial_y(k-\vartheta)_+|^2 dxdy d\tau\\
&=\mu \int_0^t \int_{\Omega_0} (\partial_y u_1)^2 |(k-\vartheta)_+|^2 dxdy d\tau
-\mu(k+1) \int_0^t \int_{\Omega_0} (\partial_y u_1)^2 (k-\vartheta)_+ dxdy d\tau\\
&\quad -\nu \int_0^t \int_{\Omega_0} (\partial_y h_1)^2 (k-\vartheta)_+ dxdy d\tau,
\end{aligned}
\end{equation}
here $f_+\triangleq \max\{f, 0\}\ge 0$.
It is easy to deduce from the \eqref{323} that
\begin{equation}\label{324}
\int_{\Omega_0} |(k-\vartheta)_+(t)|^2dxdy
\le 2\mu c_v^{-1}\int_0^t \|\partial_y u_1\|^2_{L^\infty(\Omega)}
\int_{\Omega_0} |(k-\vartheta)_+|^2 dxdy ~d\tau.
\end{equation}
Then, the application of the Gr\"{o}nwall inequality to \eqref{324} yields immediately
\begin{equation*}
(k-\vartheta)_+(x, y, t)=0, \quad {\rm a.e.}~ (t, x, y)\in
[0, T] \times \Omega_0,
\end{equation*}
which, implies that
\begin{equation}\label{325}
\vartheta(t, x, y)\ge k =l-\varepsilon_0,\quad {\rm a.e.}~ (t, x, y)\in
[0, T] \times \Omega_0.
\end{equation}
Let $\varepsilon_0 \rightarrow 0^+$ in \eqref{325}, then
it is easy to deduce that
\begin{equation*}
\vartheta(t, x, y)\ge l,\quad {\rm a.e.}~ (t, x, y)\in
[0, T] \times \Omega_0,
\end{equation*}
which, together with \eqref{322}, yields
$\vartheta(t, x, y) \ge 0, {\rm for ~all~} (t, x, y)\in [0, T]\times \mathbb{T}_x \times \mathbb{R}^+$. Then, the construction of function \eqref{def1} helps us  complete
the proof of Lemma \ref{lemma3.2}.
\end{proof}

\begin{remark}
The equation \eqref{eq3}$_2$ is not a standard parabolic type equation
due to the lack of viscous term $\partial_x^2 \vartheta$.
Then, we can't apply the minimum principle to obtain the estimate \eqref{321}
for  temperature field.
In order to reach the target \eqref{321}, we  assume the viscosity
function $\mu(\vartheta^\varepsilon)$ obeys the form \eqref{assumption}.
\end{remark}

\subsection{Weighted $H_l^m-$ estimates with norm derivatives}
\quad
For any $|\alpha|=|\beta|+k\le m$ and $|\beta|\le m-1$,
the weighted estimates on $D^\alpha(u, \theta, h)$ can be obtained by the
standard energy method since one order regularity loss is allowed.
Then, we can obtain the following estimates:

\begin{lemma}\label{lemma3.3}
[Weighted estimates for $D^\alpha (u, \theta, h)$ with $|\alpha|\le m, |\beta|\le m-1$]\\
Let $m \ge 5$ be an integer, $l \ge 0$ be a real number, and the hypotheses for
$(U, \Theta, H, P_x)(t, x)$ given in Theorem \ref{theo1} hold. Assume that
$(u, v, \theta, g, h)(t, x, y)$ is a classical solution to the problem \eqref{eq4} in $[0, T]$,
and satisfies $(u, \theta, h)(t, x, y) \in L^\infty(0, T; \mathcal{H}^m_l),
(\partial_y u, \partial_y \theta, \partial_y h)(t, x, y)\in L^2(0, T; \mathcal{H}_l^m)$.
Then, there exists a positive constant $C$, depending on $m, l$ and $\phi$ such that
for any small $0<\delta_1 <1$
\begin{equation}\label{331}
\begin{aligned}
&\sum_{\tiny\substack{|\alpha|\le m \\ |\beta|\le m-1}}
\left(\frac{d}{dt}\|D^\alpha(u, \sqrt{c_v}\theta, h)(t)\|_{L^2_{k+l}(\Omega)}^2
+c_0\|(D^\alpha \partial_y u, D^\alpha \partial_y \theta,
D^\alpha \partial_y h)(t)\|_{L^2_{k+l}(\Omega)}^2\right)\\
&\le \delta_1\|(\partial_y u, \partial_y \theta, \partial_y h)(t)\|_{\mathcal{H}_0^m}^2
+C\delta_1^{-1}(\|(u, \theta, h)(t)\|_{\mathcal{H}^m_l}^8+1)\\
&\quad +\sum_{\tiny\substack{|\alpha|\le m \\ |\beta|\le m-1}}
\|D^\alpha(r_1, r_2, r_3)(t)\|_{L^2_{l+k}(\Omega)}^2
+C\sum_{|\beta|\le m+2}\|\partial_\tau^\beta(U, \Theta, H, P)(t)\|_{L^2(\mathbb{T}_x)}^8,
\end{aligned}
\end{equation}
where $c_0\triangleq \min \{\mu, \kappa, \nu\}.$
\end{lemma}
\begin{proof}
Applying the operator $D^\alpha=\partial_\tau^\beta \partial_y^k$
for $\alpha=(\beta, k)=(\beta_1, \beta_2, k)$, satisfying
$|\alpha|=|\beta|+k\le m, |\beta|\le m-1$, to the equations
\eqref{eq4}$_1$, \eqref{eq4}$_2$ and \eqref{eq4}$_3$ respectively,
it yields that
\begin{equation}\label{332}
\left\{
\begin{aligned}
&\partial_t D^\alpha u=D^\alpha r_1
+\mu \partial_y D^\alpha[(\theta+\Theta \phi'+1)\partial_y u]-
D^\alpha I_1,\\
&c_v \partial_t D^\alpha \theta=D^\alpha r_2
+\kappa \partial_y^2 D^\alpha \theta-D^\alpha I_2,\\
&\partial_t D^\alpha h=D^\alpha r_3+\nu \partial_y^2 D^\alpha h-D^\alpha I_3,
\end{aligned}
\right.
\end{equation}
where the functions $I_i(i=1,2,3)$ are defined as follows:
\begin{equation*}\label{333}
\left\{
\begin{aligned}
I_1&\triangleq[(u+U\phi')\partial_x+(v-U_x \phi)\partial_y]u
-[(h+H\phi')\partial_x+(g-H_x \phi)\partial_y]h\\
&\quad+U_x \phi' u+U \phi'' v-H_x \phi' h-H \phi'' g-U\phi^{(3)}\theta-U\phi''\theta_y;\\
I_2&\triangleq c_v[(u+U\phi')\partial_x \!+\!(v-U_x \phi)\partial_y ]\theta
+\!c_v \Theta_x \phi' u\!+\!c_v \Theta \phi'' v
\!-\!\mu \theta (u_y)^2\!-\!\mu (U\phi'')^2 \theta\\
&\quad \!-2\mu U \phi'' \theta u_y\!-\!\mu \Theta \phi'(u_y)^2
\!-\!2\mu \Theta U \phi' \phi'' u_y\!-\!2\mu U\phi'' u_y\!-\!\mu (u_y)^2
\!-\!\nu (h_y)^2\!-\!2\nu H \phi' h_y;\\
I_3&\triangleq [(u+U\phi')\partial_x+(v-U_x \phi)\partial_y]h
-[(h+H\phi')\partial_x+(g-H_x \phi)\partial_y]u+H_x \phi' u\\
&\quad +H \phi'' v- U_x \phi' h-U \phi'' g.\\
\end{aligned}
\right.
\end{equation*}
Multiplying \eqref{332}$_1$ by $\langle y \rangle^{2l+2k}D^\alpha u$,
\eqref{332}$_2$ by $\langle y \rangle^{2l+2k}D^\alpha \theta$,
and  \eqref{332}$_3$ by $\langle y \rangle^{2l+2k}D^\alpha h$ respectively,
and integrating them over $\Omega$ with respect to the spatial variables $x$
and $y$, we find
\begin{equation}\label{334}
\begin{aligned}
&\frac{1}{2}\frac{d}{dt}\int_\Omega \langle y \rangle^{2k+2l} |D^\alpha u|^2 dxdy
+\frac{1}{2}\frac{d}{dt}\int_\Omega \langle y \rangle^{2k+2l} |D^\alpha h|^2 dxdy
+\frac{c_v}{2}\frac{d}{dt}\int_\Omega \langle y \rangle^{2k+2l} |D^\alpha \theta|^2 dxdy\\
&=\int_\Omega(D^\alpha r_1 \cdot\langle y \rangle^{2k+2l}D^\alpha u
+D^\alpha r_2 \cdot\langle y \rangle^{2k+2l}D^\alpha \theta
+D^\alpha r_3 \cdot\langle y \rangle^{2k+2l}D^\alpha h)dxdy\\
&\quad -\int_\Omega(D^\alpha I_1 \cdot\langle y \rangle^{2k+2l}D^\alpha u
+D^\alpha I_2 \cdot\langle y \rangle^{2k+2l}D^\alpha \theta
+D^\alpha I_3 \cdot\langle y \rangle^{2k+2l}D^\alpha h)dxdy\\
&\quad +\mu\int_\Omega \partial_y D^\alpha[(\theta+\Theta \phi'+1)\partial_y u] \cdot\langle y \rangle^{2k+2l}D^\alpha udxdy
+\kappa\int_\Omega \partial_y^2 D^\alpha \theta \cdot\langle y \rangle^{2k+2l}D^\alpha \theta dxdy\\
&\quad +\nu \int_\Omega \partial_y^2 D^\alpha h \cdot\langle y \rangle^{2k+2l}D^\alpha h dxdy.\\
\end{aligned}
\end{equation}
First of all, the application of Cauchy inequality implies immediately
\begin{equation}\label{335}
\begin{aligned}
&\int_\Omega(D^\alpha r_1 \cdot\langle y \rangle^{2k+2l}D^\alpha u
+D^\alpha r_2 \cdot\langle y \rangle^{2k+2l}D^\alpha \theta
+D^\alpha r_3 \cdot\langle y \rangle^{2k+2l}D^\alpha h)dxdy\\
&\le \frac{1}{2}\|D^\alpha(u, \theta, h)\|_{L^2_{l+k}(\Omega)}^2
+\frac{1}{2}\|D^\alpha(r_1, r_2, r_3)\|_{L^2_{l+k}(\Omega)}^2.
\end{aligned}
\end{equation}
Next, we assume the following two estimates hold, which will be proved later:
for any small $0<\delta_1<1$,
\begin{equation}\label{336}
\begin{aligned}
&\mu\int_\Omega \partial_y D^\alpha[(\theta+\Theta \phi'+1)\partial_y u] \cdot\langle y \rangle^{2k+2l}D^\alpha u~dxdy\\
&\quad+\kappa\int_\Omega \partial_y^2 D^\alpha \theta \cdot\langle y \rangle^{2k+2l}D^\alpha \theta ~dxdy
+\nu \int_\Omega \partial_y^2 D^\alpha h \cdot\langle y \rangle^{2k+2l}D^\alpha h ~dxdy\\
&\le -\frac{\mu}{2}\|D^\alpha \partial_y u\|_{L^2_{k+l}(\Omega)}^2
-\frac{\kappa}{2}\|D^\alpha \partial_y \theta\|_{L^2_{k+l}(\Omega)}^2
-\frac{\nu}{2}\|D^\alpha \partial_y h\|_{L^2_{k+l}(\Omega)}^2
+\delta_1 \mu\|\partial_y u\|_{\mathcal{H}^m_0}^2\\
&\quad
+\delta_1 \kappa\|\partial_y \theta\|_{\mathcal{H}^m_0}^2
+\delta_1 \nu\|\partial_y h\|_{\mathcal{H}^m_0}^2
+C\delta_1^{-1}(\|(u, \theta, h)\|_{\mathcal{H}^{m}_0}^6+1)
+\sum_{|\beta|\le m+1}\| \partial_\tau^{{\beta}}(\Theta, P)\|_{L^2(\mathbb{T}_x)}^4,
\end{aligned}
\end{equation}
and
\begin{equation}\label{337}
\begin{aligned}
&-\int_\Omega(D^\alpha I_1 \cdot\langle y \rangle^{2k+2l}D^\alpha u
+D^\alpha I_2 \cdot\langle y \rangle^{2k+2l}D^\alpha \theta
+D^\alpha I_3 \cdot\langle y \rangle^{2k+2l}D^\alpha h)dxdy\\
&\le \frac{\mu}{4}\|D^\alpha \partial_y u\|_{L^2_{l+k}(\Omega)}
+\frac{\nu}{4}\|D^\alpha \partial_y h\|_{L^2_{l+k}(\Omega)}
+\frac{\kappa}{4}\|D^\alpha \partial_y \theta\|_{L^2_{k+l}(\Omega)}^2
+\delta_1 \mu\|\partial_y u\|_{\mathcal{H}^m_0}^2\\
&\quad +\delta_1 \nu\|\partial_y h\|_{\mathcal{H}^m_0}^2
+\delta_1^{-1}(\|(u, \theta, h)\|_{\mathcal{H}^m_l}^8+1)
+\sum_{|\beta| \le m+2}\|\partial_\tau^\beta(U, \Theta, H)\|_{L^2(\mathbb{T}_x)}^8.
\end{aligned}
\end{equation}
By plugging the estimates \eqref{335}-\eqref{337} into \eqref{334}, one obtains
\begin{equation}\label{338}
\begin{aligned}
&\frac{d}{dt}\|D^\alpha(u, \sqrt{c_v}\theta, h)(t)\|_{L^2_{k+l}(\Omega)}^2
+\frac{1}{4}\min\{\mu, \kappa, \nu\}
\|D^\alpha \partial_y (u, \theta, h)(t)\|_{L^2_{k+l}(\Omega)}^2\\
&\le \delta_1\max\{\mu, \nu, \kappa\}\|(\partial_y u, \partial_y \theta, \partial_y h)(t)\|_{\mathcal{H}_0^m}^2
+C\delta_1^{-1}(\|(u, \theta, h)(t)\|_{\mathcal{H}^m_l}^8+1)\\
&\quad
+\sum_{\tiny\substack{|\alpha|\le m \\ |\beta|\le m-1}}
\|D^\alpha(r_1, r_2, r_3)(t)\|_{L^2_{l+k}(\Omega)}^2
+\sum_{|\beta|\le m+2}\|\partial_\tau^\beta(U, \Theta, H, P)(t)\|_{L^2(\mathbb{T}_x)}^8.
\end{aligned}
\end{equation}
which implies the estimate \eqref{331} immediately.
Therefore, we complete the proof of Lemma \ref{lemma3.3}.
\end{proof}

\begin{proof}[\textbf{Proof of \eqref{336}.}]
In this part, we will first handle the term $\mu\int_\Omega \partial_y D^\alpha[(\theta+\Theta \phi'+1)\partial_y u] \cdot\langle y \rangle^{2k+2l}D^\alpha udxdy$.
By integration by parts, we have
\begin{equation}\label{c11}
\begin{aligned}
&\mu\int_\Omega \partial_y D^\alpha[(\theta+\Theta \phi'+1)\partial_y u] \cdot\langle y \rangle^{2k+2l}D^\alpha udxdy\\
&=-\mu \int_{\mathbb{T}_x}  D^\alpha[(\theta+\Theta \phi'+1)\partial_y u] \cdot
D^\alpha u|_{y=0}dx\\
&\quad-2(k+l)\mu\int_\Omega  D^\alpha[(\theta+\Theta \phi'+1)\partial_y u] \cdot\langle y \rangle^{2k+2l-1}D^\alpha udxdy\\
&\quad -\mu\int_\Omega  D^\alpha[(\theta+\Theta \phi'+1)\partial_y u] \cdot\langle y \rangle^{2k+2l}\partial_y D^\alpha udxdy\\
&=-\mu\int_\Omega  (\theta+\Theta \phi'+1)
\langle y \rangle^{2k+2l}|\partial_y D^\alpha u|^2dxdy\\
&\quad -\mu \int_\Omega [D^\alpha, (\theta+\Theta \phi'+1)]\partial_y u
\cdot \langle y\rangle^{2k+2l}D^\alpha \partial_y u dxdy\\
&\quad-2(k+l)\mu\int_\Omega  D^\alpha[(\theta+\Theta \phi'+1)\partial_y u] \cdot\langle y \rangle^{2k+2l-1}D^\alpha udxdy\\
&\quad-\mu\int_{\mathbb{T}_x} \left. D^\alpha[(\theta+1)\partial_y u]
\cdot D^\alpha u\right|_{y=0}dx.
\end{aligned}
\end{equation}

It is complicated to deal with four terms on the righthand side of \eqref{c11},
then we estimate them by the following four steps.\\
\textbf{\underline{Step 1:}}
In view of the lower boundedness of $\vartheta$ in \eqref{321}, it is easy to deduce that
$$
\theta+\Theta \phi'+1\ge 1,
$$
which, implies directly that
\begin{equation}\label{c12}
-\mu\int_\Omega  (\theta+\Theta \phi'+1)
\langle y \rangle^{2k+2l}|\partial_y D^\alpha u|^2dxdy
\le -\mu \int_\Omega  \langle y \rangle^{2k+2l}|\partial_y D^\alpha u|^2dxdy.
\end{equation}
\textbf{\underline{Step 2:}}
By virtue of the H\"{o}lder inequality, one deduces that
\begin{equation}\label{c13}
\begin{aligned}
&\left|\mu \int_\Omega [D^\alpha, (\theta+\Theta \phi'+1)]\partial_y u
\cdot \langle y\rangle^{2k+2l}D^\alpha \partial_y u dxdy\right|\\
&\le \mu \sum_{0<\widetilde{\alpha}\le \alpha}
\|D^{\widetilde{\alpha}} (\theta+\Theta \phi')\cdot D^{\alpha-\widetilde{\alpha}}\partial_y u\|_{L^2_{k+l}(\Omega)}\|D^\alpha \partial_y u\|_{L^2_{k+l}(\Omega)}.
\end{aligned}
\end{equation}
For any $m \ge 4$, we apply the inequality \eqref{e3} and
Cauchy inequality to get
\begin{equation}\label{c14}
\begin{aligned}
&\|D^{\widetilde{\alpha}} \theta \cdot D^{\alpha-\widetilde{\alpha}}\partial_y u\|_{L^2_{k+l}(\Omega)}\|D^\alpha \partial_y u\|_{L^2_{k+l}(\Omega)}\\
&\le \|D^{\widetilde{\alpha}-E_i} D^{E_i}\theta \cdot D^{\alpha-\widetilde{\alpha}}\partial_y u\|_{L^2_{k+l}(\Omega)}\|D^\alpha \partial_y u\|_{L^2_{k+l}(\Omega)}\\
&\le \|D^{E_i} \theta\|_{\mathcal{H}_{0}^{m-1}}\|\partial_y u\|_{\mathcal{H}_{l+1}^{m-1}}
\|D^\alpha \partial_y u\|_{L^2_{k+l}(\Omega)}\\
&\le \frac{1}{12}\|D^\alpha \partial_y u\|_{L^2_{k+l}(\Omega)}^2
+C\| \theta\|_{\mathcal{H}_{0}^{m}}^2\|u\|_{\mathcal{H}_{l}^{m}}^2,
\end{aligned}
\end{equation}
and
\begin{equation}\label{c15}
\begin{aligned}
&\|D^{\widetilde{\alpha}} (\Theta \phi')\cdot D^{\alpha-\widetilde{\alpha}}\partial_y u\|_{L^2_{k+l}(\Omega)}\|D^\alpha \partial_y u\|_{L^2_{k+l}(\Omega)}\\
&\le\|\partial_\tau^{\widetilde{\beta}}\Theta\|_{L^\infty(\mathbb{T}_x)}
\| D^{\alpha-\widetilde{\alpha}}\partial_y u\|_{L^2_{k+l}(\Omega)}\|D^\alpha \partial_y u\|_{L^2_{k+l}(\Omega)}\\
&\le \frac{1}{12} \|D^\alpha \partial_y u\|_{L^2_{k+l}(\Omega)}^2
+C\| \partial_\tau^{\widetilde{\beta}}\Theta\|_{H^1(\mathbb{T}_x)}^2
\|u\|_{\mathcal{H}_{l}^{m}}^2.
\end{aligned}
\end{equation}
Substituting \eqref{c14} and \eqref{c15} into \eqref{c13}, one arrives at for $m \ge 4$ that
\begin{equation}\label{c16}
\begin{aligned}
&\left|-\mu \int_\Omega [D^\alpha, (\theta+\Theta \phi'+1)]\partial_y u
\cdot \langle y\rangle^{2k+2l}D^\alpha \partial_y u dxdy\right|\\
&\le  \frac{\mu}{6} \|D^\alpha \partial_y u(t)\|_{L^2_{k+l}(\Omega)}^2
+C \| (u, \theta)(t)\|_{\mathcal{H}_{l}^{m}}^4
+C \sum_{|\beta|\le m+1}\| \partial_\tau^{{\beta}}\Theta(t)\|_{L^2(\mathbb{T}_x)}^4.
\end{aligned}
\end{equation}
\textbf{\underline{Step 3:}}
By virtue of the H\"{o}lder inequality,  it is easy to check that
\begin{equation}\label{c17}
\begin{aligned}
&\left|-2(k+l)\mu\int_\Omega  D^\alpha[(\theta+\Theta \phi'+1)\partial_y u] \cdot\langle y \rangle^{2k+2l-1}D^\alpha udxdy\right|\\
&\le 2(k+l)\mu \|\partial_y D^\alpha u\|_{L^2_{k+l}(\Omega)}
\|(\theta+\Theta \phi'+1)D^\alpha u\|_{L^2_{k+l-1}(\Omega)}\\
&\quad +2(k+l)\mu\sum_{0<\widetilde{\alpha}\le \alpha}
\|D^{\widetilde{\alpha}}(\theta+\Theta \phi')\cdot D^{\alpha-\widetilde{\alpha}}u_y\|_{L^2_{k+l-1}(\Omega)}
\|D^\alpha u\|_{L^2_{k+l}(\Omega)}.
\end{aligned}
\end{equation}
One applies the Sobolev and Cauchy inequalities to obtain
\begin{equation}\label{c18}
\begin{aligned}
&2(k+l)\mu\|\partial_y D^\alpha u\|_{L^2_{k+l}(\Omega)}
\|(\theta+\Theta \phi'+1)D^\alpha u\|_{L^2_{k+l-1}(\Omega)}\\
&\le 2(k+l)\mu\|\partial_y D^\alpha u\|_{L^2_{k+l}(\Omega)}
(\|\theta\|_{L^\infty(\Omega)}+
\|\Theta\|_{L^\infty(\mathbb{T}_x)}+1)
\|D^\alpha u\|_{L^2_{k+l-1}(\Omega)}\\
&\le \frac{\mu}{12}\|\partial_y D^\alpha u\|_{L^2_{k+l}(\Omega)}^2
+C(\|\theta\|_{H^2(\Omega)}+\|\Theta\|_{H^1(\mathbb{T}_x)}+1)^2
\|D^\alpha u\|_{L^2_{k+l-1}(\Omega)}^2.
\end{aligned}
\end{equation}
By virtue of the inequality \eqref{e3} and Sobolev inequality, it is easy to deduce that
\begin{equation}\label{c19}
\begin{aligned}
&\|D^{\widetilde{\alpha}} \theta
\cdot D^{\alpha-\widetilde{\alpha}}\partial_y u\|_{L^2_{k+l-1}(\Omega)}
\|D^\alpha u\|_{L^2_{k+l}(\Omega)}\\
&=\|D^{\widetilde{\alpha}-E_i} D^{E_i}\theta
\cdot D^{\alpha-\widetilde{\alpha}}\partial_y u\|_{L^2_{k+l-1}(\Omega)}
\|D^\alpha u\|_{L^2_{k+l}(\Omega)}\\
&\le \|\theta\|_{\mathcal{H}^{m}_0}\|\partial_y u\|_{\mathcal{H}^{m-1}_l}
\|D^\alpha u\|_{L^2_{k+l}(\Omega)}\\
&\le C\|\theta\|_{\mathcal{H}^{m}_0}^2+C\|u\|_{\mathcal{H}^{m}_l}^4,
\end{aligned}
\end{equation}
and
\begin{equation}\label{c110}
\begin{aligned}
&\|D^{\widetilde{\alpha}}(\Theta \phi')
\cdot D^{\alpha-\widetilde{\alpha}}\partial_y u\|_{L^2_{k+l-1}(\Omega)}
\|D^\alpha u\|_{L^2_{k+l}(\Omega)}\\
&\le\|\partial_\tau^{\widetilde{\beta}}\Theta\|_{L^\infty(\mathbb{T}_x)}
\| D^{\alpha-\widetilde{\alpha}}\partial_y u\|_{L^2_{k+l-1}(\Omega)}
\|D^\alpha  u\|_{L^2_{k+l}(\Omega)}\\
&\le C\| \partial_\tau^{\widetilde{\beta}}\Theta\|_{H^1(\mathbb{T}_x)}^2
+C\|u\|_{\mathcal{H}_{l}^{m}}^4,
\end{aligned}
\end{equation}
provided $m\ge 4$.
Then, substituting \eqref{c18}-\eqref{c110} into \eqref{c17}, we obtain directly
\begin{equation}\label{c111}
\begin{aligned}
&\left|-2(k+l)\mu\int_\Omega  D^\alpha[(\theta+\Theta \phi'+1)\partial_y u] \cdot\langle y \rangle^{2k+2l-1}D^\alpha udxdy\right|\\
&\le \frac{\mu}{12}\|\partial_y D^\alpha u(t)\|_{L^2_{k+l}(\Omega)}^2
+C(\|(u, \theta)(t)\|_{\mathcal{H}^{m}_l}^4+1)
+C\sum_{|\beta|\le m+1}\| \partial_\tau^{{\beta}}\Theta(t)\|_{L^2(\mathbb{T}_x)}^4.
\end{aligned}
\end{equation}
\textbf{\underline{Step 4:}} Finally, we deal with the term
$
-\mu\int_{\mathbb{T}_x} \left. D^\alpha[(\theta+1)\partial_y u] \cdot
D^\alpha u\right|_{y=0}dx
$
on the righthand side of \eqref{c11}.
\textbf{{Case 1: $|\alpha|\le m-1$.}}~Indeed, we can apply the simple trace estimate
to get that
\begin{equation}\label{c112}
\begin{aligned}
&\left|-\mu\int_{\mathbb{T}_x} \left. D^\alpha[(\theta+1)\partial_y u] \cdot
D^\alpha u\right|_{y=0}dx\right|\\
&\le
\mu\|D^\alpha[(\theta+1)\partial_y u]\|_{L^2(\Omega)}\|\partial_y D^\alpha u\|_{L^2(\Omega)}
+\mu\|\partial_y D^\alpha[(\theta+1)\partial_y u]\|_{L^2(\Omega)}
\|D^\alpha u\|_{L^2(\Omega)}.
\end{aligned}
\end{equation}
By virtue of \eqref{e3} and Cauchy inequality, one arrives at for $m \ge 4$ that
\begin{equation}\label{c113}
\|D^\alpha[(\theta+1)\partial_y u]\|_{L^2(\Omega)}\|\partial_y D^\alpha u\|_{L^2(\Omega)}
\le C\|\theta\|_{\mathcal{H}^{m-1}_{0}}\|\partial_y u\|_{\mathcal{H}^{m-1}_{0}}
+\|\partial_y D^\alpha u\|_{L^2(\Omega)}^2
\le C\|(u, \theta)\|_{\mathcal{H}^{m}_{0}}^2,
\end{equation}
and
\begin{equation}\label{c114}
\begin{aligned}
&\|\partial_y D^\alpha[(\theta+1)\partial_y u]\|_{L^2(\Omega)}
\|D^\alpha u\|_{L^2(\Omega)}\\
&\le \|D^\alpha(\partial_y \theta \partial_y u)\|_{L^2(\Omega)}
\|D^\alpha u\|_{L^2(\Omega)}
+\|D^\alpha[(\theta+1)\partial^2_y u]\|_{L^2(\Omega)}
\|D^\alpha u\|_{L^2(\Omega)}\\
&\le C\|\partial_y \theta\|_{\mathcal{H}^{m-1}_{0}}\|\partial_y u\|_{\mathcal{H}^{m-1}_{0}}
\|D^\alpha u\|_{L^2(\Omega)}
+ C\|\theta\|_{\mathcal{H}^{m-1}_{0}}\|\partial^2_y u\|_{\mathcal{H}^{m-1}_{0}}
\|D^\alpha u\|_{L^2(\Omega)}\\
&\quad
+\|D^\alpha \partial_y^2 u\|_{L^2(\Omega)}\|D^\alpha u\|_{L^2(\Omega)}\\
&\le \delta_1 \|\partial_y u\|_{\mathcal{H}^{m}_{0}}^2
+C\delta_1^{-1}(\|(u, \theta)\|_{\mathcal{H}^{m}_{0}}^4+1).
\end{aligned}
\end{equation}
Substituting \eqref{c113} and \eqref{c114} into \eqref{c112},
we find
\begin{equation}\label{c115}
\left|-\mu \int_{\mathbb{T}_x} \left. D^\alpha[(\theta+1)\partial_y u] \cdot
D^\alpha u\right|_{y=0}dx\right|
\le \delta_1 \mu \|\partial_y u(t)\|_{\mathcal{H}^{m}_{0}}^2
+C\delta_1^{-1}(\|(u, \theta)(t)\|_{\mathcal{H}^{m}_{0}}^4+1).
\end{equation}
\textbf{{Case 2: $|\alpha|=m$.}}
In view of $|\alpha|=|\beta|+k=m, ~|\beta|\le m-1$, we deduce that $k \ge 1$.
Let $\gamma \triangleq \alpha-E_3=(\beta, k-1)$, then $|\gamma|\le m-1$.
Then, it is easy to deduce that
\begin{equation}\label{c116}
\begin{aligned}
&-\mu\int_{\mathbb{T}_x} \left. D^\alpha[(\theta+1)\partial_y u] \cdot
D^\alpha u\right|_{y=0}dx\\
&=-\mu\int_{\mathbb{T}_x} \left. D^\gamma[\partial_y \theta \partial_y u+
(\theta+1)\partial^2_y u] \cdot
D^\alpha u\right|_{y=0}dx\\
&=-\mu\int_{\mathbb{T}_x} \left. D^\gamma(\partial_y \theta \partial_y u)
\cdot D^\alpha u\right|_{y=0}dx
-\mu\int_{\mathbb{T}_x} \left. (\theta+1)D^\gamma \partial^2_y u \cdot
D^\alpha u\right|_{y=0}dx\\
&\quad ~~-\mu\sum_{0<\widetilde{\gamma}\le \gamma}
\int_{\mathbb{T}_x} \left. \left(\begin{array}{c}\gamma \\ \widetilde{\gamma}
\end{array}\right)D^{\widetilde{\gamma}}\theta~ D^{\widetilde{\gamma}-\gamma}\partial^2_y u
\cdot D^\alpha u\right|_{y=0}dx\\
&\triangleq G_1+G_2+G_3.
\end{aligned}
\end{equation}
\textbf{\underline{Estimate for $G_1$:}}
We apply the inequality \eqref{e3} and Cauchy inequality to deduce for $m \ge 4$
\begin{equation}\label{c117}
\begin{aligned}
|G_1|
&\le \mu\|\partial_y D^\gamma(\partial_y \theta ~\partial_y u)\|_{L^2(\Omega)}
\|D^\alpha u\|_{L^2(\Omega)}
+\mu\|D^\gamma (\partial_y \theta~ \partial_y u)\|_{L^2(\Omega)}
\|\partial_y D^\alpha u\|_{L^2(\Omega)}\\
&\le \mu(\|\partial_y^2 \theta\|_{\mathcal{H}^{m-1}_{0}}\|\partial_y u\|_{\mathcal{H}^{m-1}_{0}}
      +\|\partial_y \theta\|_{\mathcal{H}^{m-1}_{0}}\|\partial_y^2 u\|_{\mathcal{H}^{m-1}_{0}})
      \|D^\alpha u\|_{L^2(\Omega)}\\
&\quad +\mu\|\partial_y \theta\|_{\mathcal{H}^{m-1}_{0}}
       \|\partial_y u\|_{\mathcal{H}^{m-1}_{0}}
       \|\partial_y D^\alpha u\|_{L^2(\Omega)}\\
&\le \delta_1 \mu\|\partial_y u(t)\|_{\mathcal{H}^{m}_{0}}^2
     +\delta_1 \kappa\|\partial_y \theta(t)\|_{\mathcal{H}^{m}_{0}}^2
     +C\delta_1^{-1}\|(u, \theta)(t)\|_{\mathcal{H}^{m}_{0}}^4.
\end{aligned}
\end{equation}
\textbf{\underline{Estimate for $G_3$:}}
With the help of \eqref{e3} and Cauchy inequality, it is easy to check that for $m \ge 5$
\begin{equation*}
\begin{aligned}
&\left|\mu\int_{\mathbb{T}_x} \left.
D^{\widetilde{\gamma}}\theta~ D^{\widetilde{\gamma}-\gamma}\partial^2_y u
\cdot D^\alpha u\right|_{y=0}dx\right|\\
&\le C\mu\|\partial_y (D^{\widetilde{\gamma}} \theta~
        D^{\gamma-\widetilde{\gamma}}\partial_y^2 u)
        \|_{L^2(\Omega)} \|D^\alpha u\|_{L^2(\Omega)}
     +C\mu\|D^{\widetilde{\gamma}}\theta ~D^{\gamma-\widetilde{\gamma}}\partial_y^2 u
     \|_{L^2(\Omega)} \|\partial_y D^\alpha u\|_{L^2(\Omega)}\\
&\le C\mu(\|D^{\widetilde{\gamma}-E_i}D^{E_i}\partial_y \theta~
      D^{\gamma-\widetilde{\gamma}}\partial^2_y u\|_{L^2(\Omega)}
      +\|D^{\widetilde{\gamma}-E_i}D^{E_i}\theta~
      D^{\gamma-\widetilde{\gamma}}\partial^3_y u\|_{L^2(\Omega)})
     \|D^\alpha u\|_{L^2(\Omega)}\\
&\quad +C\mu\|D^{\widetilde{\gamma}-E_i}D^{E_i}\theta~
        D^{\gamma-\widetilde{\gamma}}\partial_y^2 u\|_{L^2(\Omega)}
        \|\partial_y D^\alpha u\|_{L^2(\Omega)}\\
&\le C\mu(\|D^{E_i}\partial_y \theta\|_{\mathcal{H}^{m-2}_{0}}
      \|\partial_y^2 u\|_{\mathcal{H}^{m-2}_{0}}
      +\|D^{E_i}\theta\|_{\mathcal{H}^{m-2}_{0}}
      \|\partial_y^3 u\|_{\mathcal{H}^{m-2}_{0}})
      \|D^\alpha u\|_{L^2(\Omega)}\\
&\quad +C\mu\|D^{E_i} \theta\|_{\mathcal{H}^{m-2}_{0}}
        \|\partial_y^2 u\|_{\mathcal{H}^{m-2}_{0}}
       \|\partial_y D^\alpha u\|_{L^2(\Omega)}\\
&\le \delta_1 \mu \|\partial_y u(t)\|_{\mathcal{H}^{m}_{0}}^2
     +C\delta_1^{-1}(\|(u, \theta)(t)\|_{\mathcal{H}^{m}_{0}}^4+1),
\end{aligned}
\end{equation*}
which, implies that
\begin{equation}\label{c118}
|G_3|\le \delta_1 \mu\|\partial_y u(t)\|_{\mathcal{H}^{m}_{0}}^2
     +C\delta_1^{-1}(\|(u, \theta)(t)\|_{\mathcal{H}^{m}_{0}}^4+1).
\end{equation}
\textbf{\underline{Estimate for $G_2$:}}
Indeet, by virtue of the equation \eqref{eq4}$_1$, it is easy to deduce that
\begin{equation*}
\begin{aligned}
(\theta+\Theta \phi'+1)D^\gamma \partial_y^2 u
&=D^\gamma \partial_t u
+D^\gamma I_1
-\sum_{\widetilde{\gamma}\le \gamma}
\left(\begin{array}{c} \gamma \\ \widetilde{\gamma} \end{array}\right)
D^{\widetilde{\gamma}} (\partial_y\theta+\Theta \phi'') ~
D^{\gamma-\widetilde{\gamma}}\partial_y u \\
&\quad -\sum_{0<\widetilde{\gamma}\le \gamma}
\left(\begin{array}{c} \gamma \\ \widetilde{\gamma} \end{array}\right)
D^{\widetilde{\gamma}}(\theta+\Theta \phi') ~
D^{\gamma-\widetilde{\gamma}}\partial^2_y u -D^\gamma r_1,
\end{aligned}
\end{equation*}
which yields directly
\begin{equation}\label{c119}
\begin{aligned}
G_2
&=\mu\int_{\mathbb{T}_x} \left. D^\gamma \partial_t u~D^\alpha u \right|_{y=0}dx
+\mu\int_{\mathbb{T}_x} \left. D^\gamma I_1 ~D^\alpha u\right|_{y=0}dx\\
&\quad-\mu\sum_{\widetilde{\gamma}\le \gamma}
\left(\begin{array}{c} \gamma \\ \widetilde{\gamma} \end{array}\right)
\int_{\mathbb{T}_x} \left. D^{\widetilde{\gamma}} \partial_y\theta ~
D^{\gamma-\widetilde{\gamma}}\partial_y u ~D^\alpha u\right|_{y=0}dx\\
&\quad -\mu\sum_{0<\widetilde{\gamma}\le \gamma}
\left(\begin{array}{c} \gamma \\ \widetilde{\gamma} \end{array}\right)
\int_{\mathbb{T}_x} \left. D^{\widetilde{\gamma}}\theta ~
D^{\gamma-\widetilde{\gamma}}\partial^2_y u ~D^\alpha u\right|_{y=0} dx\\
&\quad-\mu\int_{\mathbb{T}_x} \left. D^\gamma r_1 ~D^\alpha u\right|_{y=0}dx\\
&\triangleq G_{21}+G_{22}+G_{23}+G_{24}+G_{25}.
\end{aligned}
\end{equation}
In view of the Cauchy inequality, one arrives at
\begin{equation}\label{c120}
\begin{aligned}
|G_{21}|
&\le \mu\|\partial_y D^{\gamma+e_1}u\|_{L^2(\Omega)}\|D^\alpha u\|_{L^2(\Omega)}
     +\mu\|D^{\gamma+e_1}u\|_{L^2(\Omega)}\|\partial_y D^\alpha u\|_{L^2(\Omega)}\\
&\le \frac{\mu}{12} \|\partial_y D^\alpha u(t)\|_{L^2(\Omega)}^2
     +C\|u(t)\|_{\mathcal{H}^m_0}^2.
\end{aligned}
\end{equation}
By virtue of the definition of $I_1$, it is easy to deduce that
\begin{equation}\label{c121}
\begin{aligned}
G_{22}
&=
\mu\int_{\mathbb{T}_x}D^{\gamma}(uu_x)D^\alpha u|_{y=0} dx
+\mu\int_{\mathbb{T}_x}D^{\gamma}(vu_y)D^\alpha u|_{y=0} dx\\
&\quad-\mu\int_{\mathbb{T}_x}D^{\gamma}(h h_x)D^\alpha u|_{y=0} dx
-\mu\int_{\mathbb{T}_x}D^{\gamma}(g h_y)D^\alpha u|_{y=0} dx\\
&\triangleq G_{221}+G_{222}+G_{223}+G_{224}.
\end{aligned}
\end{equation}
One applies the inequality \eqref{e3} and Cauchy inequality to obtain
\begin{equation}\label{c122}
\begin{aligned}
|G_{221}|
&\le \mu\|\partial_y D^\gamma(u u_x)\|_{L^2(\Omega)}
     \|D^\alpha u\|_{L^2(\Omega)}
     +\mu\|D^\gamma(u u_x)\|_{L^2(\Omega)}\|\partial_y D^\alpha u\|_{L^2(\Omega)}\\
&\le \mu(\|\partial_y u\|_{\mathcal{H}^{m-1}_0}\|\partial_x u\|_{\mathcal{H}^{m-1}_0}
     +\|u\|_{\mathcal{H}^{m-1}_0}\|\partial_y \partial_x u\|_{\mathcal{H}^{m-1}_0})
     \|D^\alpha u\|_{L^2(\Omega)}\\
&\quad +\mu(\|u\|_{\mathcal{H}^{m-1}_0}\|\partial_x u\|_{\mathcal{H}^{m-1}_0}
     +\|u\|_{\mathcal{H}^{m-1}_0}\| \partial_x u\|_{\mathcal{H}^{m-1}_0})
     \|\partial_y D^\alpha u\|_{L^2(\Omega)}\\
&\le \delta_1 \mu \|\partial_y u\|_{\mathcal{H}^m_0}^2
+C\delta_1^{-1}(\|u\|_{\mathcal{H}^{m}_0}^4+1),
\end{aligned}
\end{equation}
provided $m \ge 3$.
Following the idea as Liu et al.\cite{Liu-Xie-Yang},
it is easy to obtain the following estimate
\begin{equation}\label{c123}
|G_{222}|\le  \delta_1 \mu\|\partial_y u\|_{\mathcal{H}^m_0}^2
+C\delta_1^{-1}(\|u\|_{\mathcal{H}^{m}_0}^4+1).
\end{equation}
Similarly, it is easy to deduce that
\begin{equation}\label{c124}
\begin{aligned}
&|G_{223}|\le  \delta_1 \mu \|\partial_y u\|_{\mathcal{H}^m_0}^2
+\delta_1 \nu \|\partial_y h\|_{\mathcal{H}^m_0}^2
+C\delta_1^{-1}(\|(u, h)\|_{\mathcal{H}^{m}_0}^4+1),\\
&|G_{224}|\le \delta_1 \mu \|\partial_y u\|_{\mathcal{H}^m_0}^2
+C\delta_1^{-1}(\|(u,h)\|_{\mathcal{H}^{m}_0}^4+1).
\end{aligned}
\end{equation}
Substituting the estimate \eqref{c122}-\eqref{c124} into \eqref{c121}, one obtains that
\begin{equation}\label{c125}
\begin{aligned}
|G_{22}|\le  \delta_1 \mu \|\partial_y u(t)\|_{\mathcal{H}^m_0}^2
+\delta_1 \nu \|\partial_y h(t)\|_{\mathcal{H}^m_0}^2
+C\delta_1^{-1}(\|(u, h)(t)\|_{\mathcal{H}^{m}_0}^4+1).
\end{aligned}
\end{equation}
In view of the trace inequality \eqref{e1} and Sobolev
inequality \eqref{e3}, we find for $m \ge 4$
\begin{equation*}
\begin{aligned}
&\mu\|\left.D^{\widetilde{\gamma}}\partial_y \theta~
D^{\gamma-\widetilde{\gamma}}\partial_y u~D^\alpha u\right|_{y=0}\|_{L^1{(\mathbb{T}_x)}}\\
&\le \mu\|\partial_y(D^{\widetilde{\gamma}}\partial_y \theta~
D^{\gamma-\widetilde{\gamma}}\partial_y u)\|_{L^2(\Omega)}\|D^\alpha u\|_{L^2(\Omega)}
+\mu\|D^{\widetilde{\gamma}}\partial_y \theta~
D^{\gamma-\widetilde{\gamma}}\partial_y u\|_{L^2(\Omega)}
\|\partial_y D^\alpha u\|_{L^2(\Omega)}\\
&\le \mu\|\partial_y^2 \theta\|_{\mathcal{H}^{m-1}_{0}}
       \|\partial_y u\|_{\mathcal{H}^{m-1}_{0}}\|D^\alpha u\|_{L^2(\Omega)}
      +\mu\|\partial_y \theta\|_{\mathcal{H}^{m-1}_{0}}
      \|\partial_y^2 u\|_{\mathcal{H}^{m-1}_{0}}\|D^\alpha u\|_{L^2(\Omega)}\\
&\quad +\mu\|\partial_y \theta\|_{\mathcal{H}^{m-1}_{0}}
   \|\partial_y u\|_{\mathcal{H}^{m-1}_{0}}
   \|\partial_y D^\alpha u\|_{L^2(\Omega)}\\
&\le \delta_1 \mu \|\partial_y u\|_{\mathcal{H}^m_0}^2
+\delta_1 \kappa\|\partial_y \theta\|_{\mathcal{H}^m_0}^2
+C\delta_1^{-1}(\|(u, \theta)\|_{\mathcal{H}_{0}^{m}}^4+1),
\end{aligned}
\end{equation*}
which, implies that
\begin{equation}\label{c126}
|G_{23}|\le \delta_1 \mu \|\partial_y u(t)\|_{\mathcal{H}^m_0}^2
+\delta_1 \kappa\|\partial_y \theta(t)\|_{\mathcal{H}^m_0}^2
+C\delta_1^{-1}(\|(u, \theta)(t)\|_{\mathcal{H}_{0}^{m}}^4+1).
\end{equation}
Similarly, it is easy to deduce that
\begin{equation}\label{c127}
|G_{24}|\le \delta_1 \mu \|\partial_y u(t)\|_{\mathcal{H}^m_0}^2
+C\delta_1^{-1}(\|(u,\theta)(t)\|_{\mathcal{H}_{0}^{m}}^4+1).
\end{equation}
In view of the definition of $r_1$(see \eqref{function2})
and trace inequality \eqref{e1}, one attains directly
\begin{equation}\label{c128}
\begin{aligned}
|G_{25}|
&=\left|\mu\int_{\mathbb{T}_x} D^\gamma P_x\cdot D^\alpha u dx\right|\\
&\le \mu\|D^\gamma P_x\|_{L^2(\mathbb{T}_x)}\|D^\alpha u\|_{L^2(\mathbb{T}_x)}\\
&\le \sqrt{2}\mu\|D^\gamma P_x\|_{L^2(\mathbb{T}_x)}
     \|D^\alpha u\|_{L^2(\Omega)}^{\frac{1}{2}}
     \|D^\alpha \partial_y u\|_{L^2(\Omega)}^{\frac{1}{2}}\\
&\le \frac{\mu}{12} \|\partial_y D^\alpha u\|_{L^2(\Omega)}^2
+C(\|u\|_{\mathcal{H}^{m}_0}^2+\|D^\gamma P_x\|_{L^2(\mathbb{T}_x)}^2).
\end{aligned}
\end{equation}
Substituting \eqref{c120}, \eqref{c125}, \eqref{c126}, \eqref{c127}
and \eqref{c128} into  \eqref{c119}, we find
\begin{equation*}
\begin{aligned}
|G_{2}|\le
&\frac{\mu}{6}\|\partial_y D^\alpha u(t)\|_{L^2(\Omega)}^2
+\delta_1 \mu\|\partial_y u(t)\|_{\mathcal{H}^m_0}^2
+\delta_1 \kappa\|\partial_y \theta(t)\|_{\mathcal{H}^m_0}^2
+\delta_1 \nu\|\partial_y h(t)\|_{\mathcal{H}^m_0}^2\\
&+C\delta_1^{-1}(\|(u, \theta, h)(t)\|_{\mathcal{H}^{m}_0}^4+1)
+C\sum_{|\beta|\le m-1}\|\partial_\tau^{\beta} P_x(t)\|_{L^2(\mathbb{T}_x)}^2,
\end{aligned}
\end{equation*}
which, together with \eqref{c117} and \eqref{c118}, gives for $|\alpha|=m$ that
\begin{equation}\label{c129}
\begin{aligned}
&|\mu\int_{\mathbb{T}_x} \left. D^\alpha[(\theta+1)\partial_y u] \cdot
D^\alpha u\right|_{y=0}dxdy|\\
&\le \frac{\mu}{6}\|\partial_y D^\alpha u(t)\|_{L^2(\Omega)}^2
+\delta_1 \mu\|\partial_y u(t)\|_{\mathcal{H}^m_0}^2
+\delta_1 \kappa\|\partial_y \theta(t)\|_{\mathcal{H}^m_0}^2
+\delta_1 \nu\|\partial_y h(t)\|_{\mathcal{H}^m_0}^2\\
&\quad +C\delta_1^{-1}(\|(u, \theta, h)(t)\|_{\mathcal{H}^{m}_0}^4+1)
+C\sum_{|\beta|\le m-1}\|\partial_\tau^{\beta} P_x(t)\|_{L^2(\mathbb{T}_x)}^2.
\end{aligned}
\end{equation}
The combination of \eqref{c115} and \eqref{c129} yields for $|\alpha|\le m$ that
\begin{equation}\label{c130}
\begin{aligned}
&|\mu\int_{\mathbb{T}_x} \left. D^\alpha[(\theta+1)\partial_y u] \cdot
D^\alpha u\right|_{y=0}dx|\\
&\le \frac{\mu}{6}\|\partial_y D^\alpha u(t)\|_{L^2(\Omega)}^2
+\delta_1 \mu\|\partial_y u(t)\|_{\mathcal{H}^m_0}^2
+\delta_1 \kappa\|\partial_y \theta(t)\|_{\mathcal{H}^m_0}^2
+\delta_1 \nu\|\partial_y h(t)\|_{\mathcal{H}^m_0}^2\\
&\quad +C\delta_1^{-1}(\|(u, \theta, h)(t)\|_{\mathcal{H}^{m}_0}^4+1)
+C\sum_{|\beta|\le m-1}\|\partial_\tau^{\beta} P_x(t)\|_{L^2(\mathbb{T}_x)}^2.
\end{aligned}
\end{equation}
Substituting \eqref{c12}, \eqref{c16}, \eqref{c111} and \eqref{c130}
into \eqref{c11}, one finds that
\begin{equation}\label{c131}
\begin{aligned}
&\mu\int_\Omega \partial_y D^\alpha[(\theta+\Theta \phi'+1)\partial_y u] \cdot\langle y \rangle^{2k+2l}D^\alpha udxdy\\
&\le-\frac{\mu}{2}\|D^\alpha \partial_y u(t)\|_{L^2_{k+l}(\Omega)}
+\delta_1 \mu\|\partial_y u(t)\|_{\mathcal{H}^m_0}^2
+\delta_1 \kappa\|\partial_y \theta(t)\|_{\mathcal{H}^m_0}^2
+\delta_1 \nu\|\partial_y h(t)\|_{\mathcal{H}^m_0}^2\\
&\quad+C\delta_1^{-1}(\|(u, \theta, h)(t)\|_{\mathcal{H}^{m}_0}^4+1)
+\sum_{|\beta|\le m+1}\| \partial_\tau^{{\beta}}\Theta(t)\|_{L^2(\mathbb{T}_x)}^4
+\sum_{|\beta|\le m-1}\|\partial_\tau^{\beta} P_x(t)\|_{L^2(\mathbb{T}_x)}^2.
\end{aligned}
\end{equation}
Similarly(or following the idea as Liu et al.\cite{Liu-Xie-Yang}), one finds that
\begin{equation}\label{c132}
\begin{aligned}
&\nu \int_\Omega \partial_y^2 D^\alpha h \cdot\langle y \rangle^{2k+2l}D^\alpha h dxdy\\
&\le -\frac{\nu}{2}\|D^\alpha \partial_y h(t)\|_{L^2_{k+l}(\Omega)}^2
+\delta_1 \mu \|\partial_y u(t)\|_{\mathcal{H}^m_0}^2
+\delta_1 \nu \|\partial_y h(t)\|_{\mathcal{H}^m_0}^2\\
&\quad +C\delta_1^{-1}(\|(u, h)(t)\|_{\mathcal{H}^m_0}^4+1),
\end{aligned}
\end{equation}
and
\begin{equation}\label{c133}
\begin{aligned}
&\kappa\int_\Omega \partial_y^2 D^\alpha \theta \cdot\langle y \rangle^{2k+2l}D^\alpha \theta dxdy\\
&\le -\frac{\kappa}{2}\|D^\alpha \partial_y \theta(t)\|_{L^2_{k+l}(\Omega)}^2
+\delta_1 \mu \|\partial_y u(t)\|_{\mathcal{H}^m_0}^2
+\delta_1 \kappa \|\partial_y \theta(t)\|_{\mathcal{H}^m_0}^2\\
&\quad +\delta_1 \nu \|\partial_y h(t)\|_{\mathcal{H}^m_0}^2
+C\delta_1^{-1}(\|(u, \theta)(t)\|_{\mathcal{H}^m_0}^6+1).
\end{aligned}
\end{equation}
Therefore, the combination of \eqref{c131}-\eqref{c133} completes the proof of \eqref{336}.
\end{proof}

\begin{proof}[\textbf{Proof of \eqref{337}.}]
We will first handle the term
$
-\int_\Omega
D^\alpha I_2 \cdot\langle y \rangle^{2k+2l}D^\alpha \theta dxdy.
$
As we know that
$$
\begin{aligned}
I_2&\triangleq c_v[(u+U\phi')\partial_x \!+\!(v-U_x \phi)\partial_y ]\theta
+\!c_v \Theta_x \phi' u\!+\!c_v \Theta \phi'' v
\!-\!\mu \theta (u_y)^2\!-\!\mu (U\phi'')^2 \theta\\
&\quad \!-2\mu U \phi'' \theta u_y\!-\!\mu \Theta \phi'(u_y)^2
\!-\!2\mu \Theta U \phi' \phi'' u_y\!-\!2\mu U\phi'' u_y\!-\!\mu (u_y)^2
\!-\!\nu (h_y)^2\!-\!2\nu H \phi' h_y.
\end{aligned}
$$
Then it is easy to deduce that
\begin{equation}\label{c21}
D^\alpha I_2= I_{21}+I_{22}+I_{23},
\end{equation}
where
\begin{equation}\label{c22}
\begin{aligned}
&I_{21}\triangleq c_v[(u+U\phi')\partial_x +(v-U_x \phi)\partial_y ]D^\alpha\theta,\\
&I_{22}\triangleq c_v[D^\alpha,(u+U\phi')\partial_x +(v-U_x \phi)\partial_y ]\theta,\\
&I_{23}\triangleq D^\alpha(c_v \Theta_x \phi' u+c_v \Theta \phi'' v
-\mu \theta (u_y)^2-\mu (U\phi'')^2 \theta
-2\mu U \phi'' \theta u_y-\mu \Theta \phi'(u_y)^2)\\
&\quad \quad -D^\alpha(\!2\mu \Theta U \phi' \phi'' u_y+2\mu U\phi'' u_y
+\mu (u_y)^2+\nu (h_y)^2+2\nu H \phi' h_y).
\end{aligned}
\end{equation}
\textbf{\underline{Step 1:}}
Integrating by part and applying the divergence
free condition of velocity, we find for $m \ge 3$
\begin{equation}\label{c23}
\begin{aligned}
&-\int_\Omega I_{21} \cdot\langle y \rangle^{2k+2l}D^\alpha \theta dxdy\\
&=(2k+2l)\int_\Omega \langle y\rangle^{2k+2l-1}
(v-U_x \phi) |D^\alpha \theta|^2 dxdy\\
&\le C\|\langle y \rangle^{-1}{v}\|_{L^\infty(\Omega)}
\|D^\alpha \theta\|_{L^2_{k+l}(\Omega)}
+C\|\langle y \rangle^{-1}{U_x \phi}\|_{L^\infty(\Omega)}
\|D^\alpha \theta\|_{L^2_{k+l}(\Omega)}\\
&\le C\|u_x\|_{L^\infty(\Omega)}
\|D^\alpha \theta\|_{L^2_{k+l}(\Omega)}
+C\|{U_x }\|_{L^\infty(\mathbb{T}_x)}
\|D^\alpha \theta\|_{L^2_{k+l}(\Omega)}\\
&\le C\|u_x\|_{H^2(\Omega)}\|\theta\|_{\mathcal{H}^m_l}^2
+C\|U_x\|_{H^1(\mathbb{T}_x)}\|\theta\|_{\mathcal{H}^m_l}^2\\
&\le C(\|(u, \theta)\|_{\mathcal{H}^m_l}^4+1)
      +C\|U_x\|_{H^1(\mathbb{T}_x)}^2,
\end{aligned}
\end{equation}
where we have used the Hardy type inequality \eqref{e5}.\\
\textbf{\underline{Step 2:}}
It is easy to deduce that
\begin{equation}\label{c24}
\begin{aligned}
\|[D^\alpha,u \partial_x +v \partial_y ]\theta\|_{L^2_{k+l}(\Omega)}
\lesssim \sum_{0<\widetilde{\alpha}\le \alpha}
\|D^{\alpha}u\cdot D^{\alpha-\widetilde{\alpha}}\theta_x\|_{L^2_{k+l}(\Omega)}
+\sum_{0<\widetilde{\alpha}\le \alpha}
\|D^{\alpha}v \cdot D^{\alpha-\widetilde{\alpha}}\theta_y\|_{L^2_{k+l}(\Omega)}.
\end{aligned}
\end{equation}
\textbf{{Case 1: $\widetilde{k}=0$.}}
One can infer that
$D^{\widetilde{\alpha}}=\partial_\tau^{\widetilde{\beta}}$,
and $\widetilde{\beta}\ge e_i, i=1~{\rm or}~2$, $|k|\le m-1$.
Then, we find
\begin{equation}\label{c25}
\begin{aligned}
\|D^{\widetilde{\alpha}}u D^{\alpha-\widetilde{\alpha}}\theta_x\|_{L^2_{k+l}(\Omega)}
=\|\partial_\tau^{\widetilde{\beta}-e_i} \partial_\tau^{e_i}u\cdot
\partial_\tau^{\beta-\widetilde{\beta}}\partial_y^k \theta_x\|_{L^2_{k+l}(\Omega)}
\le \|\partial_\tau^{e_i} u\|_{\mathcal{H}^{m-1}_{0}}
      \|\theta_x\|_{\mathcal{H}^{m-1}_{l}}
\le \|u\|_{\mathcal{H}^{m}_{0}}\|\theta\|_{\mathcal{H}^{m}_{l}},
\end{aligned}
\end{equation}
provided that $m\ge4$.
On the other hand, it is easy to deduce that
\begin{equation}\label{c26}
\|D^{\widetilde{\alpha}}v~D^{\alpha-\widetilde{\alpha}}\theta_y\|_{L^2_{k+l}(\Omega)}
=\|\partial_\tau^{\widetilde{\beta}}\partial_y^{-1}\partial_x u
~\partial_\tau^{\beta-\widetilde{\beta}}\partial_y^k\theta_y\|_{L^2_{k+l}(\Omega)}.
\end{equation}
If $|\alpha|=|\beta|+k\le m-1$, one applies the inequality \eqref{e3} to obtain that
\begin{equation}\label{c27}
\|\partial_\tau^{\widetilde{\beta}}\partial_y^{-1}\partial_x u
~\partial_\tau^{\beta-\widetilde{\beta}}\partial_y^k\theta_y\|_{L^2_{k+l}(\Omega)}
\le \|u_x\|_{\mathcal{H}^{m-1}_0}\|\theta_y\|_{\mathcal{H}^{m-1}_{l+1}}
\le \|u\|_{\mathcal{H}^{m}_0}\|\theta\|_{\mathcal{H}^{m}_l},
\end{equation}
provided that $m\ge4$.
If $|\alpha|=|\beta|+k=m$, then it infers that $k \ge 1$ and one finds
\begin{equation}\label{c28}
\begin{aligned}
&\|\partial_\tau^{\widetilde{\beta}}\partial_y^{-1}\partial_x u
~\partial_\tau^{\beta-\widetilde{\beta}}\partial_y^k\theta_y\|_{L^2_{k+l}(\Omega)}
=\|\partial_\tau^{\widetilde{\beta}}\partial_y^{-1}\partial_x u
~\partial_\tau^{\beta-\widetilde{\beta}}\partial_y^{k-1}\theta_{yy}\|_{L^2_{k+l}(\Omega)}\\
&\le \|\partial_x \partial_\tau^{e_i} u\|_{\mathcal{H}^{m-2}_0}
     \|\partial_y^2 \theta\|_{\mathcal{H}^{m-2}_{l+2}}
\le \|u\|_{\mathcal{H}^{m}_0}\|\theta\|_{\mathcal{H}^{m}_l},
\end{aligned}
\end{equation}
provided that $m\ge5$.
The combination of \eqref{c25}-\eqref{c28} yields directly that
\begin{equation}\label{c29}
\|[D^\alpha,u \partial_x +v \partial_y ]\theta\|_{L^2_{k+l}(\Omega)}
\le C\|(u, \theta)(t)\|_{\mathcal{H}^{m}_{l}}^2.
\end{equation}
\textbf{{Case 2: $\widetilde{k}\ge 1$.}} Then, we get that $\widetilde{\alpha} \ge E_3$
and obtain
\begin{equation}\label{c210}
\begin{aligned}
\|D^{\widetilde{\alpha}}u~D^{\alpha-\widetilde{\alpha}}\theta_x\|_{L^2_{k+l}(\Omega)}
=\|D^{\widetilde{\alpha}-E_3}D^{E_3}u
~D^{\alpha-\widetilde{\alpha}}\theta_x\|_{L^2_{k+l}(\Omega)}
\le \|\partial_x \theta\|_{\mathcal{H}^{m}_{0}}
\|D^{E_3}u\|_{\mathcal{H}^{m-1}_{l+1}}
\le C\|(u, \theta)\|_{\mathcal{H}^{m}_{l}}^2.
\end{aligned}
\end{equation}
On the other hand, one applies the inequality \eqref{e3}
and divergence free of velocity to deduce for $m \ge 4$
\begin{equation*}
\|D^{\widetilde{\alpha}}v~D^{\alpha-\widetilde{\alpha}}\theta_y\|_{L^2_{k+l}(\Omega)}
=\|D^{\widetilde{\alpha}-E_3}u_x~D^{\alpha-\widetilde{\alpha}}\theta_y\|_{L^2_{k+l}(\Omega)}
\le \|\partial_x u\|_{\mathcal{H}^{m-1}_{0}}
\|\partial_y \theta\|_{\mathcal{H}^{m-1}_{l+1}}
\le C\|(u,\theta)\|_{\mathcal{H}^{m}_{l}}^2,
\end{equation*}
which, together with \eqref{c210}, yields that
\begin{equation}\label{c211}
\|[D^\alpha,u \partial_x +v \partial_y ]\theta\|_{L^2_{k+l}(\Omega)}
\le C\|(u,\theta)\|_{\mathcal{H}^{m}_{l}}^2.
\end{equation}
Substituting the combination of \eqref{c29} and \eqref{c211} into
\eqref{c24}, we find for all $\widetilde{k}\ge 0$
\begin{equation}\label{c212}
\|[D^\alpha,u \partial_x +v \partial_y ]\theta\|_{L^2_{k+l}(\Omega)}
\le C\|(u,\theta)\|_{\mathcal{H}^{m}_{l}}^2.
\end{equation}
On the other hand, it is easy to deduce that
\begin{equation}\label{c213}
\|[D^\alpha, U\phi'\partial_x -U_x \phi\partial_y]\theta\|_{L^2_{k+l}(\Omega)}
\lesssim \!\!\!\!\sum_{0<\widetilde{\alpha}\le \alpha}\!\!\!
\|D^{\widetilde{\alpha}}( U\phi')
D^{\alpha-\widetilde{\alpha}}\theta_x\|_{L^2_{k+l}(\Omega)}
+\!\!\!\sum_{0<\widetilde{\alpha}\le \alpha}\!\!\!
\|D^{\widetilde{\alpha}}(U_x \phi)
D^{\alpha-\widetilde{\alpha}}\theta_y\|_{L^2_{k+l}(\Omega)}.
\end{equation}
In view of the fact $|\alpha-\widetilde{\alpha}|\le m-1$,
then one arrives at
\begin{equation}\label{c214}
\|D^{\widetilde{\alpha}}( U\phi')
D^{\alpha-\widetilde{\alpha}}\theta_x\|_{L^2_{k+l}(\Omega)}
\le \|\langle y\rangle^{\widetilde{k}}D^{\widetilde{\alpha}}(U\phi')\|_{L^\infty(\Omega)}
    \|\langle y\rangle^{k+l-\widetilde{k}}D^{\alpha-\widetilde{\alpha}}\theta_x\|_{L^2(\Omega)}
\le \|\partial_\tau^{\widetilde{\beta}}U\|_{L^\infty(\mathbb{T}_x)}
    \|\theta\|_{\mathcal{H}^m_l},
\end{equation}
and
\begin{equation}\label{c215}
\|D^{\widetilde{\alpha}}(U_x \phi)
~D^{\alpha-\widetilde{\alpha}}\theta_y\|_{L^2_{k+l}(\Omega)}
\le \|\partial_\tau^{\widetilde{\beta}}U_x\|_{L^\infty(\mathbb{T}_x)}
    \|\theta\|_{\mathcal{H}^m_l}
\end{equation}
Substituting \eqref{c214} and \eqref{c215} into \eqref{c213},
we find directly
\begin{equation*}
\|[D^\alpha, U\phi'\partial_x -U_x \phi\partial_y]\theta\|_{L^2_{k+l}(\Omega)}
\le C\|\theta\|_{\mathcal{H}^m_l}^2
+\sum_{|\beta|\le m+2}\|\partial_\tau^\beta U\|_{L^2(\mathbb{T}_x)}^2,
\end{equation*}
which, together with \eqref{c212}, yields directly
\begin{equation}\label{c216}
-\int_\Omega I_{22} \cdot\langle y \rangle^{2k+2l}D^\alpha \theta dxdy
\le C(\|(u, \theta)(t)\|_{\mathcal{H}^m_l}^4+1)
+\sum_{|\beta|\le m+2}\|\partial_\tau^\beta U(t)\|_{L^2(\mathbb{T}_x)}^4.
\end{equation}
\textbf{\underline{Step 3:}} Finally, we will give the estimate for $\|I_{23}\|_{L^2_{k+l}(\Omega)}$. Recall the definition
of $I_{23}$(see \eqref{c22})
\begin{equation*}
\begin{aligned}
&I_{23}\triangleq D^\alpha(c_v \Theta_x \phi' u+c_v \Theta \phi'' v
-\mu \theta (u_y)^2-\mu (U\phi'')^2 \theta
-2\mu U \phi'' \theta u_y-\mu \Theta \phi'(u_y)^2)\\
&\quad \quad -D^\alpha(\!2\mu \Theta U \phi' \phi'' u_y+2\mu U\phi'' u_y
+\mu (u_y)^2+\nu (h_y)^2+2\nu H \phi' h_y).
\end{aligned}
\end{equation*}
\textbf{{Estimate for $\|D^{\alpha}(\Theta_x \phi' u)\|_{L^2_{l+k}(\Omega)}$}.}
In view of the definition of $\phi$(see \eqref{function1}), we can obtain
\begin{equation}\label{c217}
\begin{aligned}
&\|D^{\alpha}(\Theta_x \phi' u)\|_{L^2_{l+k}(\Omega)}
\le C\sum_{\widetilde{\alpha}\le \alpha}
\|D^{\widetilde{\alpha}}u\cdot D^{\alpha-\widetilde{\alpha}}(\Theta_x \phi')\|_{L^2_{l+k}(\Omega)}\\
&\le C\sum_{\widetilde{\alpha}\le \alpha}
\|\langle y\rangle^{l+\widetilde{k}}D^{\widetilde{\alpha}}u\|_{L^2(\Omega)}
\|\langle y\rangle^{k-\widetilde{k}}D^{\alpha-\widetilde{\alpha}}(\Theta_x \phi')\|_{L^\infty(\Omega)}\\
&\le C\|u\|_{\mathcal{H}^m_l}
\sum_{|{\beta}| \le m+2}\|\partial_\tau^{\beta}\Theta\|_{L^2(\mathbb{T}_x)}.
\end{aligned}
\end{equation}
\textbf{{Estimate for $\|D^{\alpha}(\Theta \phi'' v)\|_{L^2_{l+k}(\Omega)}$}.}
By virtue of the divergence free condition of velocity, one arrives at
\begin{equation}\label{c218}
\|D^{\alpha}(\Theta \phi'' v)\|_{L^2_{l+k}(\Omega)}
\le \sum_{\widetilde{\alpha} \le \alpha}
\|D^{\widetilde{\alpha}}v D^{\alpha-\widetilde{\alpha}}(\Theta \phi'')\|_{L^2_{k+l}(\Omega)}
\le \sum_{\widetilde{\alpha} \le \alpha}
\|D^{\widetilde{\alpha}+E_2}\partial_y^{-1}u D^{\alpha-\widetilde{\alpha}}
(\Theta \phi'')\|_{L^2_{k+l}(\Omega)}.
\end{equation}
If $\widetilde{k}=0$, the application of Hardy type inequality \eqref{e7} yields directly
\begin{equation}\label{c219}
\begin{aligned}
&\|D^{\widetilde{\alpha}+e_2}\partial_y^{-1}u \cdot D^{\alpha-\widetilde{\alpha}}
(\Theta \phi'')\|_{L^2_{k+l}(\Omega)}\\
&\le \|\langle y \rangle^{-1}
{\partial_\tau^{\widetilde{\beta}+e_2}\partial_y^{-1}u}\|_{L^2(\Omega)}
\|\langle y \rangle^{k+l+1}\partial_\tau^{\beta-\widetilde{\beta}}
\partial_y^{k}(\Theta \phi'')\|_{L^\infty(\Omega)}\\
&\le \|\partial_\tau^{\widetilde{\beta}+e_2} u\|_{L^2(\Omega)}
\|\partial_\tau^{\beta-\widetilde{\beta}} \Theta\|_{L^\infty(\mathbb{T}_x)}\\
&\le C\|u\|_{\mathcal{H}^m_0}
\|\partial_\tau^{\beta-\widetilde{\beta}} \Theta\|_{L^\infty(\mathbb{T}_x)}.
\end{aligned}
\end{equation}
If $\widetilde{k} \ge 1$, it is easy to deduce that
\begin{equation}\label{c220}
\begin{aligned}
&\|D^{\widetilde{\alpha}+e_2}\partial_y^{-1}u \cdot D^{\alpha-\widetilde{\alpha}}
(\Theta \phi'')\|_{L^2_{k+l}(\Omega)}\\
&\le \|\partial_\tau^{\widetilde{\beta}+e_2}
\partial_y^{\widetilde{k}-1}u\cdot
\partial_\tau^{\beta-\widetilde{\beta}}(\Theta \phi'')\|_{L^2_{k+l}(\Omega)}\\
&\le \|\langle y\rangle^{\widetilde{k}-1}\partial_\tau^{\widetilde{\beta}+e_2}
\partial_y^{\widetilde{k}-1}u\|_{L^2(\Omega)}
\|\langle y\rangle^{l+k-\widetilde{k}+1}\partial_\tau^{\beta-\widetilde{\beta}}
\partial_y^{k-\widetilde{k}}(\Theta \phi'')\|_{L^\infty(\Omega)}\\
&\le C\|u\|_{\mathcal{H}^m_0}
\|\partial_\tau^{\beta-\widetilde{\beta}} \Theta\|_{L^\infty(\mathbb{T}_x)}.
\end{aligned}
\end{equation}
Then, substituting the estimates \eqref{c219} and \eqref{c220}
into \eqref{c218}, we find
\begin{equation}\label{c221}
\|D^{\alpha}(\Theta \phi'' v)\|_{L^2_{l+k}(\Omega)}
\le C\|u\|_{\mathcal{H}^m_0}
\sum_{|{\beta}| \le m+1}\|\partial_\tau^{\beta}\Theta\|_{L^2(\mathbb{T}_x)}.
\end{equation}
\textbf{{Estimate  for $\|D^{\alpha}(\theta (u_y)^2)\|_{L^2_{l+k}(\Omega)}$}.}
Indeed, it is easy to check that
\begin{equation}\label{c222}
\begin{aligned}
&\|D^{\alpha}(\theta (u_y)^2)\|_{L^2_{l+k}(\Omega)}\\
&\le  \|\theta \partial_y u \cdot \partial_y D^{\alpha} u\|_{L^2_{l+k}(\Omega)}
+C\sum_{0<\widetilde{\alpha}\le \alpha}
\|\theta D^{\widetilde{\alpha}-E_i}D^{E_i}\partial_y u\cdot
D^{\alpha-\widetilde{\alpha}}\partial_y u\|_{L^2_{l+k}(\Omega)}\\
&\quad +C\sum_{0< \alpha \le  \alpha}
\|D^{\widetilde{\alpha}-E_i}D^{E_i} \theta \cdot
D^{\alpha-\widetilde{\alpha}}((\partial_y u)^2)\|_{L^2_{l+k}(\Omega)}.
\end{aligned}
\end{equation}
On one hand, we apply the Sobolev inequality to deduce that
\begin{equation}\label{c223}
\begin{aligned}
&\|\theta \partial_y u  D^{\alpha}\partial_y u\|_{L^2_{l+k}(\Omega)}\\
&\le \|D^{\alpha}\partial_y u\|_{L^2_{l+k}(\Omega)}
\|\theta\|_{L^\infty(\Omega)}\|\partial_y u\|_{L^\infty(\Omega)}\\
&\le \|D^{\alpha}\partial_y u\|_{L^2_{l+k}(\Omega)}
\|\theta\|_{H^2(\Omega)}\|\partial_y u\|_{H^2(\Omega)}.
\end{aligned}
\end{equation}
On the other hand, the application of inequality \eqref{e3}
and Sobolev inequality yields directly
\begin{equation}\label{c224}
\begin{aligned}
&\|\theta D^{\widetilde{\alpha}-E_i}D^{E_i}\partial_y u\cdot
D^{\alpha-\widetilde{\alpha}}\partial_y u\|_{L^2_{l+k}(\Omega)}\\
&\le C\|\theta\|_{L^\infty}\|D^{E_i}\partial_y u\|_{\mathcal{H}^{m-1}_{0}}
\|\partial_y u\|_{\mathcal{H}^{m-1}_{l+1}}\\
&\le C\|\partial_y u\|_{\mathcal{H}^{m}_{0}}\|\theta\|_{H^2}
\|u\|_{\mathcal{H}^{m}_{l}},
\end{aligned}
\end{equation}
and
\begin{equation}\label{c225}
\begin{aligned}
&\|D^{\widetilde{\alpha}-E_i}D^{E_i} \theta \cdot
D^{\alpha-\widetilde{\alpha}}((\partial_y u)^2)\|_{L^2_{l+k}(\Omega)}\\
&\le C\|D^{E_i} \theta\|_{\mathcal{H}^{m-1}_0}
     \|(\partial_y u)^2\|_{\mathcal{H}^{m-1}_{l+1}}
\le C\|\theta\|_{\mathcal{H}^{m}_0}\|u\|_{\mathcal{H}^{m}_{l}}^2,
\end{aligned}
\end{equation}
provided that $m\ge 4$.
Substituting the estimates \eqref{c223}-\eqref{c225} into \eqref{c222},
we obtain for $m \ge 4$
\begin{equation}\label{c226}
\begin{aligned}
\|D^{\alpha}(\theta(\partial_y u)^2)\|_{L^2_{l+k}(\Omega)}
\le
\|D^{\alpha}\partial_y u\|_{L^2_{l+k}(\Omega)}\|(u, \theta)\|_{\mathcal{H}^{m}_{l}}^2
+\|\partial_y u\|_{\mathcal{H}^{m}_{0}}\|(u, \theta)\|_{\mathcal{H}^{m}_{l}}^2
+\|(u, \theta)\|_{\mathcal{H}^{m}_l}^3.
\end{aligned}
\end{equation}
\textbf{{Estimate  for $\|D^{\alpha}( (U\phi'')^2 \theta)\|_{L^2_{l+k}(\Omega)}$}.}
By virtue of the definition of \eqref{function1},
it is easy to deduce that
\begin{equation}\label{c227}
\begin{aligned}
&\|D^{\alpha}( (U\phi'')^2 \theta)\|_{L^2_{l+k}(\Omega)}
\le C\sum_{\widetilde{\alpha}\le \alpha}
\|D^{\widetilde{\alpha}}\theta
D^{\alpha-\widetilde{\alpha}} [(U)^2(\phi'')^2 ]\|_{L^2_{l+k}(\Omega)}\\
&\le C\sum_{\widetilde{\alpha}\le \alpha}
\|\langle y \rangle^{\widetilde{k}}D^{\widetilde{\alpha}}\theta\|_{L^2(\Omega)}
\|\langle y \rangle^{k-\widetilde{k}+l}
D^{\alpha-\widetilde{\alpha}} [(U)^2(\phi'')^2 ]\|_{L^\infty(\Omega)}\\
&\le C\|\theta\|_{\mathcal{H}^{m}_0}
\sum_{|{\beta}|\le m+1}
\|\partial_\tau^{\beta}U\|_{L^2(\mathbb{T}_x)}^2.
\end{aligned}
\end{equation}
Similarly, we  can find that
\begin{equation}\label{c228}
\begin{aligned}
&\|D^{\alpha}(\Theta U \phi' \phi'' u_y)\|_{L^2_{l+k}(\Omega)}
\le C\|\partial_y u\|_{\mathcal{H}^{m}_0}
\sum_{|{\beta}|\le m+1}
\|\partial_\tau^{\beta}(U,\Theta)\|_{L^2(\mathbb{T}_x)}^2,\\
&\|D^{\alpha}(U\phi'' u_y)\|_{L^2_{l+k}(\Omega)}
\le C \|\partial_y u\|_{\mathcal{H}^{m}_0}
\sum_{|{\beta}|\le m+1}
\|\partial_\tau^{\beta}U\|_{L^2(\mathbb{T}_x)},\\
&\|D^{\alpha}(U \phi'' \theta \partial_y u)\|_{L^2_{l+k}(\Omega)}
\le C\|\partial_y u\|_{\mathcal{H}^{m}_0}\|\theta\|_{\mathcal{H}^{m}_0}
\sum_{|\beta|\le m+1}\|\partial_\tau^{\beta}U\|_{L^2(\mathbb{T}_x)}.
\end{aligned}
\end{equation}
\textbf{{Estimate  for $\|D^{\alpha}(H \phi' h_y)\|_{L^2_{l+k}(\Omega)}$}.}
We apply the inequality \eqref{e3} to deduce that
\begin{equation}\label{c229}
\begin{aligned}
&\|D^{\alpha}(H \phi' h_y)\|_{L^2_{l+k}(\Omega)}\\
&\le
C\sum_{0<\widetilde{\alpha}\le \alpha}
\|\langle y\rangle^{\widetilde{k}}D^{\widetilde{\alpha}}(H\phi')\|_{L^\infty(\Omega)}
\|\langle y\rangle^{k-\widetilde{k}+l}D^{\alpha-\widetilde{\alpha}}h_y\|_{L^2(\Omega)}\\
&\quad +\|H \phi'\|_{L^\infty(\Omega)}\|D^\alpha h_y\|_{L^2_{l+k}(\Omega)}\\
&\le \|H\|_{L^\infty(\mathbb{T}_x)}\|D^\alpha h_y\|_{L^2_{l+k}(\Omega)}
+\|h\|_{\mathcal{H}^m_l}
\sum_{|\beta|\le m+1}\|\partial_\tau^\beta H\|_{L^2(\mathbb{T}_x)}.
\end{aligned}
\end{equation}
\textbf{{Estimate  for $\|D^{\alpha}[(\partial_y u)^2]\|_{L^2_{l+k}(\Omega)}$}.}
By virtue of the inequality \eqref{e3} and the Sobolev inequality, we find
\begin{equation}\label{c230}
\begin{aligned}
&\|D^{\alpha}[(\partial_y u)^2]\|_{L^2_{l+k}(\Omega)}\\
&\le  \|\partial_y u~D^{\alpha}\partial_y u \|_{L^2_{l+k}(\Omega)}
+C\sum_{0<\widetilde{\alpha} \le \alpha}
\|D^{\widetilde{\alpha}-E_i} D^{E_i}\partial_y u ~
D^{\alpha-\widetilde{\alpha}}\partial_y u\|_{L^2_{l+k}(\Omega)}\\
&\le  \|\partial_y u~\|_{L^\infty(\Omega)}
\|D^{\alpha}\partial_y u \|_{L^2_{l+k}(\Omega)}
+C\|D^{E_i}\partial_y u\|_{\mathcal{H}^{m-1}_0}
\|\partial_y u\|_{\mathcal{H}^{m-1}_{l+1}}\\
&\le \|\partial_y u~\|_{H^2(\Omega)}
\|D^{\alpha}\partial_y u \|_{L^2_{l+k}(\Omega)}
+C\|\partial_y u\|_{\mathcal{H}^{m}_0}\|u\|_{\mathcal{H}^{m}_{l}},
\end{aligned}
\end{equation}
provided that $m \ge 4$.
Similarly, it is easy to deduce that
\begin{equation}\label{c231}
\|D^{\alpha}[(\partial_y h)^2]\|_{L^2_{l+k}(\Omega)}
\le C\|\partial_y h\|_{H^2(\Omega)}
\|D^{\alpha}\partial_y h \|_{L^2_{l+k}(\Omega)}
+C\|\partial_y h\|_{\mathcal{H}^{m}_0}\|h\|_{\mathcal{H}^{m}_{l}}.
\end{equation}
\textbf{{Estimate  for $\|D^{\alpha}(\Theta \phi'(u_y)^2)\|_{L^2_{l+k}(\Omega)}$}.}
In view of the inequality \eqref{e3}, it is easy to deduce
\begin{equation}\label{c232}
\begin{aligned}
&\|D^{\alpha}(\Theta \phi'(u_y)^2)\|_{L^2_{l+k}(\Omega)}\\
&\le
C\sum_{0\le \widetilde{\alpha} < \alpha}
\|\langle y \rangle^{\widetilde{k}+l}D^{\widetilde{\alpha}} (u_y^2)\|_{L^2(\Omega)}
\|\langle y \rangle^{k-\widetilde{k}}D^{\alpha-\widetilde{\alpha}}
(\Theta \phi')\|_{L^\infty(\Omega)}\\
&\quad +\|\Theta \phi'\|_{L^\infty(\Omega)}
\|\partial^\alpha (u_y^2)\|_{L^2_{k+l}(\Omega)}\\
&\le  C\|\Theta\|_{H^1(\mathbb{T}_x)}(\|\partial_y u~\|_{H^2(\Omega)}
     \|D^{\alpha}\partial_y u \|_{L^2_{l+k}(\Omega)}+
     \|\partial_y u\|_{\mathcal{H}^{m}_0}\|u\|_{\mathcal{H}^{m}_{l}})\\
&\quad +C\|\partial_y u\|_{\mathcal{H}^m_0}
     \|u\|_{\mathcal{H}^m_l}
     \sum_{|\beta| \le m+1}
     \|\partial_\tau^{{\beta}}\Theta\|_{L^2(\mathbb{T}_x)}.
\end{aligned}
\end{equation}
Combining the estimates \eqref{c217}, \eqref{c221},
\eqref{c226}-\eqref{c232} and applying the Cauchy inequality, we find
\begin{equation*}
\begin{aligned}
&-\int_\Omega
D^\alpha I_{23} \cdot\langle y \rangle^{2k+2l}D^\alpha \theta dxdy\\
&\le \frac{\mu}{4}\|D^\alpha \partial_y u(t)\|_{L^2_{l+k}(\Omega)}
+\frac{\nu}{4}\|D^\alpha \partial_y h(t)\|_{L^2_{l+k}(\Omega)}
+\delta_1 \mu\|\partial_y u(t)\|_{\mathcal{H}^m_0}^2\\
&\quad +\delta_1 \nu\|\partial_y h(t)\|_{\mathcal{H}^m_0}^2
+\delta_1^{-1}(\|(u, \theta, h)(t)\|_{\mathcal{H}^m_l}^8+1)
+\sum_{|\beta| \le m+2}\|\partial_\tau^\beta(U, \Theta, H)(t)\|_{L^2(\mathbb{T}_x)}^8,
\end{aligned}
\end{equation*}
which, together with the estimates \eqref{c23} and \eqref{c216}, yields
\begin{equation}\label{c233}
\begin{aligned}
&-\int_\Omega
D^\alpha I_2 \cdot\langle y \rangle^{2k+2l}D^\alpha \theta dxdy\\
&\le \frac{\mu}{4}\|D^\alpha \partial_y u(t)\|_{L^2_{l+k}(\Omega)}
+\frac{\nu}{4}\|D^\alpha \partial_y h(t)\|_{L^2_{l+k}(\Omega)}
+\delta_1 \mu\|\partial_y u(t)\|_{\mathcal{H}^m_0}^2\\
&\quad +\delta_1 \nu\|\partial_y h(t)\|_{\mathcal{H}^m_0}^2
+\delta_1^{-1}(\|(u, \theta, h)(t)\|_{\mathcal{H}^m_l}^8+1)
+\sum_{|\beta| \le m+2}\|\partial_\tau^\beta(U, \Theta, H)(t)\|_{L^2(\mathbb{T}_x)}^8.
\end{aligned}
\end{equation}
Similarly(or see \cite{Liu-Xie-Yang}), it is easy to deduce that
\begin{equation}\label{c234}
\begin{aligned}
&-\int_\Omega
D^\alpha I_1 \cdot\langle y \rangle^{2k+2l}D^\alpha u dxdy
-\int_\Omega
D^\alpha I_3 \cdot\langle y \rangle^{2k+2l}D^\alpha h dxdy\\
&\le \frac{\kappa}{4}\|D^\alpha \partial_y \theta(t)\|_{L^2_{k+l}(\Omega)}^2
+C(\|(u, \theta, h)(t)\|_{\mathcal{H}^m_l}^4+1)
+\sum_{|\beta| \le m+2}\|\partial_\tau^\beta(U, \Theta, H)(t)\|_{L^2(\mathbb{T}_x)}^4.
\end{aligned}
\end{equation}
Therefore, the combination of \eqref{c233} and \eqref{c234} completes the proof of \eqref{337}.
\end{proof}

\subsection{Weighted $H_l^m-$estimates only in tangential variable.}

\quad Similar to the classical Prandtl equations, an essential difficulty
for solving the problem \eqref{eq4} arises from the loss of one derivative
in the tangential variable $x$ in the terms
$v\partial_y u-g\partial_y h, v\partial_y h-g\partial_y u$ and
$v\partial_y \theta$.
More precisely, we recall the following nonlinear MHD boundary layer equations
\begin{equation}\label{eq5}
\left\{
\begin{aligned}
&\partial_t u+[(u+U\phi')\partial_x+(v-U_x \phi)\partial_y]u
-[(h+H\phi')\partial_x+(g-H_x \phi)\partial_y]h\\
&~~-\mu{\partial_y}[(\theta+\Theta \phi'(y)+1){\partial_y}u]
+U \phi'' v-H \phi'' g+II_1=r_1,\\
&c_v\{\partial_t \theta\!+\!(u+U\phi')\partial_x \theta\!+\!(v-U_x \phi)\partial_y \theta\}
\!-\!\kappa \partial_y^2 \theta\!+\!c_v \Theta \phi'' v+II_2=r_2,\\
&\partial_t h+[(u+U\phi')\partial_x+(v-U_x \phi)\partial_y]h
-[(h+H\phi')\partial_x+(g-H_x \phi)\partial_y]u
-\nu \partial_y^2 h\\
&\quad +H \phi'' v-U \phi'' g+II_3=r_3,\\
&\partial_x u+\partial_y v=0,\quad \partial_x h+\partial_y g=0,\\
\end{aligned}
\right.
\end{equation}
where
\begin{equation*}
\begin{aligned}
&II_1\triangleq U_x \phi' u-H_x \phi' h-U\phi^{(3)}\theta-U\phi''\theta_y,
\quad II_3\triangleq H_x \phi' u- U_x \phi' h,\\
&II_2
\triangleq c_v \Theta_x \phi' u-\mu \theta (u_y)^2-\mu (U\phi'')^2 \theta
-2\mu U \phi'' \theta u_y\!-\!\mu \Theta \phi'(u_y)^2\\
&\quad \quad ~ -2\mu \Theta U \phi' \phi'' u_y-2\mu U\phi'' u_y-\mu (u_y)^2
-\nu (h_y)^2-2\nu H \phi' h_y.
\end{aligned}
\end{equation*}
Then, applying $\beta-$th$(|\beta|=m)$ order tangential derivatives
on the equations \eqref{eq5}, we find
\begin{equation}\label{eq6}
\left\{
\begin{aligned}
&\partial_t \partial_\tau^{\beta}u
+[(u+U\phi')\partial_x+(v-U_x \phi)\partial_y]\partial_\tau^\beta u
-[(h+H\phi')\partial_x+(g-H_x \phi)\partial_y]\partial_\tau^\beta h\\
&\quad \!+\!(\partial_y u+U\phi'')\partial_\tau^\beta v
\!-\!(\partial_y h+H\phi'')\partial_\tau^\beta g
\!-\!\mu \partial_y[(\theta+\Theta \phi'+1)\partial_y \partial_\tau^\beta u]
=\partial_\tau^\beta(r_1-II_1)+R_u^\beta,\\
&c_v[\partial_t\!+\!(u+U\phi')\partial_x\!+\!(v-U_x \phi)\partial_y]
\partial_\tau^\beta \theta
\!-\!\kappa \partial_y^2 \partial_\tau^\beta \theta
\!+\!c_v(\partial_y \theta+\Theta \phi'')\partial_\tau^\beta v
=\partial_\tau^\beta(r_2\!-\!II_2)+R_{\theta}^\beta,\\
&\partial_t \partial_\tau^\beta h
+[(u+U\phi')\partial_x+(v-U_x \phi)\partial_y]\partial_\tau^\beta h
-[(h+H\phi')\partial_x+(g-H_x \phi)\partial_y]\partial_\tau^\beta u\\
&\quad+(\partial_y h+H\phi'')\partial_\tau^\beta v
-(\partial_y u+U\phi'')\partial_\tau^\beta g
-\nu \partial_y^2 \partial_\tau^\beta h
=\partial_\tau^\beta(r_3-II_3)+R_h^\beta,
\end{aligned}
\right.
\end{equation}
where
\begin{equation}\label{eq6a}
\begin{aligned}
R_u^\beta
\triangleq &-[\partial_\tau^\beta, (u+U\phi')\partial_x-U_x \phi \partial_y]u
+[\partial_\tau^\beta, (h+H\phi')\partial_x-H_x \phi\partial_y]h\\
&-[\partial_\tau^\beta, U\phi'']v
+[\partial_\tau^\beta, H\phi'']g
+\mu\partial_y([\partial_\tau^\beta, (\theta+\Theta \phi'+1)\partial_y]u)\\
&-\sum_{0<\widetilde{\beta} <\beta}
C^{\beta}_{\widetilde{\beta}}
\partial_\tau^{\beta}v \cdot \partial_\tau^{\beta-\widetilde{\beta}} \partial_y u
+\sum_{0<\widetilde{\beta} <\beta}
C^{\beta}_{\widetilde{\beta}}
\partial_\tau^{\beta}g \cdot \partial_\tau^{\beta-\widetilde{\beta}} \partial_y h,\\
R_\theta^\beta \triangleq
&-c_v[\partial_\tau^\beta, (u+U\phi')\partial_x-U_x \phi \partial_y]\theta
-c_v [\partial_\tau^\beta, \Theta \phi'']v
-\sum_{0<\widetilde{\beta} <\beta}
 c_v C^{\beta}_{\widetilde{\beta}}
\partial_\tau^{\widetilde{\beta}}v\cdot
\partial_\tau^{\beta-\widetilde{\beta}}\partial_y \theta,\\
R_h^\beta\triangleq
&-[\partial_\tau^\beta, (u+U\phi')\partial_x-U_x \phi \partial_y]h
+[\partial_\tau^\beta, (h+H\phi')\partial_x-H_x \phi \partial_y]u
-[\partial_\tau^\beta, H\phi'']v\\
&+[\partial_\tau^\beta, U\phi'']g
-\sum_{0<\widetilde{\beta} <\beta}
C^{\beta}_{\widetilde{\beta}}
\partial_\tau^{\widetilde{\beta}}v\cdot
\partial_\tau^{\beta-\widetilde{\beta}}\partial_y h
+\sum_{0<\widetilde{\beta} <\beta}
C^{\beta}_{\widetilde{\beta}}
\partial_\tau^{\widetilde{\beta}}g\cdot
\partial_\tau^{\beta-\widetilde{\beta}}\partial_y u.
\end{aligned}
\end{equation}
On the other hand, similar to \cite{Liu-Xie-Yang},
it is easy to verify that the function $\partial_y^{-1} h$ satisfies
\begin{equation}\label{eq7}
\partial_t(\partial_y^{-1}h)+(v-U_x \phi)(h+H\phi')
-(g-H_x \phi)(u+U\phi')-\nu \partial_y h
=-H_t \phi+\nu H \phi'',
\end{equation}
or equivalently
\begin{equation}\label{eq8}
\partial_t (\partial_y^{-1}h)+(h+H\phi')v
+(u+U\phi')\partial_x(\partial_y^{-1}h)
-U_x \phi h+H_x \phi u-\nu \partial_y h
=H_t\phi(\phi'-1)+\nu H \phi''.
\end{equation}
In view of the divergence free condition $\partial_x h+\partial_y g=0$,
then there exists a stream function $\psi$ such that
\begin{equation}\label{eq9}
h=\partial_y \psi,\quad g=-\partial_x \psi,\quad \left. \psi\right|_{y=0}=0.
\end{equation}
Then, the combination of \eqref{eq8} and \eqref{eq9} implies that
the function $\psi$ satisfies the following equation
\begin{equation}\label{eq10}
\begin{aligned}
\partial_t \psi
+[(u+U\phi')\partial_x+(v-U_x \phi)\partial_y]\psi
-\nu \partial_y^2 \psi+H\phi' v+H_x \phi u=r_4,
\end{aligned}
\end{equation}
where
\begin{equation}\label{function5}
r_4 \triangleq H_t\phi(1-\phi')+\nu H \phi''.
\end{equation}
Then applying $m-$th order tangential spatial derivative
to the equation \eqref{eq10}, we find
\begin{equation}\label{eq11}
\partial_t \partial_\tau^\beta \psi
+[(u+U\phi')\partial_x+(v-U_x \phi)\partial_y]\partial_\tau^\beta \psi
+(h+H\phi')\partial_\tau^\beta v-\nu \partial_y^2 \partial_\tau^\beta \psi
=\partial_\tau^\beta r_4+R_4^\beta,
\end{equation}
where
\begin{equation}
\begin{aligned}
R_4^\beta\triangleq
&\!-\partial_\tau^\beta(H_x \phi u)
\!-[\partial_\tau^\beta, H\phi']v
\!-[\partial_\tau^\beta,(u+U\phi')\partial_x\!-U_x \phi \partial_y]\psi
\!-\!\!\!\sum_{0<\widetilde{\beta} <\beta}
C^{\beta}_{\widetilde{\beta}}
\partial_\tau^{\widetilde{\beta}}v\cdot
\partial_\tau^{\beta-\widetilde{\beta}}\partial_y \psi.
\end{aligned}
\end{equation}
Let us define the functions
\begin{equation}\label{function3}
u_\beta\triangleq \partial_\tau^\beta u-\eta_1 \partial_\tau^\beta \psi,
\quad
\theta_\beta\triangleq \partial_\tau^\beta \theta-\eta_2 \partial_\tau^\beta \psi,
\quad
h_\beta\triangleq \partial_\tau^\beta h-\eta_3 \partial_\tau^\beta \psi,
\end{equation}
where
\begin{equation*}
\eta_1\triangleq\frac{\partial_y u+U\phi''}{h+H\phi'},\quad
\eta_2\triangleq\frac{\partial_y \theta+\Theta\phi''}{h+H\phi'},\quad
\eta_3\triangleq\frac{\partial_y h+H\phi''}{h+H\phi'}.
\end{equation*}
Then, we can obtain the following estimates:
\begin{equation}\label{equ1}
M(t)^{-1}\|\partial_\tau^\beta(u, \theta, h)(t)\|_{L^2_l(\Omega)}
\le \|(u_\beta, \theta_\beta, h_\beta)(t)\|_{L^2_l(\Omega)}
\le M(t)\|\partial_\tau^\beta(u, \theta, h)(t)\|_{L^2_l(\Omega)}
\end{equation}
and
\begin{equation}\label{equ2}
\|\partial_y \partial_\tau^\beta (u, \theta, h)(t)\|_{L^2_l(\Omega)}
\le \|\partial_y(u_\beta, \theta_\beta, h_\beta)(t)\|_{L^2_l(\Omega)}
+M(t)\|h_\beta(t)\|_{L^2_l(\Omega)},
\end{equation}
where
\begin{equation}\label{mt}
\begin{aligned}
M(t)\triangleq
&~2\delta_0^{-1}(C\|(U, \Theta, H)(t)\|_{L^\infty(\mathbb{T}_x)}
+\|\langle y\rangle^{l+1}\partial_y (u, \theta, h)(t)\|_{L^\infty(\Omega)})\\
&~+2\delta_0^{-1}
(\|\langle y\rangle^{l+1}\partial_y^2 (u, \theta, h)(t)\|_{L^\infty(\Omega)}+1).
\end{aligned}
\end{equation}
The detail of proof for the estimates \eqref{equ1} and \eqref{equ2}
can be found in Appendix \ref{appendixB}.
On the other hand, we know from the assumption \eqref{p1} that
\begin{equation}\label{a341}
\|\langle y \rangle^{l+1}\partial_y^i (u, \theta, h)(t)\|_{L^\infty(\Omega)}
\le C\delta_0^{-1},\quad {\rm for}~i=1,2,\quad t\in [0, T].
\end{equation}
Then, one can get, for $\delta_0$ sufficiently small, that
\begin{equation}\label{a342}
M(t)\le 2 \delta_0^{-1}(C\|(U, \Theta, H)(t)\|_{L^\infty(T_x)}+2\delta_0^{-1}+1)
\le 6\delta_0^{-2}.
\end{equation}
Therefore, we deduce from the definition of functions \eqref{function3},
and equations \eqref{eq6} and \eqref{eq11} that
\begin{equation}\label{eq12}
\left\{
\begin{aligned}
&\partial_t u_\beta
+[(u+U\phi')\partial_x+(v-U_x \phi)\partial_y]u_\beta
-[(h+H\phi')\partial_x+(g-H_x \phi)\partial_y]h_\beta
+\nu \eta_1 \partial_y h_\beta\\
&\quad
-\mu\partial_y[(\theta+\Theta \phi'+1)\partial_y u_\beta]
-\mu \partial_y[(\theta+\Theta \phi'+1)(\partial_y \eta_1 \partial_\tau^\beta \psi
+\eta_1 \partial_\tau^\beta h)]=R_1^\beta,\\
&c_v[\partial_t\!+(u\!+U\phi')\partial_x\!+(v\!-U_x \phi)\partial_y]\theta_\beta
\!-\kappa \partial_y^2 \theta_\beta
\!-\kappa \partial_y(\partial_y \eta_2 \partial_\tau^\beta \psi
\!+\eta_2 \partial_\tau^\beta h)
+c_v \nu \eta_2 \partial_y h_\beta=R_2^\beta,\\
&\partial_t h_\beta
+[(u+U\phi')\partial_x+(v-U_x \phi)\partial_y]h_\beta
-[(h+H\phi')\partial_x+(g-H_x \phi)\partial_y]u_\beta
-\nu \partial_y^2 h_\beta=R_3^\beta,
\end{aligned}
\right.
\end{equation}
where
\begin{equation}\label{eq13}
\left\{
\begin{aligned}
R_1^\beta
&\triangleq
\partial_\tau^\beta(r_1-II_1)+R_u^\beta
-\eta_1 \partial_\tau^\beta r_4-\eta_1 R_4^\beta
-\nu \eta_1 \eta_3 \partial_\tau^\beta h
+\eta_3(g-H_x \phi)\partial_\tau^\beta h
-\zeta_1(\partial_\tau^\beta \psi),\\
R_2^\beta
&\triangleq
\partial_\tau^\beta(r_2-II_2)+R_\theta^\beta
-c_v \eta_2 \partial_\tau^\beta r_4-c_v \eta_2 R_4^\beta
-c_v \nu \eta_2 \eta_3 \partial_\tau^\beta h
-c_v \zeta_2 (\partial_\tau^\beta \psi),\\
R_3^\beta
&\triangleq
\partial_\tau^\beta(r_3-II_3)
+R_h^\beta-\eta_3 \partial_\tau^\beta r_4-\eta_3 R_4^\beta
+[2\nu \partial_y \eta_3+(g-H_x \phi)\eta_1]\partial_\tau^\beta h
-\zeta_3 (\partial_\tau^\beta \psi),
\end{aligned}
\right.
\end{equation}
with
\begin{equation}\label{eq14}
\left\{
\begin{aligned}
&\zeta_1\triangleq
\partial_t \eta_1
+[(u+U\phi')\partial_x+(v-U_x \phi)\partial_y]\eta_1
-[(h+H\phi')\partial_x+(g-H_x \phi)\partial_y]\eta_3
+\nu \eta_1 \partial_y \eta_3,\\
&\zeta_2\triangleq
\partial_t \eta_2
+[(u+U\phi')\partial_x+(v-U_x \phi)\partial_y]\eta_2
+\nu \eta_2 \partial_y \eta_3,\\
&\zeta_3\triangleq
\partial_t \eta_3
+[(u+U\phi')\partial_x+(v-U_x \phi)\partial_y]\eta_3
-[(h+H\phi')\partial_x+(g-H_x \phi)\partial_y]\eta_1
-\nu\partial_y^2 \eta_3.
\end{aligned}
\right.
\end{equation}
Also, we have the corresponding initial and boundary conditions as follows:
\begin{equation}\label{eq15}
\left\{
\begin{aligned}
&u_\beta|_{t=0}
=\partial_\tau^\beta u(0,x,y)
-\frac{\partial_y u_0(x,y)+U(0,x)\phi''(y)}{h_0(x,y)+H(0,x)\phi'(y)}
\int_0^y \partial_\tau^\beta  h(0,x,z)dz\triangleq u_{\beta 0}(x, y),\\
&\theta_\beta|_{t=0}
=\partial_\tau^\beta \theta(0,x,y)
-\frac{\partial_y \theta_0(x,y)+\Theta(0,x)\phi''(y)}{h_0(x,y)+H(0,x)\phi'(y)}
\int_0^y \partial_\tau^\beta  h(0,x,z)dz
\triangleq \theta_{\beta 0}(x, y),\\
&h_\beta|_{t=0}
=\partial_\tau^\beta h(0,x,y)
-\frac{\partial_y h_0(x,y)+H(0,x)\phi''(y)}{h_0(x,y)+H(0,x)\phi'(y)}
\int_0^y \partial_\tau^\beta  h(0,x,z)dz\triangleq h_{\beta 0}(x, y),\\
&u_\beta|_{y=0}=\theta_\beta|_{y=0}=\partial_y h_\beta|_{y=0}=0.
\end{aligned}
\right.
\end{equation}
Moreover, by combining $\psi=\partial_y^{-1}h$ with the inequality \eqref{e3},
it is easy to check that
\begin{equation}\label{estimate3}
\|\langle y\rangle^{-1}\partial_\tau^\beta \psi(t)\|_{L^2(\Omega)}
\le 2 \|\partial_\tau^\beta h(t)\|_{L^2(\Omega)}.
\end{equation}
By virtue of Sobolev embedding inequality and direct computation, we have
for any $\lambda \in \mathbb{R}$ and $i=1,2,3$,
\begin{equation}\label{estimate4}
\begin{aligned}
&\|\langle y\rangle^\lambda \eta_i\|_{L^\infty(\Omega)}
\le C\delta_0^{-1}(\|(U, \Theta, H)(t)\|_{L^\infty(\mathbb{T}_x)}
+\|(u, \theta, h)(t)\|_{\mathcal{H}^{3}_{\lambda-1}}),\\
&\|\langle y\rangle^\lambda \partial_y \eta_i\|_{L^\infty(\Omega)}
\le C\delta_0^{-2}(\|(U, \Theta, H)(t)\|_{L^\infty(\mathbb{T}_x)}
+\|(u, \theta, h)(t)\|_{\mathcal{H}^{4}_{\lambda-1}})^2,\\
&\|\langle y\rangle^\lambda \zeta_i\|_{L^\infty(\Omega)}
\le C\delta_0^{-3}(\sum_{|\beta|\le 1}\|\partial_\tau^\beta
(U, \Theta, H)(t)\|_{L^\infty(\mathbb{T}_x)}
+\|(u, \theta, h)(t)\|_{\mathcal{H}^{5}_{\lambda-1}})^3.
\end{aligned}
\end{equation}

Now, we are going to establish the $L^2_l-$norms for the
quantity $(u_\beta, \theta_\beta, h_\beta)$.

\begin{lemma}\label{lemma3.4}[$L^2_l$-estimate on $(u_\beta, \theta_\beta, h_\beta)$]
Under the hypotheses of Proposition \ref{pro1}, we have for any $t\in [0, T]$
and the quantity $(u_\beta, \theta_\beta, h_\beta)$ given in \eqref{function3}  that
\begin{equation}\label{341}
\begin{aligned}
&\sum_{|\beta|=m}\left\{\frac{d}{dt}\left(
\|(u_\beta, h_\beta)(t)\|_{L^2_l(\Omega)}^2
+c_v\|\theta_\beta(t)\|_{L^2_l(\Omega)}^2\right)
+c_0\|\partial_y(u_\beta,\theta_\beta, h_\beta)(t)\|_{L^2_l(\Omega)}^2\right\}\\
&\le C\delta_0^{-4}(1+
\sum_{|\beta|\le m+2}\|\partial_\tau^\beta(U, \Theta, H)(t)\|_{L^2(\mathbb{T}_x)}^{10})
+C\delta_0^{-4}\|(u_\beta, \theta_\beta, h_\beta)(t)\|_{L^2_l(\Omega)}^4\\
&\quad +C\delta_0^{-4}\|(u, \theta, h)(t)\|_{\mathcal{H}^m_l}^{12}
+\sum_{|\beta|=m}
\left\{\sum_{i=1}^3\|\partial_\tau^\beta r_i\|_{L^2_l(\Omega)}^2
+\|\eta_i \partial_\tau^\beta r_4\|_{L^2_l(\Omega)}^2\right\},
\end{aligned}
\end{equation}
where the quantity $c_0$ is defined in Lemma \ref{lemma3.3}.
\end{lemma}
\begin{proof}
Multiplying the equation \eqref{eq12}$_1$ by $\langle y \rangle^{2l}u_\beta$,
integrating over $\Omega$ and integrating by part, we find
\begin{equation}\label{342}
\begin{aligned}
&\frac{1}{2}\frac{d}{dt}\int_{\Omega} \langle y \rangle^{2l}|u_\beta|^2dxdy
+\mu \int_{\Omega} (\theta+\Theta \phi'+1) \langle y \rangle^{2l}|\partial_y u_\beta|^2dxdy\\
&=-\int_{\Omega} [(h+H\phi')\partial_x +(g-H_x \phi)\partial_y]
u_\beta  \cdot \langle y \rangle^{2l} h_\beta dxdy
+l \int_{\Omega} (v-U_x \phi)\cdot \langle y \rangle^{2l-1}|u_\beta|^2 dxdy\\
&\quad -\nu \int_{\Omega} \eta_1 \partial_y h_\beta \cdot\langle y \rangle^{2l} u_\beta dxdy
-2l\int_{\Omega} (g-H_x \phi)\cdot\langle y \rangle^{2l-1}u_\beta h_\beta dxdy\\
&\quad -\mu \int_{\Omega} (\theta+\Theta \phi'+1)
(\partial_y \eta_1 \partial^\beta_\tau \psi+\eta_1 \partial_\tau^\beta h)
\cdot\langle y \rangle^{2l} \partial_y u_\beta dxdy\\
&\quad-2l \mu \int_{\Omega} (\theta+\Theta \phi'+1)
(\partial_y \eta_1 \partial_\tau^\beta \psi+\eta_1 \partial_\tau^\beta h)
\cdot\langle y \rangle^{2l-1}u_\beta dxdy\\
&\quad -2l \mu \int_{\Omega} (\theta+\Theta \phi'+1)\partial_y u_\beta
\cdot\langle y \rangle^{2l-1}u_\beta dxdy
+\int_{\Omega} R_1^\beta \cdot \langle y \cdot\rangle^{2l}u_\beta dxdy.
\end{aligned}
\end{equation}
In view of the inequality \eqref{e5}, Sobolev and Cauchy inequalities,
it is easy to deduce that
\begin{equation}\label{343}
\begin{aligned}
&\left|l \int_{\Omega} (v-U_x \phi)\langle y \rangle^{2l-1}|u_\beta|^2 dxdy\right|\\
&\le C(\|\langle y\rangle^{-1}\partial_y^{-1}u_x\|_{L^\infty(\Omega)}
+\|\langle y\rangle^{-1}(U_x \phi)\|_{L^\infty(\Omega)})\|u_\beta\|_{L^2_l(\Omega)}^2\\
&\le  C(\|u_x\|_{L^\infty(\Omega)}
+\|U_x \|_{L^\infty(\mathbb{T}_x)})\|u_\beta\|_{L^2_l(\Omega)}^2\\
&\le C(1+\|U_x \|_{H^1(\mathbb{T}_x)}^2)
+C(\|u\|_{\mathcal{H}^{m}_l}^2+\|u_\beta\|_{L^2_l(\Omega)}^4).
\end{aligned}
\end{equation}
Similarly, we also find that
\begin{equation}\label{344}
\left|-2l\int_\Omega(g-H_x \phi)\langle y \rangle^{2l-1}u_\beta h_\beta dxdy\right|
\le C(1+\|H_x \|_{H^1(\mathbb{T}_x)}^2)
+C(\|h\|_{\mathcal{H}^{m}_l}^2+\|(u_\beta,h_\beta)\|_{L^2_l(\Omega)}^4).
\end{equation}
By virtue of the Cauchy inequality and the estimate \eqref{estimate4},
one arrives at directly
\begin{equation}\label{345}
\begin{aligned}
&|-\nu \int_\Omega \eta_1 \partial_y h_\beta \langle y \rangle^{2l} u_\beta dxdy|\\
&\le \frac{\nu}{12} \|\partial_y h_\beta\|_{L^2_l(\Omega)}^2
+C\|\eta_1\|_{L^\infty(\Omega)}^2\|u_\beta\|_{L^2_l(\Omega)}^2\\
&\le \frac{\nu}{12} \|\partial_y h_\beta\|_{L^2_l(\Omega)}^2
+C\delta_0^{-2}
(\|(U, H)\|_{H^1(\mathbb{T}_x)}^2+\|(u, h)\|_{\mathcal{H}^3_{0}}^2)
\|u_\beta\|_{L^2_l(\Omega)}^2\\
&\le \frac{\nu}{12} \|\partial_y h_\beta\|_{L^2_l(\Omega)}^2
+C\delta_0^{-2}\|(U, H)\|_{H^1(\mathbb{T}_x)}^4
+C\delta_0^{-2}(\|(u, h)\|_{\mathcal{H}^m_{l}}^4+\|u_\beta\|_{L^2_l(\Omega)}^4).
\end{aligned}
\end{equation}
With the help of estimate \eqref{estimate4}, H\"{o}lder and Cauchy inequalities,
it is easy to check that
\begin{equation}\label{346}
\begin{aligned}
&\left|-\mu \int_\Omega (\theta+\Theta \phi'+1)
(\partial_y \eta_1 \partial^\beta_\tau \psi+\eta_1 \partial_\tau^\beta h)
\langle y \rangle^{2l} \partial_y u_\beta dxdy\right|\\
&\le \mu \|\partial_y u_\beta\|_{L^2_l(\Omega)}
      \|\theta+\Theta \phi'+1\|_{L^\infty(\Omega)}
      \|\langle y\rangle^{l+1}\partial_y \eta_1 \|_{L^\infty(\Omega)}
      \|\langle y\rangle^{-1}\partial_\tau^\beta \psi\|_{L^2(\Omega)}\\
&\quad +\mu \|\partial_y u_\beta\|_{L^2_l(\Omega)}
      \|\theta+\Theta \phi'+1\|_{L^\infty(\Omega)}
      \|\eta_1 \|_{L^\infty(\Omega)}
      \|\partial_\tau^\beta h\|_{L^2_l(\Omega)}\\
&\le   C\delta_0^{-4}(\|\theta\|_{H^2(\Omega)}
      +\|\Theta\|_{H^1(\mathbb{T}_x)}+1)^2
      (\|(U, \Theta, H)\|_{H^1(\mathbb{T}_x)}
      +\|(u, \theta, h)\|_{\mathcal{H}^4_l})^4
      \|\partial_\tau^\beta h\|_{L^2(\Omega)}^2\\
&\quad  +C\delta_0^{-2}(\|\theta\|_{H^2(\Omega)}
      +\|\Theta\|_{H^1(\mathbb{T}_x)}+1)^2
      (\|(U, \Theta, H)\|_{H^1(\mathbb{T}_x)}
      +\|(u, \theta, h)\|_{\mathcal{H}^3_l})^2
      \|\partial_\tau^\beta h\|_{L^2(\Omega)}^2\\
&\quad    +\frac{\mu}{12} \|\partial_y u_\beta\|_{L^2_l(\Omega)}^2\\
&\le \frac{\mu}{12} \|\partial_y u_\beta\|_{L^2_l(\Omega)}^2
     +C\delta_0^{-4}(1+\|(U, \Theta, H)\|_{H^1(\mathbb{T}_x)}^8)
     +C\delta_0^{-4}\|(u, \theta, h)\|_{\mathcal{H}_l^m}^{12}.
\end{aligned}
\end{equation}
Similarly, we can obtain the following estimates
\begin{equation}\label{347}
\begin{aligned}
&\left|-2l \mu \int_\Omega (\theta+\Theta \phi'+1)
(\partial_y \eta_1 \partial_\tau^\beta \psi+\eta_1 \partial_\tau^\beta h)
\langle y \rangle^{2l-1}u_\beta dxdy\right|\\
&\le C(\|\theta\|_{L^\infty(\Omega)}+\|\Theta\|_{L^\infty(\mathbb{T}_x)}+1)
     \|\langle y\rangle^l \partial_y \eta_1\|_{L^\infty(\Omega)}
     \|\langle y\rangle^{-1}\partial_\tau^\beta \psi\|_{L^2(\Omega)}
     \|u_\beta\|_{L^2_l(\Omega)}\\
&\quad +
     C(\|\theta\|_{L^\infty(\Omega)}+\|\Theta\|_{L^\infty(\mathbb{T}_x)}+1)
     \|\eta_1\|_{L^\infty(\Omega)}\|\partial_\tau^\beta h\|_{L^2_l(\Omega)}
     \|u_\beta\|_{L^2_l(\Omega)}\\
&\le C\delta_0^{-2}(\|\theta\|_{H^2(\Omega)}\!+\!\|\Theta\|_{H^1(\mathbb{T}_x)}\!+\!1)
      (\|(U, \Theta, H)\|_{H^1(\mathbb{T}_x)}
      \!+\!\|(u, \theta, h)\|_{\mathcal{H}^4_l})^2
      \|\partial_\tau^\beta h\|_{L^2(\Omega)}
      \|u_\beta\|_{L^2_l(\Omega)}\\
&\quad +C\delta_0^{-1}(\|\theta\|_{H^2(\Omega)}\!+\!\|\Theta\|_{H^1(\mathbb{T}_x)}\!+\!1)
      (\|(U, \Theta, H)\|_{H^1(\mathbb{T}_x)}
      \!+\!\|(u, \theta, h)\|_{\mathcal{H}^3_l})
      \|\partial_\tau^\beta h\|_{L^2(\Omega)}
      \|u_\beta\|_{L^2_l(\Omega)}\\
&\le C\delta_0^{-2}(1+\|(U, \Theta, H)\|_{H^1(\mathbb{T}_x)}^{10})
     +C\delta_0^{-2}(\|(u, \theta, h)\|_{\mathcal{H}^m_l}^8
     +\|u_\beta\|_{L^2_l(\Omega)}^4),
\end{aligned}
\end{equation}
and
\begin{equation}\label{348}
\begin{aligned}
&\left|-2l \mu \int_\Omega (\theta+\Theta \phi'+1)\partial_y u_\beta
\cdot \langle y \rangle^{2l-1}u_\beta dxdy\right|\\
&\le C\mu \|\partial_y u_\beta\|_{L^2_l(\Omega)}
          \|\theta+\Theta \phi'+1\|_{L^\infty(\Omega)}
          \|u_\beta\|_{L^2_l(\Omega)}\\
&\le \frac{\mu}{12} \|\partial_y u_\beta\|_{L^2_l(\Omega)}^2
      +C(1+\|\Theta\|_{H^1(T_x)}^2)^2
      +C(\|\theta\|_{\mathcal{H}^m_l}^4+\|u_\beta\|_{L^2_l(\Omega)}^4).
\end{aligned}
\end{equation}
By virtue of the lower bound estimate for temperature \eqref{321}, we get
\begin{equation}\label{349}
\mu \int_{\Omega} (\theta+\Theta \phi'+1) \langle y \rangle^{2l}|\partial_y u_\beta|^2dxdy
\ge \mu \int_{\Omega} \langle y \rangle^{2l}|\partial_y u_\beta|^2dxdy.
\end{equation}
Plugging the estimates \eqref{343}-\eqref{349} into \eqref{342}, it is easy to deduce that
\begin{equation}\label{3410}
\begin{aligned}
&\frac{d}{dt}\int_\Omega \langle y \rangle^{2l}|u_\beta|^2 dxdy
+\mu \int_\Omega \langle y \rangle^{2l}|\partial_y u_\beta|^2dxdy\\
&\le-\int_{\Omega} [(h+H\phi')\partial_x +(g-H_x \phi)\partial_y]
u_\beta  \cdot \langle y \rangle^{2l} h_\beta dxdy\\
&\quad +\frac{\nu}{12} \|\partial_y h_\beta\|_{L^2_l(\Omega)}^2
+C\delta_0^{-4}(1+\|(U, \Theta, H)\|_{H^2(\mathbb{T}_x)}^{10})\\
&\quad +C\delta_0^{-4}(\|(u, \theta, h)\|_{\mathcal{H}^m_l}^{12}
+\|(u_\beta, h_\beta)\|_{L^2_l(\Omega)}^4)
+\int_\Omega R_1^\beta \cdot \langle y \rangle^{2l}u_\beta dxdy.
\end{aligned}
\end{equation}
Multiplying equation \eqref{eq12}$_2$ by $\langle y \rangle^{2l} h_\beta$,
integrating over $\Omega$ and integrating by part, we obtain
\begin{equation}\label{3411}
\begin{aligned}
&\frac{1}{2}\frac{d}{dt}\int_\Omega \langle y \rangle^{2l}|h_\beta|^2 dxdy
+\nu \int_\Omega \langle y \rangle^{2l} |\partial_y h_\beta|^2 dxdy\\
&=l\int_\Omega (v-U_x \phi)\langle y \rangle^{2l-1}|h_\beta|^2dxdy
-2l \nu \int_\Omega \langle y \rangle^{2l-1} \partial_y h_\beta\cdot h_\beta dxdy\\
&\quad +\int_\Omega[(h+H\phi')\partial_x+(g-H_x \phi)\partial_y]
u_\beta \cdot \langle y \rangle^{2l} h_\beta dxdy
+\int_\Omega R_3^\beta \cdot \langle y \rangle^{2l} h_\beta dxdy.
\end{aligned}
\end{equation}
Similar to the estimate \eqref{343}, we can obtain directly
\begin{equation}\label{3412}
\left|c_v l \int_\Omega (v-U_x \phi)\langle y \rangle^{2l-1}|h_\beta|^2dxdy\right|
\le C\|U_x\|_{H^1(\mathbb{T}_x)}^2
     +C(\|u\|_{\mathcal{H}^m_l}^2+\|h_\beta\|_{L^2_l(\Omega)}^4).
\end{equation}
In view of the Cauchy inequality, it is easy to deduce that
\begin{equation}\label{3413}
\left|2l \nu \int_\Omega \langle y \rangle^{2l-1} \partial_y h_\beta\cdot h_\beta dxdy\right|
\le \frac{\nu}{2} \|\partial_y h_\beta\|_{L^2_l(\Omega)}^2
+C\|h_\beta\|_{L^2_l(\Omega)}^2.
\end{equation}
Substituting \eqref{3412} and \eqref{3413} into \eqref{3411},
we find
\begin{equation*}
\begin{aligned}
&\frac{d}{dt}\int_\Omega \langle y \rangle^{2l}|h_\beta|^2dxdy
+\nu \int_\Omega \langle y \rangle^{2l} |\partial_y h_\beta|^2 dxdy\\
&\le \int_\Omega[(h+H\phi')\partial_x+(g-H_x \phi)\partial_y]
u_\beta \cdot \langle y \rangle^{2l} h_\beta dxdy
+\int_\Omega R_3^\beta \cdot \langle y \rangle^{2l} h_\beta dxdy \\
&\quad +C(1+\|U_x\|_{H^1(\mathbb{T}_x)}^2)
     +C(\|u\|_{\mathcal{H}^m_l}^2+\|h_\beta\|_{L^2_l(\Omega)}^4),
\end{aligned}
\end{equation*}
which, together with the inequality \eqref{3410}, implies that
\begin{equation}\label{3414}
\begin{aligned}
&\frac{d}{dt}\int_\Omega \langle y \rangle^{2l}(|u_\beta|^2+|h_\beta|^2) dxdy
+\mu \int_\Omega \langle y \rangle^{2l}|\partial_y u_\beta|^2dxdy
+\nu \int_\Omega \langle y \rangle^{2l} |\partial_y h_\beta|^2 dxdy\\
&\le C\delta_0^{-4}(1+\|(U, \Theta, H)(t)\|_{H^2(\mathbb{T}_x)}^{10})
+C\delta_0^{-4}(\|(u, \theta, h)(t)\|_{\mathcal{H}^m_l}^{12}
+\|(u_\beta, h_\beta)(t)\|_{L^2_l(\Omega)}^4)\\
&\quad
+\int_\Omega  R_1^\beta \cdot \langle y \rangle^{2l}u_\beta dxdy
+\int_\Omega  R_3^\beta \cdot \langle y \rangle^{2l} h_\beta dxdy.
\end{aligned}
\end{equation}
Multiplying equation \eqref{eq12}$_3$ by $\langle y \rangle^{2l} \theta_\beta$,
integrating over $\Omega$ and integrating by part, we find
\begin{equation}\label{3415}
\begin{aligned}
&\frac{c_v}{2}\frac{d}{dt}\int_{\Omega} \langle y \rangle^{2l}|\theta_\beta|^2dxdy
+\frac{\kappa}{2}\int_{\Omega} \langle y \rangle^{2l} |\partial_y \theta_\beta|^2dxdy\\
&=c_v l \int_{\Omega} (v-U_x \phi)\langle y \rangle^{2l-1}|\theta_\beta|^2dxdy
-2\kappa l \int_{\Omega} \langle y \rangle^{2l-1}\partial_y \theta_\beta \cdot
\theta_\beta dxdy\\
&\quad -2\kappa l\int_{\Omega} (\partial_y \eta_2 \partial_\tau^\beta \psi
+\eta_2 \partial_\tau^\beta h)\cdot \langle y \rangle^{2l-1}\theta_\beta dxdy
-c_v \nu \int_{\Omega} \eta_2 \partial_y h_\beta \cdot
\langle y \rangle^{2l} \theta_\beta dxdy\\
&\quad-\kappa \int_{\Omega} (\partial_y \eta_2 \partial_\tau^\beta \psi
+\eta_2 \partial_\tau^\beta h)\cdot\langle y \rangle^{2l} \theta_\beta dxdy
+\int_{\Omega} R_2^\beta \cdot \langle y \rangle^{2l} \theta^\beta dxdy.
\end{aligned}
\end{equation}
Similar to the estimates \eqref{343} and \eqref{345} , we can obtain directly
\begin{equation}\label{3416}
\begin{aligned}
&\left|c_v l \int (v-U_x \phi)\langle y \rangle^{2l-1}|\theta_\beta|^2dxdy\right|
\le C\|U_x\|_{H^1(\mathbb{T}_x)}^2
     +C(\|u_x\|_{H^2(\Omega)}^2+\|\theta_\beta\|_{L^2_l(\Omega)}^4),\\
&\left|c_v \nu \int \eta_2 \partial_y h_\beta \cdot \langle y \rangle^{2l}
\theta_\beta dxdy\right|
\le  \frac{\nu}{2} \|\partial_y h_\beta\|_{L^2_l(\Omega)}^2
\!+\!C\delta_0^{-4}(\|(U, H)\|_{H^1(\mathbb{T}_x)}^4
\!+\!\|(u, h)\|_{\mathcal{H}^m_l}^4
\!+\!\|\theta_\beta\|_{L^2_l(\Omega)}^4).
\end{aligned}
\end{equation}
With the help of Cauchy inequality, one arrives at immediately
\begin{equation}\label{3417}
\left|2\kappa l \int \langle y \rangle^{2l-1}\partial_y \theta_\beta \cdot
\theta_\beta dxdy\right|
\le \frac{\kappa}{2} \|\partial_y \theta_\beta\|_{L^2_l(\Omega)}^2
+C\|\theta_\beta\|_{L^2_l(\Omega)}^2.
\end{equation}
On the other hand, we apply the H\"{o}lder inequality
and estimate \eqref{estimate4} to get that
\begin{equation}\label{3418}
\begin{aligned}
&\left|-2\kappa l\int (\partial_y \eta_2 \partial_\tau^\beta \psi
+\eta_2 \partial_\tau^\beta h)\cdot \langle y \rangle^{2l-1}\theta_\beta dxdy\right|\\
&\le C\|\langle y\rangle^l \partial_y \eta_2\|_{L^\infty(\Omega)}
     \|\langle y\rangle^{-1}\partial_\tau^\beta \psi\|_{L^2(\Omega)}
     \|\langle y\rangle^{l}\theta_\beta\|_{L^2(\Omega)}\\
&\quad +C\|\eta_2\|_{L^\infty(\Omega)}\|\partial_\tau^\beta h\|_{L^2_l(\Omega)}
       \|\theta_\beta\|_{L^2_l(\Omega)}\\
&\le C\delta_0^{-2}(\|(U, \Theta, h)\|_{L^\infty(\mathbb{T}_x)}
+\|(u, \theta, h)\|_{\mathcal{H}^4_l})^2
\|\partial_\tau^\beta h\|_{L^2(\Omega)}
\|\theta_\beta\|_{L^2_l(\Omega)}\\
&\quad +C\delta_0^{-1}(\|(U, \Theta, h)\|_{L^\infty(\mathbb{T}_x)}
+\|(u, \theta, h)\|_{\mathcal{H}^3_l})
\|\partial_\tau^\beta h\|_{L^2(\Omega)}
\|\theta_\beta\|_{L^2_l(\Omega)}\\
&\le C\delta_0^{-2}(1+\|(U, \Theta, H)\|_{H^1(\mathbb{T}_x)}^8)
+C\delta_0^{-2}(\|(u, \theta, h)\|_{\mathcal{H}^m_l}^8
+\|\theta_\beta\|_{L^2_l(\Omega)}^2),
\end{aligned}
\end{equation}
and
\begin{equation}\label{3419}
\begin{aligned}
&\left|-\kappa \int (\partial_y \eta_2 \partial_\tau^\beta \psi
+\eta_2 \partial_\tau^\beta h)\cdot \langle y \rangle^{2l} \theta_\beta dxdy\right|\\
&\le \|\langle y\rangle^{l+1}\partial_y \eta_2\|_{L^\infty(\Omega)}
     \|\langle y\rangle^{-1}\partial_\tau^\beta \psi\|_{L^2(\Omega)}
     \|\theta_\beta\|_{L^2_l(\Omega)}\\
&\quad +\|\eta_2\|_{L^\infty(\Omega)}
       \|\partial_\tau^\beta h\|_{L^2_l(\Omega)}
       \|\theta_\beta\|_{L^2_l(\Omega)}\\
&\le C\delta_0^{-2}(\|(U,\Theta, H)\|_{L^\infty(\mathbb{T}_x)}
      +\|(u, \theta, h)\|_{\mathcal{H}^4_l})^2
      \|\partial_\tau^\beta h\|_{L^2(\Omega)}
      \|\theta_\beta\|_{L^2_l(\Omega)}\\
&\quad +C\delta_0^{-1}(\|(U,\Theta, H)\|_{L^\infty(\mathbb{T}_x)}
      +\|(u, \theta, h)\|_{\mathcal{H}^3_l})
      \|\partial_\tau^\beta h\|_{L^2(\Omega)}
      \|\theta_\beta\|_{L^2_l(\Omega)}\\
&\le C\delta_0^{-2}(1+\|(U, \Theta, H)\|_{H^1(\mathbb{T}_x)}^8)
     +C\delta_0^{-2}(\|(u, \theta, h)\|_{\mathcal{H}_l^m}^8
     +\|\theta_\beta\|_{L^2_l(\Omega)}^2).
\end{aligned}
\end{equation}
Substituting the estimates \eqref{3416}-\eqref{3419} into \eqref{3415},
we find immediately
\begin{equation}\label{3420}
\begin{aligned}
&\frac{d}{dt}\int_\Omega \langle y \rangle^{2l}|\theta_\beta|^2dxdy
+\kappa\int_\Omega \langle y \rangle^{2l} |\partial_y \theta_\beta|^2dxdy\\
&\le C\delta_0^{-4}(1+\|(U, \Theta, H)\|_{H^1(\mathbb{T}_x)}^8)
+C\delta_0^{-4}(\|(u, \theta, h)\|_{\mathcal{H}^m_l}^8
+\|\theta_\beta\|_{L^2_l(\Omega)}^4)\\
&\quad +\frac{\nu}{2}\|\partial_y h_\beta\|_{L^2_l(\Omega)}^2
+\int_\Omega R_2^\beta \cdot \langle y \rangle^{2l} \theta^\beta dxdy.
\end{aligned}
\end{equation}
The combination of \eqref{3414} and \eqref{3420} yields directly
\begin{equation}\label{3421}
\begin{aligned}
&\frac{d}{dt}\int \langle y \rangle^{2l}
(|u_\beta|^2+|\theta_\beta|^2+|h_\beta|^2) dxdy
+ \int \langle y \rangle^{2l}(\mu|\partial_y u_\beta|^2
+\kappa |\partial_y \theta_\beta|^2
+\nu|\partial_y h_\beta|^2)dxdy\\
&\le C\delta_0^{-4}(1+\|(U, \Theta, H)(t)\|_{H^2(\mathbb{T}_x)}^{10})
+C\delta_0^{-4}(\|(u, \theta, h)(t)\|_{\mathcal{H}^m_l}^{12}
+\|(u_\beta, \theta_\beta, h_\beta)(t)\|_{L^2_l(\Omega)}^4)\\
&\quad
+\int_\Omega R_1^\beta \cdot \langle y \rangle^{2l}u_\beta dxdy
+\int_\Omega R_2^\beta \cdot \langle y \rangle^{2l} \theta^\beta dxdy
+\int_\Omega R_3^\beta \cdot \langle y \rangle^{2l} h_\beta dxdy.
\end{aligned}
\end{equation}
\textbf{ We claim the following estimate:}
\begin{equation}\label{claim3}
\begin{aligned}
&\int_\Omega R_1^\beta \cdot \langle y \rangle^{2l}u_\beta dxdy
+\int_\Omega R_2^\beta \cdot \langle y \rangle^{2l} \theta^\beta dxdy
+\int_\Omega R_3^\beta \cdot \langle y \rangle^{2l} h_\beta dxdy\\
&\le  \frac{\mu}{2} \|\partial_y u_\beta\|_{L^2_l(\Omega)}^2
+\frac{\kappa}{2} \|\partial_y \theta_\beta\|_{L^2_l(\Omega)}^2
+\frac{\nu}{2} \|\partial_y h_\beta\|_{L^2_l(\Omega)}^2
+\|\partial_\tau^\beta r_1\|_{L^2_l(\Omega)}^2\\
&\quad+\|\eta_1 \partial_\tau^\beta r_4\|_{L^2_l(\Omega)}^2
+\|\partial_\tau^\beta r_2\|_{L^2_l(\Omega)}^2
+\|\eta_2 \partial_\tau^\beta r_4\|_{L^2_l(\Omega)}^2
+\|\partial_\tau^\beta r_3\|_{L^2_l(\Omega)}^2\\
&\quad+\|\eta_3 \partial_\tau^\beta r_4\|_{L^2_l(\Omega)}^2
+C\delta_0^{-4}
(1+\sum_{|\beta| \le m+2}\|\partial_\tau^\beta (U, \Theta, H)(t)\|_{L^2(\mathbb{T}_x)}^8)\\
&\quad +C\delta_0^{-4}(\|(u, \theta, h)(t)\|_{\mathcal{H}^m_l}^8
+\|(u_\beta, \theta_\beta, h_\beta)(t)\|_{L^2_l(\Omega)}^4).
\end{aligned}
\end{equation}
The combination of \eqref{3421} and \eqref{claim3} yields the estimate
\eqref{341} directly. Therefore, we complete the proof of Lemma \ref{lemma3.4}.
\end{proof}

\begin{proof}[\textbf{Proof of \eqref{claim3}.}]
Firstly, we give the estimate for the term
$\int_\Omega R_1^\beta \cdot \langle y \rangle^{2l}u_\beta dxdy$.
By virtue of H\"{o}lder inequality, we find
\begin{equation}\label{c31}
\int_\Omega \partial_\tau^\beta r_1 \cdot \langle y \rangle^{2l}u_\beta dxdy
\le \|\partial_\tau^\beta r_1\|_{L^2_l(\Omega)}\|u_\beta\|_{L^2_l(\Omega)}
\le C\|\partial_\tau^\beta r_1\|_{L^2_l(\Omega)}^2
+C\|u_\beta\|_{L^2_l(\Omega)}^2.
\end{equation}
In view of the definition of function \eqref{function1}, one attains that
\begin{equation}\label{c32}
\|\partial_\tau^\beta(U_x \phi' u)\|_{L^2_l(\Omega)}
\le \sum_{\widetilde{\beta} \le \beta}C_{\widetilde{\beta}}^{\beta}
\|\partial_\tau^{\widetilde{\beta}}(U_x \phi')\|_{L^\infty(\Omega)}
\|\partial_\tau^{\widetilde{\beta}-\beta} u\|_{L^2_l(\Omega)}
\le C\sum_{|\beta| \le m}\|\partial_\tau^\beta U_x\|_{L^\infty(\mathbb{T}_x)}
\|u\|_{\mathcal{H}^m_l}.
\end{equation}
Similarly, it is easy to deduce that
\begin{equation}\label{c33}
\begin{aligned}
&\|\partial_\tau^\beta(H_x \phi' h)\|_{L^2_l(\Omega)}
\le C\sum_{|\beta| \le m}\|\partial_\tau^\beta H_x\|_{L^\infty(\mathbb{T}_x)}
\|h\|_{\mathcal{H}^m_l},\\
&\|\partial_\tau^\beta(U \phi^{(3)} \theta)\|_{L^2_l(\Omega)}
\le C\sum_{|\beta| \le m}\|\partial_\tau^\beta U\|_{L^\infty(\mathbb{T}_x)}
\|\theta\|_{\mathcal{H}^m_l},\\
&\|\partial_\tau^\beta(U\phi'' \theta_y)\|_{L^2_l(\Omega)}
\le C\sum_{|\beta| \le m}\|\partial_\tau^\beta U\|_{L^\infty(\mathbb{T}_x)}
\|\theta\|_{\mathcal{H}^m_l}
+\|U\|_{L^\infty(\mathbb{T}_x)}\|\partial_\tau^\beta \theta_y\|_{L^2_l(\Omega)}.
\end{aligned}
\end{equation}
The combination of estimates \eqref{c32} and \eqref{c33} gives that
\begin{equation}\label{c34a}
\begin{aligned}
\int_\Omega \partial_\tau^\beta II_1\cdot \langle y \rangle^{2l}u_\beta dxdy
&\le \frac{\kappa}{4} \|\partial_\tau^\beta \theta_y\|_{L^2_l(\Omega)}^2
+C(1+\sum_{|\beta| \le m+2}\|\partial_\tau^\beta (U, \Theta, H)\|_{L^2(\mathbb{T}_x)}^4)\\
&\quad +C(\|(u, \theta, h)\|_{\mathcal{H}^m_l}^4
+\|u_\beta\|_{L^2_l(\Omega)}^2).
\end{aligned}
\end{equation}
On the other hand, we get from the inequality \eqref{equ2}(or see \eqref{b9}) that
\begin{equation}\label{c34b}
\|\partial_\tau^\beta \theta_y\|_{L^2_l(\Omega)}^2
\le 2\|\partial_y \theta_\beta\|_{L^2_l(\Omega)}^2
+C\delta_0^{-4}\|h_\beta\|_{L^2_l(\Omega)}^2,
\end{equation}
where we have used the estimate \eqref{a342}.
Then combination of \eqref{c34a} and \eqref{c34b} yields directly
\begin{equation}\label{c34}
\begin{aligned}
\int_\Omega \partial_\tau^\beta II_1\cdot \langle y \rangle^{2l}u_\beta dxdy
&\le \frac{\kappa}{2} \|\partial_y \theta_\beta\|_{L^2_l(\Omega)}^2
+C(1+\sum_{\beta \le m+2}\|\partial_\tau^\beta (U, \Theta, H)\|_{L^2(\mathbb{T}_x)}^4)\\
&\quad +C\|(u, \theta, h)\|_{\mathcal{H}^m_l}^4
+C\delta_0^{-4}\|(u_\beta, h_\beta)\|_{L^2_l(\Omega)}^2.
\end{aligned}
\end{equation}
By the definition of communicator operator $[\cdot, \cdot]$, it is easy to deduce that
\begin{equation}\label{c35}
[\partial_\tau^\beta, (u+U\phi')\partial_x-U_x \phi \partial_y]u
=\sum_{0<\widetilde{\beta} \le \beta}~C_{\widetilde{\beta}}^\beta~
\partial_\tau^{\widetilde{\beta}}
[(u+U\phi')\partial_x-U_x \phi \partial_y]\partial_\tau^{\beta-\widetilde{\beta}}u.
\end{equation}
In view of the inequality \eqref{e3}, one arrives at directly
\begin{equation}\label{c36}
\begin{aligned}
&\|\partial_\tau^{\widetilde{\beta}}
[(u+U\phi')\partial_x-U_x \phi \partial_y]\partial_\tau^{\beta-\widetilde{\beta}}u\|_{L^2_l(\Omega)}\\
&\le \|\partial_\tau^{e_i} u\|_{\mathcal{H}_l^{m-1}}
     \|\partial_x u\|_{\mathcal{H}_l^{m-1}}
     +\|\partial^{\widetilde{\beta}}_\tau U\|_{L^\infty(\mathbb{T}_x)}
     \|\partial_\tau^{\beta-\widetilde{\beta}}\partial_x u\|_{L^2_l(\Omega)}\\
&\quad +\|\langle y\rangle^{-1}\partial_\tau^{\widetilde{\beta}}
          (U_x \phi)\|_{L^\infty(\Omega)}
        \|\langle y\rangle^{l+1}\partial_\tau^{\beta-\widetilde{\beta}}
          \partial_y u\|_{L^2(\Omega)}\\
&\le \|u\|_{\mathcal{H}^m_l}^2
     +C\sum_{|\beta| \le m+1}\|\partial_\tau^\beta U\|_{L^\infty(\mathbb{T}_x)}
     \|u\|_{\mathcal{H}^m_l}.
\end{aligned}
\end{equation}
Then, we can deduce from \eqref{c35} and \eqref{c36} that
\begin{equation}\label{c37}
\|[\partial_\tau^\beta, (u+U\phi')\partial_x-U_x \phi \partial_y]u\|_{L^2_l(\Omega)}
\le  \|u\|_{\mathcal{H}^m_l}^2
     +C\sum_{|\beta| \le m+2}\|\partial_\tau^\beta U\|_{L^2(\mathbb{T}_x)}
     \|u\|_{\mathcal{H}^m_l}.
\end{equation}
Similarly, we can also obtain that
\begin{equation}\label{c38}
\|[\partial_\tau^\beta, (h+H\phi')\partial_x-H_x \phi \partial_y]h\|_{L^2_l(\Omega)}
\le  \|h\|_{\mathcal{H}^m_l}^2
     +C\sum_{|\beta| \le m+2}\|\partial_\tau^\beta H\|_{L^2(\mathbb{T}_x)}
     \|h\|_{\mathcal{H}^m_l}.
\end{equation}
By virtue of the divergence free condition of velocity
(i.e. $\partial_x u+\partial_y v=0$), one arrives at
\begin{equation}\label{c39}
\begin{aligned}
\|[\partial_\tau^\beta, U\phi^{(2)}]v\|_{L^2_l(\Omega)}
&\le C\sum_{0<\widetilde{\beta} \le \beta}
\|\partial_\tau^{\widetilde{\beta}}(U\phi^{(2)})\|_{L^\infty(\Omega)}
\|\langle y\rangle^{-1}\partial_\tau^{\beta-\widetilde{\beta}}
   \partial_y^{-1}u_x\|_{L^2(\Omega)}\\
&\le C\sum_{|\beta| \le m+1}\|\partial_\tau^\beta U\|_{L^2(\mathbb{T}_x)}\|u\|_{\mathcal{H}^m_l}.
\end{aligned}
\end{equation}
Similarly, we obtain directly
\begin{equation}\label{c310}
\|[\partial_\tau^\beta, H\phi^{(2)}]g\|_{L^2_l(\Omega)}
\le C\sum_{|\beta| \le m+1}\|\partial_\tau^\beta H\|_{L^2(\mathbb{T}_x)}\|h\|_{\mathcal{H}^m_l}.
\end{equation}
Integrating by part and applying the homogeneous
boundary condition \eqref{eq15}$_4$ yields that
\begin{equation*}
\begin{aligned}
&\mu \int_\Omega \partial_y([\partial_\tau^\beta, (\theta+\Theta \phi'+1)\partial_y]u)
\cdot \langle y \rangle^{2l}u_\beta dxdy\\
&=-\mu \int_\Omega [\partial_\tau^\beta, (\theta+\Theta \phi'+1)\partial_y]u
\cdot \langle y \rangle^{2l}\partial_y u_\beta dxdy\\
&\quad -2 \mu l \int_\Omega [\partial_\tau^\beta, (\theta+\Theta \phi'+1)\partial_y]u
\cdot \langle y \rangle^{2l-1}u_\beta dxdy.
\end{aligned}
\end{equation*}
In view of the definition of communicator operator $[\cdot, \cdot]$,
it is easy to check that
\begin{equation*}
[\partial_\tau^\beta, (\theta+\Theta \phi'+1)\partial_y]u
=\sum_{0<\widetilde{\beta} \le \beta}C^\beta_{\widetilde{\beta}}
\partial_\tau^{\widetilde{\beta}-e_i}\partial_\tau^{e_i}\theta\cdot
\partial_\tau^{\beta-\widetilde{\beta}}\partial_y u
+\sum_{0<\widetilde{\beta}\le \beta}C^\beta_{\widetilde{\beta}}
\partial_\tau^{\widetilde{\beta}}
(\Theta \phi')\partial^{\beta-\widetilde{\beta}}_\tau \partial_y u.
\end{equation*}
Then, we apply the inequality \eqref{e3} and Cauchy inequality to get that
\begin{equation}\label{c311}
\begin{aligned}
&\left|\mu \int \partial_y([\partial_\tau^\beta, (\theta+\Theta \phi'+1)\partial_y]u)
\cdot \langle y \rangle^{2l}u_\beta dxdy\right|\\
&\le \mu(\|\theta\|_{\mathcal{H}^m_l}
+\sum_{|\beta|\le m }\|\partial_\tau^\beta \Theta\|_{L^\infty(\mathbb{T}_x)})
\|u\|_{\mathcal{H}^m_l}\|\partial_y u_\beta\|_{L^2_l(\Omega)}\\
&\quad +
\mu(\|\theta\|_{\mathcal{H}^m_l}
+\sum_{|\beta|\le m }\|\partial_\tau^\beta \Theta\|_{L^\infty(\mathbb{T}_x)})
\|u\|_{\mathcal{H}^m_l}\| u_\beta\|_{L^2_l(\Omega)}\\
&\le \frac{\mu}{2} \|\partial_y u_\beta\|_{L^2_l(\Omega)}^2
+C(1+\sum_{|\beta| \le m+1}\|\partial_\tau^\beta \Theta\|_{L^2(\mathbb{T}_x)}^4)
+C(\|(u, \theta)\|_{\mathcal{H}^m_l}^4+\|u_\beta\|_{L^2_l(\Omega)}^2).
\end{aligned}
\end{equation}
By virtue of the divergence free condition of velocity,
it is easy to deduce that for $0< \widetilde{\beta} <\beta$
\begin{equation*}
\|\partial_\tau^\beta v\cdot
\partial_\tau^{\beta-\widetilde{\beta}}\partial_y u\|_{L^2_l(\Omega)}
=\|\partial_\tau^{\widetilde{\beta}-e_i} \partial_y^{-1}\partial_\tau^{e_i+e_2}u\cdot
\partial_\tau^{\beta-\widetilde{\beta}-e_j}
\partial_\tau^{e_j+E_3} u\|_{L^2_l(\Omega)}
\le C\|u\|_{\mathcal{H}^m_l}^2,
\end{equation*}
which implies that
\begin{equation}\label{c312}
\|\sum_{0<\widetilde{\beta}<\beta}C_{\widetilde{\beta}}^\beta
\partial_\tau^\beta v\cdot \partial_\tau^{\beta-\widetilde{\beta}}
\partial_y u\|_{L^2_l(\Omega)}\le C\|u\|_{\mathcal{H}^m_l}^2.
\end{equation}
Similarly, we can find that
\begin{equation}\label{c313}
\|\sum_{0<\widetilde{\beta}<\beta}C_{\widetilde{\beta}}^\beta
\partial_\tau^\beta g\cdot \partial_\tau^{\beta-\widetilde{\beta}}
\partial_y h\|_{L^2_l(\Omega)}\le C\|h\|_{\mathcal{H}^m_l}^2.
\end{equation}
By virtue of the definition of $R_u^\beta$(see \eqref{eq6a}),
estimates \eqref{c37}-\eqref{c313} and Cauchy inequality, we find
\begin{equation}\label{c314}
\begin{aligned}
\int_\Omega R_u^\beta \cdot \langle y \rangle^{2l}u_\beta dxdy
&\le \frac{\mu}{2} \|\partial_y u_\beta\|_{L^2_l(\Omega)}^2
+C(1+\sum_{|\beta| \le m+2}\|\partial_\tau^\beta (U, \Theta, H)\|_{L^2(\mathbb{T}_x)}^4)\\
&\quad +C(\|(u, \theta, h)\|_{\mathcal{H}^m_l}^4
+\|u_\beta\|_{L^2_l(\Omega)}^2).
\end{aligned}
\end{equation}
In view of the definition of $r_4$(see \eqref{function5})
estimate \eqref{estimate4} and Cauchy inequality, it is easy to check that
\begin{equation}\label{c315}
\left|\int_\Omega \eta_1 \partial_\tau^\beta r_4\cdot
\langle y \rangle^{2l}u_\beta dxdy\right|
\le \|\eta_1 \partial_\tau^\beta r_4\|_{L^2_l(\Omega)}
\|u_\beta\|_{L^2_l(\Omega)}
\le \|\eta_1 \partial_\tau^\beta r_4\|_{L^2_l(\Omega)}^2
+\|u_\beta\|_{L^2_l(\Omega)}^2.
\end{equation}
Similarly, we can find  directly
\begin{equation}\label{c316}
\begin{aligned}
\|\eta_1 \partial_\tau^\beta(H_x \phi u)\|_{L^2_l(\Omega)}^2
&\le C\sum_{\widetilde{\beta}\le \beta}
\|\langle y\rangle \eta_1\|_{L^\infty(\Omega)}^2
\|\langle y\rangle^{-1}\partial_\tau^{\widetilde{\beta}}H_x \phi\|_{L^\infty(\Omega)}^2
\|\partial_\tau^{\beta-\widetilde{\beta}}u\|_{L^2_l(\Omega)}^2\\
&\le C\delta_0^{-2}(1+\sum_{|\beta| \le m+1}
\|\partial_\tau^\beta(U, \Theta, H)\|_{L^2(\mathbb{T}_x)}^8)
+C\delta_0^{-2}\|(u, \theta, h)\|_{\mathcal{H}^m_l}^8.
\end{aligned}
\end{equation}
With the help of estimates \eqref{estimate4}, \eqref{e7} and
divergence free condition of velocity, one arrives at
\begin{equation}\label{c317}
\begin{aligned}
\|\eta_1[\partial_\tau^\beta, H\phi']v\|_{L^2_l(\Omega)}^2
&\le C\sum_{0<\widetilde{\beta}\le \beta}
\|\langle y\rangle^{l+1}\eta_1\|_{L^\infty(\Omega)}^2
\|\partial_\tau^{\widetilde{\beta}}(H\phi')\|_{L^\infty(\Omega)}^2
\|\langle y\rangle^{-1}\partial_\tau^{\beta-\widetilde{\beta}}
\partial_y^{-1}u_x\|_{L^2(\Omega)}^2\\
&\le C\delta_0^{-2}
(1+\sum_{|\beta| \le m+1}\|\partial_\tau^\beta(U, \Theta, H)\|_{L^2(\mathbb{T}_x)}^8)
+C\delta_0^{-2}\|(u, \theta, h)\|_{\mathcal{H}^m_l}^8.
\end{aligned}
\end{equation}
Obviously, it is easy to check that
\begin{equation*}
\begin{aligned}
&\eta_1[\partial_\tau^\beta,
(u+U\phi')\partial_x-U_x \phi \partial_y]\psi\\
&=-\sum_{0<\widetilde{\beta} \le \beta}C_{\widetilde{\beta}}^\beta~ \eta_1~
\partial_\tau^{\widetilde{\beta}}u
\cdot \partial_\tau^{\beta-\widetilde{\beta}}\partial_y^{-1}h_x
-\sum_{0<\widetilde{\beta} \le \beta}C_{\widetilde{\beta}}^{\beta}
~\eta_1~ \partial_\tau^{\widetilde{\beta}}(U\phi')
\partial_\tau^{\beta-\widetilde{\beta}}\partial_y^{-1}h_x\\
&\quad +\sum_{0<\widetilde{\beta} \le \beta}C_{\widetilde{\beta}}^\beta
~\eta_1~ \partial_\tau^{\widetilde{\beta}}(U_x \phi)
\partial_\tau^{\beta-\widetilde{\beta}}h.
\end{aligned}
\end{equation*}
In view of the estimate \eqref{estimate4} and inequality \eqref{e8}, one finds
\begin{equation}\label{c318}
\begin{aligned}
&\|\eta_1 \partial_\tau^{\widetilde{\beta}}u~
\partial_\tau^{\beta-\widetilde{\beta}}\partial_y^{-1}h_x\|_{L^2_l(\Omega)}
\le \|\langle y\rangle \eta_1\|_{L^\infty}
\|\partial_\tau^{\widetilde{\beta}-e_i}\partial_\tau^{e_i}u~
\partial_\tau^{\beta-\widetilde{\beta}}\partial_y^{-1}h_x\|_{L^2_{l-1}(\Omega)}\\
&\le \delta_0^{-1}(\|(U, \Theta, H)\|_{L^\infty(\mathbb{T}_x)}
+\|(u, \theta, h)\|_{\mathcal{H}^3_0})
\|(u, h)\|_{\mathcal{H}^m_l}^2
\end{aligned}
\end{equation}
Similarly, we obtain that
\begin{equation}\label{c319}
\begin{aligned}
&\|\eta_1 \partial_\tau^{\widetilde{\beta}}(U\phi')
\partial_\tau^{\beta-\widetilde{\beta}}\partial_y^{-1}h_x\|_{L^2_l(\Omega)}\\
&\le C\delta_0^{-1}
(\|(U, \Theta, H)\|_{L^\infty(\mathbb{T}_x)}+\|(u, \theta, h)\|_{\mathcal{H}_l^3})
\|\partial_\tau^{\widetilde{\beta}}U\|_{L^\infty(\mathbb{T}_x)}
\|\partial_\tau^{\beta-\widetilde{\beta}} h_x\|_{L^2(\Omega)},
\end{aligned}
\end{equation}
and
\begin{equation}\label{c320}
\begin{aligned}
&\|\eta_1 \partial_\tau^{\widetilde{\beta}}(U_x \phi)
\partial_\tau^{\beta-\widetilde{\beta}}h\|_{L^2_l(\Omega)}\\
&\le C\delta_0^{-1}
(\|(U, \Theta, H)\|_{L^\infty(\mathbb{T}_x)}+\|(u, \theta, h)\|_{\mathcal{H}_l^3})
\|\partial_\tau^{\widetilde{\beta}}U_x\|_{L^\infty(\mathbb{T}_x)}
\|\partial_\tau^{\beta-\widetilde{\beta}} h\|_{L^2(\Omega)}.
\end{aligned}
\end{equation}
The combination of \eqref{c318}-\eqref{c320} yields directly
\begin{equation}\label{c321}
\|\eta_1~[\partial_\tau^\beta,
(u+U\phi')\partial_x-U_x \phi \partial_y]\psi\|_{L_l^2(\Omega)}
\le C\delta_0^{-1}(1+\|(U, \Theta, H)\|_{H^1(\mathbb{T}_x)}^4)
+C\delta_0^{-1}\|(u, \theta, h)\|_{\mathcal{H}_l^m}^4.
\end{equation}
By virtue of the divergence free condition of velocity, it is easy to check that
\begin{equation*}
\sum_{0<\widetilde{\beta} < \beta} \partial_\tau^{\widetilde{\beta}} v\cdot
\partial_\tau^{\beta-\widetilde{\beta}} \partial_y \psi
=-\sum_{0<\widetilde{\beta} < \beta} \partial_\tau^{\widetilde{\beta}-e_i}
\partial_y^{-1}(\partial_\tau^{e_i}u_x)
\partial_\tau^{\beta-\widetilde{\beta}-e_i}(\partial_\tau^{e_i}h).
\end{equation*}
Then, the application of \eqref{estimate4} and \eqref{e7} yields directly
\begin{equation*}
\begin{aligned}
&\|\eta_1 \partial_\tau^{\widetilde{\beta}} v\cdot
\partial_\tau^{\beta-\widetilde{\beta}}\partial_y \psi\|_{L^2_l(\Omega)}\\
&\le \|\langle y\rangle^{l+1}\eta_1\|_{L^\infty(\Omega)}
\|\partial_\tau^{\widetilde{\beta}-e_i}
\partial_y^{-1}(\partial_\tau^{e_i}u_x)
\partial_\tau^{\beta-\widetilde{\beta}-e_i}(\partial_\tau^{e_i}h)\|_{L^2_{-1}(\Omega)}\\
&\le C\delta_0^{-1}(\|(U, \Theta, H)\|_{H^1(\mathbb{T}_x)}
+\|(u, \theta, h)\|_{\mathcal{H}^3_l})\|(u, h)\|_{\mathcal{H}^m_0}^2,
\end{aligned}
\end{equation*}
which implies that
\begin{equation}\label{c322}
\|\sum_{0<\widetilde{\beta} < \beta}\eta_1 \partial_\tau^{\widetilde{\beta}} v\cdot
\partial_\tau^{\beta-\widetilde{\beta}}\partial_y \psi\|_{L^2_l(\Omega)}
\le C\delta_0^{-1}(\|(U, \Theta, H)\|_{H^1(\mathbb{T}_x)}
+\|(u, \theta, h)\|_{\mathcal{H}^3_l})\|(u, h)\|_{\mathcal{H}^m_0}^2.
\end{equation}
Hence, the combination of \eqref{c316}, \eqref{c317}, \eqref{c321} and \eqref{c322}
gives that
\begin{equation}\label{c323}
\begin{aligned}
&\left|\int_\Omega \eta_1 R_4^\beta\cdot
\langle y \rangle^{2l}u_\beta dxdy\right|\\
&\le C\delta_0^{-2}(1+\sum_{|\beta| \le m+1}
\|\partial_\tau^\beta(U, \Theta, H)\|_{L^2(\mathbb{T}_x)}^8)
+C\delta_0^{-2}(\|(u, \theta, h)\|_{\mathcal{H}^m_l}^8
+\|u_\beta\|_{L^2_l(\Omega)}^2).
\end{aligned}
\end{equation}
One applies the estimate \eqref{estimate4} and Sobolev inequality to get directly
\begin{equation}\label{c324}
\begin{aligned}
&\left|\nu \int_\Omega \eta_1 \eta_3 \partial_\tau^\beta h\cdot
\langle y \rangle^{2l}u_\beta dxdy\right|\\
&\le C\|\eta_1\|_{L^\infty(\Omega)}\|\eta_3\|_{L^\infty(\Omega)}
\|\partial_\tau^\beta h\|_{L^2_l(\Omega)}\|u_\beta\|_{L^2_l(\Omega)}\\
&\le C\delta_0^{-2}(1+\|(U, \Theta, H)\|_{H^1(\mathbb{T}_x)}^4)
+C\delta_0^{-2}(\|(u, \theta, h)\|_{\mathcal{H}^m_l}^4
+\|u_\beta\|_{L^2_l(\Omega)}^4).
\end{aligned}
\end{equation}
The application of \eqref{estimate4}, \eqref{e7},  H\"{o}lder inequality
and divergence free condition of magnetic field yields
\begin{equation}\label{c325}
\begin{aligned}
&\left|\int_\Omega \eta_3 (g-H_x \phi)\partial_\tau^\beta h\cdot
\langle y \rangle^{2l}u_\beta dxdy\right|\\
&\le \|\langle y\rangle^{l+1}\eta_3\|_{L^\infty(\Omega)}
     (\|\langle y\rangle^{-1}\partial_y^{-1}h_x\|_{L^\infty(\Omega)}
     +\|\langle y\rangle^{-1} H_x \phi\|_{L^\infty(\Omega)})
     \|\partial_\tau^\beta h\|_{L^2(\Omega)}
     \|u_\beta\|_{L^2(\Omega)}\\
&\le C\delta_0^{-1}(1+\|(U, \Theta, H)\|_{H^1(\mathbb{T}_x)}^8)
     +C\delta_0^{-1}(\|(u, \theta, h)\|_{\mathcal{H}^m_l}^8
     +\|u_\beta\|_{L^2_l(\Omega)}^4).
\end{aligned}
\end{equation}
By virtue of H\"{o}lder inequality, \eqref{e7} and estimate \eqref{estimate4}$_3$,
we get that
\begin{equation}\label{c326}
\begin{aligned}
&\left|\int_\Omega \zeta_1 (\partial_\tau^\beta \psi)
\cdot \langle y\rangle^{2l}u_\beta dxdy\right|\\
&\le \|\langle y\rangle^{l+1}\zeta_1\|_{L^\infty(\Omega)}
\|\langle y\rangle^{-1}\partial_\tau^\beta \partial_y^{-1}h\|_{L^2(\Omega)}
\|u_\beta\|_{L^2_l(\Omega)}\\
&\le C\delta_0^{-3}(1+\sum_{|\beta|\le 1}\|\partial_\tau^\beta(U, \Theta, H)\|_{H^1(\mathbb{T}_x)}^6)
+C\delta_0^{-3}(\|(u, \theta, h)\|_{\mathcal{H}^m_l}^6
+\|u_\beta\|_{L^2_l(\Omega)}^4).
\end{aligned}
\end{equation}
Then, the combination of \eqref{c31}, \eqref{c34}, \eqref{c314}, \eqref{c315}
and \eqref{c324}-\eqref{c326} yields directly
\begin{equation}\label{c327}
\begin{aligned}
\left|\int_\Omega R_1^\beta \cdot \langle y \rangle^{2l}u_\beta dxdy\right|
&\le  \frac{\mu}{4} \|\partial_y u_\beta\|_{L^2_l(\Omega)}^2
+\frac{\kappa}{2} \|\partial_y \theta_\beta\|_{L^2_l(\Omega)}^2
+\|\partial_\tau^\beta r_1\|_{L^2_l(\Omega)}^2\\
&\quad
+C\delta_0^{-4}(\|(u, \theta, h)\|_{\mathcal{H}^m_l}^8
+\|(u_\beta, h_\beta)\|_{L^2_l(\Omega)}^4)\\
&\quad +C\delta_0^{-4}
(1+\sum_{|\beta| \le m+2}\|\partial_\tau^\beta (U, \Theta, H)\|_{L^2(\mathbb{T}_x)}^8).
\end{aligned}
\end{equation}
Similarly, we can also obtain that
\begin{equation}\label{c328}
\begin{aligned}
\int_\Omega R_2^\beta \cdot \langle y \rangle^{2l} \theta_\beta dxdy
&\le  \frac{\mu}{4}\|\partial_y u_\beta\|_{L^2_l(\Omega)}^2
+\frac{\nu}{2}\|\partial_y h_\beta\|_{L^2_l(\Omega)}^2
+\|\partial_\tau^\beta r_2\|_{L^2_l(\Omega)}^2\\
&\quad+\|\eta_2 \partial_\tau^\beta r_4\|_{L^2_l(\Omega)}^2
+C\delta_0^{-4}(\|(u, \theta, h)\|_{\mathcal{H}^m_l}^8
+\|(\theta_\beta, h_\beta)\|_{L^2_l(\Omega)}^4)\\
&\quad +C\delta_0^{-4}
(1+\sum_{|\beta| \le m+2}\|\partial_\tau^\beta (U, \Theta, H)\|_{L^2(\mathbb{T}_x)}^8),
\end{aligned}
\end{equation}
and
\begin{equation}\label{c329}
\begin{aligned}
\int_\Omega R_3^\beta \cdot \langle y \rangle^{2l} h_\beta dxdy
&\le \|\partial_\tau^\beta r_3\|_{L^2_l(\Omega)}^2
+\|\eta_3 \partial_\tau^\beta r_4\|_{L^2_l(\Omega)}^2
+C\delta_0^{-3}\|(u, \theta, h)\|_{\mathcal{H}^m_l}^6\\
&\quad
+C\delta_0^{-3}\|h_\beta\|_{L^2_l(\Omega)}^6+C\delta_0^{-3}
(1+\sum_{|\beta| \le m+2}\|\partial_\tau^\beta (U, \Theta, H)\|_{L^2(\mathbb{T}_x)}^6).
\end{aligned}
\end{equation}
Therefore, the combination of \eqref{c327}-\eqref{c329} completes the proof of \eqref{claim3}.
\end{proof}

\subsection{Closeness of the a priori estimates}
\quad
In this subsection, we will give the proof for the Proposition \ref{pro1}
by collecting all the estimates obtained in this section.
Indeed, the combination of estimates \eqref{equ1}, \eqref{equ2} and
\eqref{a342} yields immediately
\begin{equation}\label{a343}
\begin{aligned}
\|(u, \theta, h)(t)\|_{\mathcal{H}^m_l}^2
&=\sum_{\tiny\substack{|\alpha|\le m \\ |\beta|\le m-1}}
\|D^\alpha(u, \theta, h)(t)\|_{L^2_{l+k}(\Omega)}^2
+\sum_{|\beta|=m}\|\partial_\tau^{\beta}(u, \theta, h)(t)\|_{L^2_{l}(\Omega)}^2\\
&\le \sum_{\tiny\substack{|\alpha|\le m \\ |\beta|\le m-1}}
\|D^\alpha(u, \theta, h)(t)\|_{L^2_{l+k}(\Omega)}^2
+36\delta_0^{-4}\sum_{|\beta|=m}\|(u_\beta, \theta_\beta, h_\beta)(t)\|_{L^2_l(\Omega)}^2,
\end{aligned}
\end{equation}
and
\begin{equation}\label{a344}
\begin{aligned}
\|\partial_y(u, \theta, h)(t)\|_{\mathcal{H}^m_l}^2
&=\sum_{\tiny\substack{|\alpha|\le m \\ |\beta|\le m-1}}
\|D^\alpha \partial_y(u, \theta, h)(t)\|_{L^2_{l+k}(\Omega)}^2
+\sum_{|\beta|=m}\|\partial_\tau^{\beta}
\partial_y(u, \theta, h)(t)\|_{L^2_{l}(\Omega)}^2\\
&\le \sum_{\tiny\substack{|\alpha|\le m \\ |\beta|\le m-1}}
\|D^\alpha \partial_y(u, \theta, h)(t)\|_{L^2_{l+k}(\Omega)}^2
+2\sum_{|\beta|=m}\|\partial_y(u_\beta, \theta_\beta, h_\beta)(t)\|_{L^2_l(\Omega)}^2\\
&\quad +72\delta_0^{-4}\sum_{|\beta|=m}\| h_\beta(t)\|_{L^2_l(\Omega)}^2.
\end{aligned}
\end{equation}

Then, we can obtain the following proposition that will play an important role
in giving the proof for the Proposition \ref{pro1}.

\begin{proposition}\label{pro2}
Under the assumptions of Proposition \ref{pro1}, there exists a constant
$C>0$, depending only on $m, M_0$ and $\phi$, such that
\begin{equation}\label{351}
\underset{0\le s \le t}{\sup}\|(u, \theta, h)(s)\|_{\mathcal{H}^m_l}
\le \left\{F(0)+\int_0^t G(\tau)d\tau\right\}^{\frac{1}{2}}
\left\{1-C\delta_0^{-8}~t~[F(0)+\int_0^t G(\tau)d\tau]^5\right\}^{-\frac{1}{10}}.
\end{equation}
for small time $t$, where the quantities $F(0)$ and $G(t)$ are defined as follows
\begin{equation}\label{functionf}
F(0)\triangleq
\sum_{\tiny\substack{|\alpha|\le m \\ |\beta|\le m-1}}
\|D^\alpha(u, \theta, h)(0)\|_{L^2_{l+k}(\Omega)}^2
+36\delta_0^{-4}
\sum_{|\beta|=m}\|(u_{\beta 0}, \theta_{\beta 0}, h_{\beta 0})\|_{L^2_{l}(\Omega)}^2,
\end{equation}
\begin{equation}\label{functiong}
\begin{aligned}
G(t)
&\triangleq C\delta_0^{-8}(\sum_{|\beta|\le m+2}
\|\partial_\tau^\beta(U, \Theta, H, P)(t)\|_{L^2(\mathbb{T}_x)}^2+1)^5
+C\sum_{\tiny\substack{|\alpha|\le m \\ |\beta|\le m-1}}
\|D^\alpha(r_1, r_2, r_3)(t)\|_{L^2_{l+k}(\Omega)}^2\\
&\quad +C\delta_0^{-4}\sum_{|\beta|=m}
\left\{\|\partial_\tau^\beta(r_1, r_2, r_3)(t)\|_{L^2_l(\Omega)}^2
+4\delta_0^{-4}\|\partial_\tau^\beta r_4(t)\|_{L^2_{-1}(\Omega)}^2\right\}.
\end{aligned}
\end{equation}
Also, we have that for $i=1,2$,
\begin{equation}\label{352}
\begin{aligned}
&\|\langle y\rangle^{l+1}\partial_y^i (u, \theta, h)(t, x, y)\|_{L^\infty(\Omega)}\\
&\le \|\langle y\rangle^{l+1}\partial_y^i (u_0, \theta_0, h_0)(x, y)\|_{L^\infty(\Omega)}\\
&\quad +C t  \left\{F(0)+\int_0^t G(\tau)d\tau\right\}^{\frac{1}{2}}
\left\{1-C\delta_0^{-8}~t~[F(0)+\int_0^t G(\tau)d\tau]^5\right\}^{-\frac{1}{10}},
\end{aligned}
\end{equation}
and
\begin{equation}\label{353}
h(t, x, y)
\ge h_0(x, y)-C t  \left\{F(0)+\int_0^t G(\tau)d\tau\right\}^{\frac{1}{2}}
\left\{1-C\delta_0^{-8}~t~[F(0)+\int_0^t G(\tau)d\tau]^5\right\}^{-\frac{1}{10}}.
\end{equation}
\end{proposition}
\begin{proof}
Indeed, multiplying \eqref{341} by $36\delta_0^{-4}$ and adding with
\eqref{331}, then we find
\begin{equation}\label{354}
\begin{aligned}
&\frac{d}{dt}(\sum_{\tiny\substack{|\alpha|\le m \\ |\beta|\le m-1}}
\|D^\alpha(u, \theta, h)(t)\|_{L^2_{l+k}(\Omega)}^2
+36\delta_0^{-4}
\sum_{|\beta|=m}\|(u_\beta, \theta_\beta, h_\beta)(t)\|_{L^2_{l}(\Omega)}^2)\\
&\quad +
\sum_{\tiny\substack{|\alpha|\le m \\ |\beta|\le m-1}}
\|D^\alpha\partial_y(u, \theta, h)(t)\|_{L^2_{l+k}(\Omega)}^2
+36\delta_0^{-4}
\sum_{|\beta|=m}\|\partial_y(u_\beta, \theta_\beta, h_\beta)(t)\|_{L^2_{l}(\Omega)}^2\\
&\le \delta_1 C
\|\partial_y( u, \theta, h)(t)\|_{\mathcal{H}^m_0}^2
+C\delta_1^{-1}(\|(u, \theta, h)(t)\|_{\mathcal{H}^m_l}^8+1)
+\sum_{\tiny\substack{|\alpha|\le m \\ |\beta|\le m-1}}
\|D^\alpha(r_1, r_2, r_3)(t)\|_{L^2_{l+k}(\Omega)}^2\\
&\quad +36\delta_0^{-4}\sum_{|\beta|=m}
\left\{\sum_{i=1}^3\|\partial_\tau^\beta r_i\|_{L^2_l(\Omega)}^2
+\|\eta_i \partial_\tau^\beta r_4\|_{L^2_l(\Omega)}^2\right\}
+C\delta_0^{-4}\!\!\sum_{|\beta|\le m+2}\!\!
\|\partial_\tau^\beta(U, \Theta, H)(t)\|_{L^2(\mathbb{T}_x)}^{10}\\
&\quad +C\delta_0^{-8}(\|(u, \theta, h)(t)\|_{\mathcal{H}^m_l}^{12}
+\|(u_\beta, \theta_\beta, h_\beta)(t)\|_{L^2_l(\Omega)}^4+1).
\end{aligned}
\end{equation}
Choosing $\delta_1$ small enough in \eqref{354}
and applying the estimates \eqref{equ1}, \eqref{equ2}, \eqref{a343},
and \eqref{a344}, it is easy to deduce that
\begin{equation}\label{355}
\begin{aligned}
&\frac{d}{dt}(\sum_{\tiny\substack{|\alpha|\le m \\ |\beta|\le m-1}}
\|D^\alpha(u, \theta, h)(t)\|_{L^2_{l+k}(\Omega)}^2
+36\delta_0^{-4}
\sum_{|\beta|=m}\|(u_\beta, \theta_\beta, h_\beta)(t)\|_{L^2_{l}(\Omega)}^2)\\
&+c_1
(\sum_{\tiny\substack{|\alpha|\le m \\ |\beta|\le m-1}}
\|D^\alpha\partial_y(u, \theta, h)(t)\|_{L^2_{l+k}(\Omega)}^2
+36\delta_0^{-4}
\sum_{|\beta|=m}\|\partial_y(u_\beta, \theta_\beta, h_\beta)(t)\|_{L^2_{l}(\Omega)}^2)\\
&\le C\delta_0^{-8}(\sum_{\tiny\substack{|\alpha|\le m \\ |\beta|\le m-1}}
\|D^\alpha(u, \theta, h)(t)\|_{L^2_{l+k}(\Omega)}^2
+36\delta_0^{-4}
\sum_{|\beta|=m}\|(u_\beta, \theta_\beta, h_\beta)(t)\|_{L^2_{l}(\Omega)}^2)\\
&\quad +C\sum_{\tiny\substack{|\alpha|\le m \\ |\beta|\le m-1}}
\|D^\alpha(r_1, r_2, r_3)(t)\|_{L^2_{l+k}(\Omega)}^2
+C\delta_0^{-8}(\sum_{|\beta|\le m+2}
\|\partial_\tau^\beta(U, \Theta, H)(t)\|_{L^2(\mathbb{T}_x)}^2+1)^5\\
&\quad +C\delta_0^{-4} \sum_{|\beta|=m}
\left\{\|\partial_\tau^\beta(r_1, r_2, r_3)(t)\|_{L^2_l(\Omega)}^2
+4\delta_0^{-4}\|\partial_\tau^\beta r_4(t)\|_{L^2_{-1}(\Omega)}^2\right\}.
\end{aligned}
\end{equation}
Then, we apply the comparison principle of ordinary differential equation
to \eqref{355} get that
\begin{equation*}
\begin{aligned}
&\sum_{\tiny\substack{|\alpha|\le m \\ |\beta|\le m-1}}
\|D^\alpha(u, \theta, h)(t)\|_{L^2_{l+k}(\Omega)}^2
+36\delta_0^{-4}
\sum_{|\beta|=m}\|(u_\beta, \theta_\beta, h_\beta)(t)\|_{L^2_{l}(\Omega)}^2\\
&+c_1\int_0^t
(\sum_{\tiny\substack{|\alpha|\le m \\ |\beta|\le m-1}}
\|D^\alpha\partial_y(u, \theta, h)(\tau)\|_{L^2_{l+k}(\Omega)}^2
+36\delta_0^{-4}
\sum_{|\beta|=m}\|\partial_y(u_\beta, \theta_\beta, h_\beta)(\tau)
\|_{L^2_{l}(\Omega)}^2)d\tau\\
&\le [F(0)+\int_0^t G(\tau)d\tau]
\left\{1-C\delta_0^{-8}t[F(0)+\int_0^t G(\tau)d\tau]^5\right\}^{-\frac{1}{5}},
\end{aligned}
\end{equation*}
which, together with \eqref{a343}, yields directly
\begin{equation*}\label{356}
\|(u, \theta, h)(t)\|_{\mathcal{H}^m_l}
\le \left\{F(0)+\int_0^t G(\tau)d\tau \right\}^{\frac{1}{2}}
\left\{1-C\delta_0^{-8}~t~[F(0)+\int_0^t G(\tau)d\tau]^5\right\}^{-\frac{1}{10}}.
\end{equation*}
As we know, we have for $i=1,2$,
\begin{equation}\label{357}
\langle y \rangle^{l+1}\partial_y^i (u, \theta, h)(t, x, y)
=\langle y \rangle^{l+1}\partial_y^i (u_0, \theta_0, h_0)(x, y)
+\int_0^t \langle y \rangle^{l+1}\partial_y^i \partial_s(u, \theta, h)(s, x, y)ds,
\end{equation}
and
\begin{equation*}\label{358}
h(t, x, y)=h_0(x, y)+\int_0^t \partial_s h(s, x, y)ds.
\end{equation*}
In view of the Sobolev embedding theorem and the relation \eqref{357}, one arrives at
\begin{equation*}
\begin{aligned}
&\|\langle y\rangle^{l+1}\partial_y^i (u, \theta, h)(t, x, y)\|_{L^\infty(\Omega)}\\
&\le \|\langle y\rangle^{l+1}\partial_y^i (u_0, \theta_0, h_0)(x, y)\|_{L^\infty(\Omega)}
+C t \underset{0\le s \le t}{\sup}\|(u, \theta, h)(s)\|_{\mathcal{H}_l^5}\\
&\le \|\langle y\rangle^{l+1}\partial_y^i (u_0, \theta_0, h_0)(x, y)\|_{L^\infty(\Omega)}\\
&\quad +C t  \left\{F(0)+\int_0^t G(\tau)d\tau\right\}^{\frac{1}{2}}
\left\{1-C\delta_0^{-8}~t~[F(0)+\int_0^t G(\tau)d\tau]^5\right\}^{-\frac{1}{10}}\\
\end{aligned}
\end{equation*}
and similarly, we also have
\begin{equation*}
\begin{aligned}
h(t, x, y)
&\ge h_0(x, y)-\int_0^t \|\partial_\tau h(\tau, x, y)\|_{L^\infty(\Omega)}d\tau\\
&\ge h_0(x, y)-C t \underset{0\le s \le t}{\sup}\|h(s)\|_{\mathcal{H}_0^3}\\
&\ge h_0(x, y)-C t  \left\{F(0)+\int_0^t G(\tau)d\tau\right\}^{\frac{1}{2}}
\left\{1-C\delta_0^{-8}~t~[F(0)+\int_0^t G(\tau)d\tau]^5\right\}^{-\frac{1}{10}}.
\end{aligned}
\end{equation*}
Therefore, we complete the proof of Proposition \ref{pro2}.
\end{proof}

\begin{proof}[\textbf{Proof of Proposition \ref{pro1}.}]
Indeed, it is easy to deduce from the definition of \eqref{functiong} that
\begin{equation}\label{ap1}
G(t)\le C\delta_0^{-8}M_0^{10}.
\end{equation}
On the other hand, it is easy to check that
$D^{\alpha}(u, \theta, h)(0, x, y), |\alpha |\le m$
can be expressed by the spatial derivatives of initial data
$(u_0, \theta_0, h_0)$ up to order $2m$.
Then, we get that
\begin{equation}\label{ap2}
F(0)\le \delta_0^{-8}\mathcal{P}(M_0+\|(u_0, \theta_0, h_0)\|_{H^{2m}_l(\Omega)}).
\end{equation}
Substituting the estimates \eqref{ap1} and \eqref{ap2} into \eqref{351}-\eqref{353},
then it is easy to get the estimates \eqref{p11}-\eqref{p13}.
Therefore, we complete the proof of Proposition \ref{pro1}.

\end{proof}

\section{Local-in-time Existence and Uniqueness}\label{sc}

\quad In this section, we will establish the local-in-time existence and
uniqueness of solution to the nonlinear MHD boundary layer problem
\eqref{eq4}.

\subsection{Local-in-time Existence for the Boundary Layer Equations}
\quad In this subsection, we investigate a parabolic regularized system
for the nonlinear problem \eqref{eq4}, which we can obtain
the local-in-time existence of solution by using the classical energy estimates.
More precisely, for a small parameter $0<\epsilon<1$, one investigates the following system:
\begin{equation}\label{eq41}
\left\{
\begin{aligned}
&\partial_t u^\epsilon+[(u^\epsilon+U\phi')\partial_x+(v^\epsilon-U_x \phi)\partial_y]u^\epsilon
-[(h^\epsilon+H\phi')\partial_x+(g^\epsilon-H_x \phi)\partial_y]h^\epsilon\\
&\quad -\mu{\partial_y}[(\theta^\epsilon+\Theta \phi'(y)+1){\partial_y}u^\epsilon]
-\epsilon \partial_x^2 u^\epsilon
+U_x \phi' u^\epsilon+U \phi'' v^\epsilon-H_x \phi' h^\epsilon\\
&\quad -H \phi'' g^\epsilon-U\phi^{(3)}\theta^\epsilon-U\phi''\theta^\epsilon_y
=r^\epsilon_1,\\
&c_v\{\partial_t \theta^\epsilon+\!(u^\epsilon+U\phi')\partial_x \theta^\epsilon
+(v^\epsilon-U_x \phi)\partial_y \theta^\epsilon \}
-\kappa \partial_y^2 \theta-\epsilon \partial_x^2 \theta^\epsilon
+c_v \Theta_x \phi' u^\epsilon+c_v \Theta \phi'' v^\epsilon\\
&\quad -\mu\theta^\epsilon (u^\epsilon_y)^2-\mu (U\phi'')^2 \theta^\epsilon
-2\mu U \phi'' \theta^\epsilon \partial_y u^\epsilon
-\mu \Theta \phi'(\partial_y u^\epsilon)^2
-2\mu \Theta U \phi' \phi'' \partial_y u^\epsilon\\
&\quad -2\mu U\phi'' \partial_y u^\epsilon
-\mu (\partial_y u^\epsilon)^2-\nu (\partial_y h^\epsilon)^2
-2\nu H \phi' \partial_y h^\epsilon=r^\epsilon_2,\\
&\partial_t h^\epsilon+[(u^\epsilon+U\phi')\partial_x
+(v^\epsilon-U_x \phi)\partial_y]h^\epsilon
-[(h^\epsilon+H\phi')\partial_x+(g^\epsilon-H_x \phi)\partial_y]u^\epsilon\\
&\quad -\nu \partial_y^2 h^\epsilon-\epsilon \partial_x^2 h^\epsilon
+H_x \phi' u^\epsilon+H \phi'' v^\epsilon- U_x \phi' h^\epsilon-U \phi'' g^\epsilon
=r^\epsilon_3,\\
&\partial_x u^\epsilon+\partial_y v^\epsilon=0,
\quad \partial_x h^\epsilon+\partial_y g^\epsilon=0,\\
&\left.(u^\epsilon, \theta^\epsilon, h^\epsilon)\right|_{t=0}
=(u_0, \theta_0, h_0)(x, y)\quad
\left.(u^\epsilon, v^\epsilon, \theta^\epsilon,
\partial_y h^\epsilon, g^\epsilon)\right|_{y=0}=0,
\end{aligned}
\right.
\end{equation}
where the source term $(r_1^\epsilon, r_2^\epsilon, r_3^\epsilon)$
is defined by
\begin{equation}
(r_1^\epsilon, r_2^\epsilon, r_3^\epsilon)(t, x, y)
\triangleq (r_1, r_2, r_3)(t, x, y)
+\epsilon (\widetilde{r}_1^\epsilon, \widetilde{r}_2^\epsilon,
\widetilde{r}_3^\epsilon)(t, x, y).
\end{equation}
Here, $(r_1, r_2, r_3)$ is source term of the original problem \eqref{eq4},
and $(\widetilde{r}^\epsilon_1, \widetilde{r}^\epsilon_2, \widetilde{r}^\epsilon_3)$
is constructed to ensure that the initial data $(u_0, \theta_0, h_0)$
also satisfies the compatibility conditions of \eqref{eq41} up to the order of $m$.
Indeed, we can use the given functions
$\partial_t^i(u, \theta, h)(0, x, y), 0\le i \le m$, which can be derived from the equations
and initial data of \eqref{eq4} by induction with respect to $i$,
and it follows that $\partial_t^i(u, \theta, h)(0, x, y)$ can be expressed as
polynomials of the spatial derivatives, up to order $2i$, of the initial data
$(u_0, \theta_0, h_0)$.
Then, similar to \cite{Liu-Xie-Yang}, one can choose
the corrector
$(\widetilde{r}^\epsilon_1, \widetilde{r}^\epsilon_2, \widetilde{r}^\epsilon_3)$
in the following form:
\begin{equation}
(\widetilde{r}_1^\epsilon, \widetilde{r}_2^\epsilon, \widetilde{r}_3^\epsilon)(t, x, y)
\triangleq -\sum_{i=0}^{m}
\frac{t^i}{i!}\partial_t^i\partial_x^2(u, \theta, h)(0, x, y),
\end{equation}
which, yields that by a direct calculation
\begin{equation}
\partial_t^i(u^\epsilon, \theta^\epsilon, h^\epsilon)(0, x, y)
=\partial_t^i(u, \theta, h)(0, x, y), \quad 0 \le i \le m.
\end{equation}
Similarly, we can derive that
$\psi^\epsilon \triangleq \partial_y^{-1} h^\epsilon$ satisfies
\begin{equation}
\partial_t \psi^\epsilon+
[(u^\epsilon+U\phi')\partial_x+(v^\epsilon-U_x \phi)\partial_y]\psi^\epsilon
+H_x \phi u^\epsilon
-\nu \partial_y^2 \psi^\epsilon
-\epsilon \partial_x^2 \psi^\epsilon=r_4^\epsilon,
\end{equation}
where
\begin{equation}
r_4^\epsilon \triangleq r_4+\epsilon\widetilde{ r}_4^\epsilon, \quad
\widetilde{r}_4^\epsilon \triangleq-\sum_{i=0}^m
\frac{t^i}{i!}\int_0^y \partial_t^i \partial_x^2 h(0, x, z)dz.
\end{equation}
Then, one attains directly for $\alpha=(\beta, k)=(\beta_1, \beta_2, k)$ with
$|\alpha|\le m$,
\begin{equation}\label{estimate8}
\|D^{\alpha}(\widetilde{r}^\epsilon_1, \widetilde{r}^\epsilon_2,
\widetilde{r}^\epsilon_3)\|_{L^2_{l+k}(\Omega)},~
\|\partial_\tau^\beta \widetilde{r}^\epsilon_4\|_{L^2_{-1}(\Omega)}
\le t^{i-\beta_1}\sum_{\beta_1 \le i \le m}\mathcal{P}
(M_0+\|(u_0, \theta_0, h_0)\|_{H^{2i+2+\beta_2+k}_l}).
\end{equation}

Now, we are can obtain the following proposition by the previous estimate.

\begin{proposition}\label{pro41}
Under the hypotheses of Theorem \ref{theo2}, there exist a positive time
$0<T_*\le T$, independent of $\epsilon$, and a solution
$(u^\epsilon, v^\epsilon, \theta^\epsilon, h^\epsilon, g^\epsilon)$
to the initial boundary value problem \eqref{eq41} with
$(u^\epsilon, \theta^\epsilon, h^\epsilon)\in L^\infty(0, T_*; \mathcal{H}^m_l)$,
which satisfies the following uniform estimates in $\epsilon$:
\begin{equation}\label{41}
\underset{0\le t\le T_*}{\sup}\|
(u^\epsilon, \theta^\epsilon, h^\epsilon)(t)\|_{\mathcal{H}^m_l}
\le 2F(0)^{\frac{1}{2}},
\end{equation}
where $F(0)$ is given by \eqref{functionf}.
Moreover, for $(t, x, y)\in [0, T_*]\times \Omega$, it holds on
\begin{equation}\label{42a}
|\langle y \rangle^{l+1}\partial_y^i(u, \theta, h)(t, x, y)|\le \delta_0^{-1},
\quad i=1,2,
\end{equation}
and
\begin{equation}\label{42b}
~h(t, x, y)+H(t, x)\phi'(y)\ge \delta_0.
\end{equation}
\end{proposition}
\begin{proof}
First of all, one can establish the a priori estimates as in Proposition \ref{pro2}
for the regularized boundary layer systems \eqref{eq41}.
Then, the standard continuity argument helps us obtain
the existence of solution in a time interval $[0, T_*], T_*>0$
independent of $\epsilon$.
Hence, the only task for us is to determine the uniform lifespan $T_*$,
and verify estimates \eqref{41}- \eqref{42b}.
Indeed, we apply the Proposition \ref{pro2} to get that
\begin{equation}\label{43}
\underset{0\le s \le t}{\sup}\|(u, \theta, h)(s)\|_{\mathcal{H}^m_l}
\le \left\{F(0)+\int_0^t G^\epsilon(\tau)d\tau\right\}^{\frac{1}{2}}
\left\{1-C\delta_0^{-8}~t~[F(0)+\int_0^t G^\epsilon(\tau)d\tau]^5\right\}^{-\frac{1}{10}},
\end{equation}
where the function $G^\epsilon(t)$ is defined as follows:
\begin{equation}\label{functiong1}
\begin{aligned}
G^\epsilon(t)
&\triangleq C\delta_0^{-8}(1+\sum_{|\beta|\le m+2}
\|\partial_\tau^\beta(U, \Theta, H, P)(t)\|_{L^2(\mathbb{T}_x)}^2)^5
+C\sum_{\tiny\substack{|\alpha|\le m \\ |\beta|\le m-1}}
\|D^\alpha(r^\epsilon_1, r^\epsilon_2, r^\epsilon_3)(t)\|_{L^2_{l+k}(\Omega)}^2\\
&\quad +C\delta_0^{-4}\sum_{|\beta|=m}
\left\{\|\partial_\tau^\beta (r^\epsilon_1, r^\epsilon_2, r^\epsilon_3)\|_{L^2_l(\Omega)}^2
+4\delta_0^{-4}\|\partial_\tau^\beta r^\epsilon_4\|_{L^2_{-1}(\Omega)}^2\right\}.\\
\end{aligned}
\end{equation}
Recalling the definition of $G(t)$(see \eqref{functiong}), it is easy to check that
\begin{equation*}
\begin{aligned}
G^\epsilon(t)
&=G(t)+C\epsilon^2\delta_0^{-4}\sum_{|\beta|=m}
\left\{\|\partial_\tau^\beta (\widetilde{r}^\epsilon_1, \widetilde{r}^\epsilon_2, \widetilde{r}^\epsilon_3)\|_{L^2_l(\Omega)}^2
+4\delta_0^{-4}\|\partial_\tau^\beta \widetilde{r}^\epsilon_4\|_{L^2_{-1}(\Omega)}^2\right\}\\
&\quad +C\epsilon^2\sum_{\tiny\substack{|\alpha|\le m \\ |\beta|\le m-1}}
\|D^\alpha(\widetilde{r}^\epsilon_1, \widetilde{r}^\epsilon_2, \widetilde{r}^\epsilon_3)(t)\|_{L^2_{l+k}(\Omega)}^2.
\end{aligned}
\end{equation*}
Then, the combination of \eqref{ap1} and \eqref{estimate8} yields immediately
\begin{equation}\label{44}
G^\epsilon(t)\le C\delta_0^{-8}M_0^{10}
+\epsilon^2 \delta_0^{-8}\mathcal{P}(M_0+\|(u_0, \theta_0, h_0)\|_{H^{3m+2}_l})
\le \delta_0^{-8}\mathcal{P}(M_0+\|(u_0, \theta_0, h_0)\|_{H^{3m+2}_l}).
\end{equation}
Therefore, we choose the existence time
\begin{equation*}
T_1 \triangleq \min\left\{\frac{\delta_0^8 F(0)}
{\mathcal{P}(M_0+\|(u_0, \theta_0, h_0)\|_{H^{3m+2}_l})},
\frac{3\delta_0^8}{128CF(0)^5}\right\}
\end{equation*}
in \eqref{43}, we can get the estimate \eqref{41} for all $T_*\le T_1$.

On the other hand, by virtue of the Proposition \ref{pro2},
it is easy to deduce the following bounds for
$\langle y \rangle^{l+1} \partial_y^i (u^\epsilon, \theta^\epsilon, h^\epsilon), i=1,2$ that
\begin{equation}\label{45}
\begin{aligned}
&\|\langle y\rangle^{l+1}\partial_y^i (u^\epsilon, \theta^\epsilon, h^\epsilon)
(t, x, y)\|_{L^\infty(\Omega)}\\
&\le \|\langle y\rangle^{l+1}\partial_y^i (u_0, \theta_0, h_0)(x, y)\|_{L^\infty(\Omega)}
+C t \underset{0\le s \le t}{\sup}
\|(u^\epsilon, \theta^\epsilon, h^\epsilon)(s)\|_{\mathcal{H}_l^5}\\
&\le (2\delta_0)^{-1}+2C F(0)^{\frac{1}{2}}t.
\end{aligned}
\end{equation}
Then, choosing the existence of time
\begin{equation*}
T_2 \triangleq \min \left\{T_1, \frac{1}{4C\delta_0 F(0)^{\frac{1}{2}}}\right\},
\end{equation*}
in \eqref{45}, one can find the estimate \eqref{42a}.
Similarly,  choosing the existence of time
\begin{equation*}
T_3 \triangleq \min \left\{T_2, \frac{\delta_0}{C^2(M_0+2F(0)^{\frac{1}{2}})^2}\right\},
\end{equation*}
then we apply the estimate in Proposition \ref{pro2} to deduce that
\begin{equation*}
\begin{aligned}
&h^\epsilon(t, x, y)+H(t,x)\phi'(y)\\
&\ge h_0(x, y)+H(t,x)\phi'(y)-C t \underset{0\le s \le t}{\sup}\|h(s)\|_{\mathcal{H}_0^3}\\
&\ge h_0(x, y)+(H(t,x)-H(0, x))\phi'(y)
-2CF(0)^{\frac{1}{2}}t\\
&\ge 2\delta_0-C(M_0+2F(0)^{\frac{1}{2}})t^{\frac{1}{2}}
\ge \delta_0,
\end{aligned}
\end{equation*}
where we have used the H\"{o}lder inequality.
Therefore, we find the lifespan $T_*=T_3$ and establish the estimates \eqref{41}
-\eqref{42b}, and consequently complete the proof of the Proposition \ref{pro41}.
\end{proof}

\begin{proof}[\textbf{Proof of Local Existence.}]
Indeed, we get the local existence of solutions
$(u^\epsilon, v^\epsilon, \theta^\epsilon, h^\epsilon, g^\epsilon)$
to the nonlinear MHD boundary layer problem \eqref{eq4} and
their uniform estimates in $\epsilon$.
Now, by letting $\epsilon \rightarrow 0$ one obtains the solution
to the original problem \eqref{eq4} by applying  some compactness argument.
Indeed, from the uniform estimates \eqref{41}, by the Lions-Aubin lemma
and the compact embedding of $H^m_l(\Omega)$ in $H^{m'}_{loc}(\Omega)$
for $m' < m$(see \cite[Lemma 6.2]{Masmoudi}), we know that there exists
$$
(u, \theta, h)\in L^\infty(0, T_*; \mathcal{H}^m_l)\cap
(\cap_{m'<m-1}C^1([0, T_*]; H^{m'}_{loc}(\Omega))),
$$
such that, up to a subsequeness,
\begin{equation}
\begin{array}{cc}
&\partial_t^i(u^\epsilon, \theta^\epsilon, h^\epsilon)
\stackrel{*}{\rightharpoonup}
\partial_t^i(u, \theta, h),
\quad {\rm in}\quad L^\infty(0, T_*; H^{m-i}_l(\Omega)),
\quad 0\le i \le m,\\
&(u^\epsilon, \theta^\epsilon, h^\epsilon)
\rightarrow (u, \theta, h)
\quad {\rm in}\quad
C^1([0, T_*];H^{m'}_{loc}(\Omega)).
\end{array}
\end{equation}
On the other hand, by virtue of
$(\partial_x u^\epsilon, \partial_x h^\epsilon) \in Lip(\Omega_{T_*})$,
we find the uniform convergence of $(\partial_x u^\epsilon, \partial_x h^\epsilon)$.
Then, we can obtain the the pointwise convergence for $(v^\epsilon, g^\epsilon)$, i.e.,
\begin{equation}
(v^\epsilon, g^\epsilon)
=(-\int_0^y \partial_x u^\epsilon dz, -\int_0^y \partial_x h^\epsilon dz)
\rightarrow
(-\int_0^y \partial_x u dz, -\int_0^y \partial_x h dz)
\triangleq (v, g).
\end{equation}

Now, we can pass the limit $\epsilon \rightarrow 0$ in problem \eqref{eq41},
and obtain that $(u, v, \theta, h, g)$, solves the original problem \eqref{eq4}.
By virtue of the definition of function space $\mathcal{H}_l^m$(see \eqref{bn}),
it is easy to get $(u, \theta, h)\in \cap_{i=0}^m W^{i,\infty}(0, T_*; H_l^{m-i}(\Omega))$
from the fact $(u, \theta, h)\in L^\infty(0, T_*; \mathcal{H}^m_l)$,
then one proves \eqref{211} directly.
On the other hand, the relation \eqref{212} follows directly by combining the
divergence free conditions
$v=-\partial_y^{-1}\partial_x u,~g=-\partial_y^{-1}\partial_x h$,
with \eqref{e7}. Therefore, we prove the local-in-time existence
of Theorem \ref{theo2}.
\end{proof}

\subsection{Uniqueness for the Boundary Layer Equations}
\quad In this subsection, we will give the uniqueness of the solution to
the nonlinear MHD boundary layer problem \eqref{eq4}.
Let $(u_1, v_1, \theta_1, h_1, g_1)$ and $(u_2, v_2, \theta_2, h_2, g_2)$
be two solutions in the lifespan $[0, T_*]$, constructed in the previous subsection,
with respect to the initial data $(u_{10}, \theta_{10}, h_{10})$ and
$(u_{20}, \theta_{20}, h_{20})$ respectively. Set
\begin{equation*}
(\widetilde{u}, \widetilde{v}, \widetilde{\theta}, \widetilde{h}, \widetilde{g})
\triangleq (u_1-u_2, v_1-v_2, \theta_1-\theta_2, h_1-h_2, g_1-g_2),
\end{equation*}
then we obtain the following systems:
\begin{equation}\label{4beq1}
\left\{
\begin{aligned}
&\partial_t \widetilde{u}
+[(u_1+U\phi')\partial_x+(v_1-U_x \phi)\partial_y]\widetilde{u}
-[(h_1+H\phi')\partial_x+(g_1-H_x \phi)\partial_y]\widetilde{h}\\
&\quad
-\mu \partial_y [(\theta_1+\Theta \phi'+1)\partial_y \widetilde{u}]
+(\partial_x u_2+U_x \phi')\widetilde{u}
+(\partial_y u_2+U \phi'')\widetilde{v}
-(\partial_x h_2+H_x \phi')\widetilde{h}\\
&\quad -(\partial_y h_2+H\phi'')\widetilde{g}
-\mu \partial_y(\widetilde{\theta} \partial_y u_2)
-U \phi^{(3)}\widetilde{\theta} -U\phi'' \partial_y \widetilde{\theta}=0,\\
&c_v \left[\partial_t \widetilde{\theta}
+(u_1+U\phi')\partial_x \widetilde{\theta}
+(v_1-U_x \phi)\partial_y \widetilde{\theta}\right]
-\kappa \partial_y^2 \widetilde{\theta}
+c_v (\partial_x \theta_2 + \partial_x \Theta \phi')\widetilde{u}
+c_v (\partial_y \theta_2+\Theta \phi'')\widetilde{v}\\
&\quad
-\mu \widetilde{\theta} [(U \phi'')^2+(\partial_y u_1)^2+2U \phi'' \partial_y u_1]
-\nu \partial_y \widetilde{h}(2H\phi'+\partial_y h_1+\partial_y h_2)
-\mu a_1 \partial_y \widetilde{u} =0,\\
&\partial_t \widetilde{h}
+[(u_1+U\phi')\partial_x +(v_1-U_x \phi)\partial_y]\widetilde{h}
-[(h_1+H\phi')\partial_x+(g_1-H_x \phi)\partial_y]\widetilde{u}
-\nu \partial_y^2 \widetilde{h}\\
&\quad
+(\partial_x h_2+H_x \phi')\widetilde{u}
+(\partial_y h_2+H \phi'')\widetilde{v}
-(\partial_x u_2+U_x \phi')\widetilde{h}
-(\partial_y u_2+U \phi'')\widetilde{g}=0,\\
&\partial_x \widetilde{u}+\partial_y \widetilde{v}=0, \quad
\partial_x \widetilde{h}+\partial_y \widetilde{g}=0,\\
&\left.(\widetilde{u}, \widetilde{\theta}, \widetilde{h})\right|_{t=0}
=(u_{10}-u_{20}, \theta_{10}-\theta_{20}, h_{10}-h_{20}),
\quad \left.(\widetilde{u}, \widetilde{v}, \widetilde{\theta},
\partial_y \widetilde{h}, \widetilde{g})\right|_{y=0}=0,
\end{aligned}
\right.
\end{equation}
where the function $a_1$ is defined by
\begin{equation*}
a_1 \triangleq 2\Theta U \phi' \phi''+2U\phi''
+\theta_2 \partial_y u_1 +\theta_2 \partial_y u_2
+2U \phi'' \theta_2
+\Theta \phi' \partial_y u_1+\Theta \phi' \partial_y u_2
+\partial_y u_1+\partial_y u_2.
\end{equation*}

On the other hand, denote
$\widetilde{\psi}=\partial_y^{-1} \widetilde{h}=\partial_y^{-1}(h_2-h_1)$,
then it is easy to check that $\widetilde{\psi}$ satisfies the following equation:
\begin{equation}\label{4beq2}
\partial_t \widetilde{\psi}
+[(u_1+U\phi')\partial_x+(v_1-U_x \phi)\partial_y]\widetilde{\psi}
-\nu \partial_y^2 \widetilde{\psi}
-(g_2-H_x \phi)\widetilde{u}
+(h_2+H\phi')\widetilde{v}=0.
\end{equation}
Let us introduce the following new quantities:
\begin{equation}\label{df1}
\overline{u} \triangleq \widetilde{u}-\eta_4 \widetilde{\psi}, \quad
\overline{\theta} \triangleq \widetilde{\theta}-\eta_5 \widetilde{\psi}, \quad
\overline{h} \triangleq \widetilde{h}-\eta_6 \widetilde{\psi},
\end{equation}
where
\begin{equation*}
\eta_4=\frac{\partial_y u_2+U\phi''}{h_2+H\phi'},\quad
\eta_5=\frac{\partial_y \theta_2+\Theta\phi''}{h_2+H\phi'},\quad
\eta_6=\frac{\partial_y h_2+H\phi''}{h_2+H\phi'}.
\end{equation*}
By virtue of the equations \eqref{4beq1}, \eqref{4beq2} and the definition \eqref{df1},
it is easy to verify that
$(\overline{u}, \overline{\theta}, \overline{h})$ admits the following problem:
\begin{equation}\label{4beq3}
\left\{
\begin{aligned}
& \partial_t \overline{u}
+[(u_1+U\phi')\partial_x+(v_1-U_x \phi)\partial_y]\overline{u}
-\mu \partial_y [(\theta_1+\Theta \phi'+1)\partial_y \overline{u}]
-\mu \partial_y [(\theta_1+\Theta \phi'+1) \partial_y \eta_4 \widetilde{\psi}]\\
&\quad -\mu \partial_y [(\theta_1+\Theta \phi'+1) \eta_4 \overline{h}]
-\mu \partial_y [(\theta_1+\Theta \phi'+1) \eta_4 \eta_6 \widetilde{\psi}]
-[(h_1+H\phi')\partial_x+(g_1-H_x \phi)\partial_y]\overline{h}\\
&\quad
+b_{1} \overline{u}+b_{2} \overline{\theta}+b_{3} \overline{h}
+c_1 \widetilde{\psi}
-(\mu \partial_y u_2+U \phi'') \partial_y \overline{\theta}=0,\\
& c_v[\partial_t
+(u_1+U\phi')\partial_x
+(v_1-U_x \phi)\partial_y ]\overline{\theta}-\kappa \partial_y^2 \overline{\theta}
+b_{4}\overline{u}+b_{5}\overline{\theta}+b_{6}\overline{h}
+c_2 \widetilde{\psi}-\mu a_1 \partial_y \overline{u}\\
&\quad +[(c_v \nu -\kappa)\eta_5
-\nu(2H \phi'+\partial_y h_1+\partial_y h_2)]\partial_y \overline{h}=0,\\
&\partial_t \overline{h}
+[(u_1+U\phi')\partial_x+(v_1-U_x \phi)\partial_y]\overline{h}
-[(h_1+H\phi')\partial_x+(g_1-H_x \phi)\partial_y]\overline{u}
-\nu \partial_y^2 \overline{h}\\
&\quad
+b_{7}\overline{u}+b_{8}\overline{h}+c_3 \widetilde{\psi}=0,
\end{aligned}
\right.
\end{equation}
where
\begin{equation*}
\begin{aligned}
& b_{1} \triangleq \partial_x u_2+U_x \phi'+(g_2-H_x\phi)\eta_4,
\quad b_{2}\triangleq -\mu \partial_y^2 u_2-U \phi^{(3)},\\
&b_{3}\triangleq \nu(\eta_4 \eta_6+2\partial_y \eta_4)
-(g_2-H_x \phi)\eta_6-(\partial_x h_2+H_x \phi')-2\nu \partial_y \eta_4
-(\mu \partial_y u_2+U \phi'')\eta_5,\\
&c_1 \triangleq \partial_t \eta_4
+[(u_1+U\phi')\partial_x+(v_1-U_x \phi)\partial_y]\eta_4
-\nu \partial_y^2 \eta_4
+\nu(\eta_4 \eta_6^2+\eta_4 \partial_y \eta_6
+2\eta_6 \partial_y \eta_4+\partial_y^2 \eta_4)\\
&\quad\quad
-[(h_1+H\phi')\partial_x+(g_1-H_x \phi)\partial_y]\eta_6
-(g_2-H_x \phi)\eta_6^2
+[\partial_x u_2+U_x \phi'+\eta_4(g_2-H_x \phi)]\eta_4\\
&\quad \quad
-(\partial_x h_2+H_x \phi'+2\nu \partial_y \eta_4)\eta_6
-(\mu \partial_y u_2+U\phi'')(\partial_y \eta_5+\eta_5 \eta_6)
-(\mu \partial_y^2 u_2+U\phi^{(3)})\eta_5,\\
&b_{4} \triangleq c_v[\partial_x \theta_2\!+\partial_x \Theta \phi'
\!+\eta_5(g_2-H_x \phi)\!+\eta_5(g_2-H_x \phi)],
b_{5} \triangleq -\mu[(U\phi'')^2+(\partial_y u_1)^2
+2U \phi'' \partial_y u_1],\\
&b_{6} \triangleq (c_v \nu-\kappa)(\eta_5 \eta_6+2\partial_y \eta_5)
-\nu \eta_6(2H\phi'+\partial_y h_1+\partial_y h_2)-\mu \eta_4 a_1,\\
&c_2 \triangleq (c_v \nu -\kappa)
(\eta_5\eta_6^2+\eta_5 \partial_y \eta_6+\partial_y^2 \eta_5+2\eta_6 \partial_y \eta_5)
+c_v \eta_4 [\partial_x \theta_2+\partial_x \Theta \phi'+\eta_5(g_2-H_x \phi)]\\
&\quad \quad -\mu \eta_5[(U\phi'')^2+(\partial_y u_1)^2+2U\phi'' \partial_y u_1]
-\mu(\eta_4 \eta_6+\partial_y \eta_4)a_1
-2c_v \nu \eta_6 \partial_y \eta_5
+c_v \eta_4 \eta_5 (g_2-H_x \phi)\\
&\quad \quad
\!+c_v (\partial_t \eta_5
\!+[(u_1+U\phi')\partial_x+(v_1-U_x \phi)\partial_y]\eta_5
\!-\nu \partial_y^2 \eta_5)
\!-\nu(\eta_6^2+\partial_y \eta_6)(2H\phi'+\partial_y h_1+\partial_y h_2),\\
&b_{7} \triangleq \partial_x h_2+H_x \phi'+\eta_6(g_2-H_x \phi),
\quad
b_{8}\triangleq -(\partial_x u_2+U_x \phi')
-2\nu \partial_y \eta_6-\eta_4(g_2-H_x \phi),\\
&c_3 \triangleq
\partial_t \eta_6+[(u_1+U\phi')\partial_x +(v_1-U_x \phi)\partial_y]\eta_6
-\nu \partial_y^2 \eta_6
-[(h_1+H\phi')\partial_x+(g_1-H_x \phi)\partial_y]\eta_4\\
&\quad \quad -2\nu \eta_6 \partial_y \eta_6
+(\partial_x h_2+H_x \phi')\eta_4
-(\partial_x u_2+U_x \phi')\eta_6.
\end{aligned}
\end{equation*}
Furthermore, we can also obtain the following boundary condition
\begin{equation}\label{bc5}
\left.(\overline{u}, \overline{\theta}, \partial_y \overline{h})\right|_{y=0}=0,
\end{equation}
and the initial data
\begin{equation}\label{initial}
\left\{
\begin{aligned}
&\overline{u}(0, x, y)
=u_{10}-u_{20}-
\frac{\partial_y u_{20}+U(0, x)\phi''(y)}{h_{20}+H(0, x)\phi'(y)}
\partial_y^{-1}(h_{10}-h_{20}),\\
&\overline{\theta}(0, x, y)
=\theta_{10}-\theta_{20}-
\frac{\partial_y \theta_{20}+\Theta(0, x)\phi''(y)}{h_{20}+H(0, x)\phi'(y)}
\partial_y^{-1}(h_{10}-h_{20}),\\
&\overline{h}(0, x, y)
=h_{10}-h_{20}-
\frac{\partial_y h_{20}+H(0, x)\phi''(y)}{h_{20}+H(0, x)\phi'(y)}
\partial_y^{-1}(h_{10}-h_{20}).
\end{aligned}
\right.
\end{equation}
\quad Furthermore, similar to \cite{Liu-Xie-Yang},
it is easy to deduce from \eqref{df1} that
\begin{equation*}
\overline{h}=(h_2+H\phi')\partial_y \left\{\frac{\widetilde{\psi}}{h_2+H\phi'}\right\},
\end{equation*}
which, together with the homogeneous boundary condition(i.e., $\widetilde{\psi}|_{y=0}=0$),
yields directly
\begin{equation}\label{df2}
\widetilde{\psi}(t, x, y)=(h_2(t, x, y)+H(t, x)\phi'(y))
\int_0^y \frac{\overline{h}(t, x, z)}{h_2(t, x, z)+H(t, x)\phi'(z)}dz.
\end{equation}
In view of $h_2+H\phi' \ge \delta_0$, then we applying inequality \eqref{e7} and
the representation \eqref{df2} to get
\begin{equation}\label{estimate5}
\left\|\langle y \rangle^{-1}{\widetilde{\psi}(t)}\right\|_{L^2(\Omega)}
\le 2\delta_0^{-1}\|h_2+H\phi'\|_{L^\infty([0, T_*]\times \Omega)}
\|\overline{h}(t)\|_{L^2(\Omega)}.
\end{equation}
On the other hand, similar to \eqref{estimate4}, one can get that
there exists a constant
\begin{equation*}
C=C(T_*, \delta_0, \phi, U, \Theta, H,
\|(u_1, \theta_1, h_1)\|_{\mathcal{H}_l^5},
\|(u_2, \theta_2, h_2)\|_{\mathcal{H}_l^5})>0,
\end{equation*}
such that
\begin{equation}\label{estimate6}
\|a_1\|_{L^\infty([0, T_*]\times \Omega)},~
\|b_i\|_{L^\infty([0, T_*]\times \Omega)},~
\|\langle y \rangle c_j\|_{L^\infty([0, T_*]\times \Omega)}, \le C ,\quad
i=1,2,...,8, ~j=1,2,3,
\end{equation}
which, together with \eqref{estimate5}, yields directly
\begin{equation}\label{estimate7}
\|(c_j \widetilde{\psi})(t)\|_{L^2(\Omega)}\le C\|\overline{h}(t)\|_{L^2(\Omega)},
\quad j=1,2,3.
\end{equation}

Now, we can establish the following proposition for the quantity
$(\overline{u}, \overline{\theta}, \overline{h})$
that will play an important in giving the uniqueness of solution
for nonlinear MHD boundary layer problem \eqref{eq4}.
\begin{proposition}\label{pro42}
Let $(u_1, v_1, \theta_1, h_1, g_1)$ and $(u_2, v_2, \theta_2, h_2, g_2)$ be two
solutions of the problem \eqref{eq4} with respect to the initial data
$(u_{10}, \theta_0, h_{10})$ and $(u_{20}, \theta_0, h_{20})$ respectively,
satisfying that $(u_i, \theta_i, h_i)\in \cap_{i=0}^m
W^{i,\infty}(0, T_*; H_l^{m-i}(\Omega))$
for $m \ge 5, i=1,2$. Then, there exists a positive constant
\begin{equation}\label{421}
C=C(T_*, \delta_0, \phi, U, \Theta, H,
\|(u_1, \theta_1, h_1)\|_{\mathcal{H}_l^5},
\|(u_2, \theta_2, h_2)\|_{\mathcal{H}_l^5})>0,
\end{equation}
such that for the quantity given by satisfying the following differential inequality:
\begin{equation}\label{422}
\frac{d}{dt}\|(\overline{u}, \overline{\theta}, \overline{h})(t)\|_{L^2(\Omega)}^2
+\|(\partial_y \overline{u},\partial_y \overline{\theta},
\partial_y \overline{h})(t)\|_{L^2(\Omega)}^2
\le C\|(\overline{u}, \overline{\theta}, \overline{h})(t)\|_{L^2(\Omega)}^2.
\end{equation}
\end{proposition}
\begin{proof}
Multiplying equations \eqref{4beq3}$_1$ and \eqref{4beq3}$_3$
by $\overline{u}$ and  $\overline{h}$ respectively,
and integrating by part, we find
\begin{equation*}
\begin{aligned}
&\frac{1}{2}\frac{d}{dt}\int_\Omega (|\overline{u}|^2+|\overline{h}|^2) dxdy
+\mu \int_\Omega (\theta_1+\Theta \phi'+1)|\partial_y \overline{u}|^2 dxdy
+\nu \int_\Omega |\partial_y \overline{h}|^2 dxdy\\
&=-\mu \int_\Omega (\theta_1+\Theta \phi'+1)
(\partial_y \eta_4 \widetilde{\psi}
+\eta_4 \eta_6 \widetilde{\psi}+\eta_4 \overline{h})\partial_y \overline{u} dxdy
-\int_\Omega (\mu \partial_y u_2+U\phi'')\partial_y \overline{\theta} \cdot \overline{u} dxdy\\
&\quad -\int_\Omega (b_{1}\overline{u}+b_{2}\overline{\theta}
+b_{3}\overline{h}+c_1 \widetilde{\psi})\cdot \overline{u} dxdy
-\int_\Omega (b_{7}\overline{u}+b_{8}\overline{h}+c_3\widetilde{\psi})\cdot \overline{h} dxdy.
\end{aligned}
\end{equation*}
By using lower bound estimate for temperature \eqref{321},
the estimates \eqref{estimate6}-\eqref{estimate7}
and Cauchy inequality, it is easy to check that
\begin{equation}\label{424}
\begin{aligned}
&\frac{1}{2}\frac{d}{dt}\int_\Omega (|\overline{u}|^2+|\overline{h}|^2) dxdy
+\mu \int_\Omega |\partial_y \overline{u}|^2 dxdy
+\nu \int_\Omega |\partial_y \overline{h}|^2 dxdy\\
&\le \frac{\mu}{2}\int_\Omega |\partial_y \overline{u}|^2  dxdy
+\frac{\kappa}{2}\int_\Omega |\partial_y \overline{\theta}|^2  dxdy
+C\int_\Omega|(\overline{u}, \overline{\theta}, \overline{h})|^2 dxdy.
\end{aligned}
\end{equation}
Multiplying the equation \eqref{4beq3}$_2$
by $\overline{\theta}$ and integrating by part, one arrives at directly
\begin{equation*}
\frac{c_v}{2}\frac{d}{dt}\int_\Omega |\overline{\theta}|^2 dxdy
+\kappa \int_\Omega |\partial_y \overline{\theta}|^2dxdy
=\mu \int_\Omega a_1 \partial_y \overline{u}\cdot \overline{\theta} dxdy
-\int_\Omega(b_{4}\overline{u}+b_{5}\overline{\theta}
+b_{6}\overline{h}+c_2 \widetilde{\psi})\cdot \overline{\theta} dxdy.
\end{equation*}
Then, we apply the Cauchy inequality and the estimates \eqref{estimate6}-\eqref{estimate7}
to get that
\begin{equation*}
\frac{c_v}{2}\frac{d}{dt}\int_\Omega |\overline{\theta}|^2 dxdy
+\kappa \int_\Omega |\partial_y \overline{\theta}|^2dxdy
\le \frac{\mu}{2}\int_\Omega |\partial_y \overline{u}|^2 dxdy
+C\int_\Omega|(\overline{u}, \overline{\theta}, \overline{h})|^2 dxdy,
\end{equation*}
which, together with the inequality \eqref{424}, yields the estimate \eqref{422}.
Therefore, we complete the proof of Proposition \ref{pro42}.
\end{proof}

\begin{proof}[\textbf{Proof of Uniqueness.}]
Indeed, if the initial data satisfying
$(u_{10}, \theta_{10}, h_{10})=(u_{20}, \theta_{20}, h_{20})$,
then we deduce from the representation \eqref{initial} that
$(\overline{u}, \overline{\theta}, \overline{h})$ admits
the zero initial data
$\left.(\overline{u}, \overline{\theta}, \overline{h})\right|_{t=0}=0$.
Then, we apply the Gr\"{o}nwall inequality to \eqref{422} to get that
$(\overline{u}, \overline{\theta}, \overline{h})=0$.
Putting $\overline{h}\equiv 0$ into the representation \eqref{df2}, one get
that $\widetilde{\psi} \equiv 0$.
By direct calculation, one arrives at directly
\begin{equation*}
(u_1, \theta_1, h_1)-(u_2, \theta_2, h_2)
=(\overline{u}, \overline{\theta}, \overline{h})+(\eta_4, \eta_5, \eta_6)\psi=0,
\end{equation*}
which yields that $(u_1, \theta_1, h_1)\equiv(u_2, \theta_2, h_2)$.
Finally, in view of the relation
\begin{equation*}
v_i=-\partial_y^{-1}\partial_x u_i, \quad
g_i=-\partial_y^{-1}\partial_x h_i, \quad    i=1,2,
\end{equation*}
we find the uniqueness of solution of the nonlinear MHD boundary layer problem \eqref{eq4}.
\end{proof}

\appendix

\section{Calculus Inequalities}\label{appendixA}

\quad In this appendix, we will introduce some basic inequality that
be used frequently in this paper. For the proof in detail,
the interested readers can refer to \cite{Liu-Xie-Yang}.

\begin{lemma}
For proper functions $f, g, h$, the following holds.\\
{\rm (i)}If $\underset{y\rightarrow +\infty}{\lim}(fg)(x,y)=0$, then
\begin{equation}\label{e1}
\left|\int_{\mathbb{T}_x}\left.(fg)\right|_{y=0}\right|
\le \|\partial_y f\|_{L^2(\Omega)}\|g\|_{L^2(\Omega)}
+\|f\|_{L^2(\Omega)}\|\partial_y g\|_{L^2(\Omega)}.
\end{equation}
In particular, if $\underset{y\rightarrow +\infty}{\lim}f(x,y)=0$, then
\begin{equation}\label{e2}
\|\left.f\right|_{y=0}\|_{L^2(\mathbb{T}_x)}
\le \sqrt{2}\|f\|_{L^2(\Omega)}^{\frac{1}{2}}
\|\partial_y f\|_{L^2(\Omega)}^{\frac{1}{2}}.
\end{equation}
{\rm (ii)}If $l\in \mathbb{R}$ and an integer $m \ge 3$, any
$\alpha=(\beta, k)\in \mathbb{N}^3,
\widetilde{\alpha}=(\widetilde{\beta}, \widetilde{k})\in \mathbb{N}^3$
with $|\alpha|+|\widetilde{\alpha}|\le m$,
\begin{equation}\label{e3}
\|(D^\alpha f\cdot D^{\widetilde{\alpha}}g)(t, \cdot)\|_{L^2_{l+k+\widetilde{k}}(\Omega)}
\le C\|f(t)\|_{\mathcal{H}_{l_1}^m}\|g(t)\|_{\mathcal{H}_{l_2}^m},
~{\rm for~all}~l_1, l_2 \in \mathbb{R},~~l_1+l_2=l.
\end{equation}
{\rm (iii)}For any $\lambda > \frac{1}{2}, \widetilde{\lambda}>0$,
\begin{equation}\label{e4}
\|{\langle y\rangle}^{-\lambda}(\partial_y^{-1}f)(y)\|_{L^2_y(\mathbb{R}_+)}
\le \frac{2}{2\lambda-1}\|{\langle y\rangle}^{1-\lambda}f(y)\|_{L^2_y(\mathbb{R}_+)}
\end{equation}
and
\begin{equation}\label{e5}
\|{\langle y\rangle}^{-\widetilde{\lambda}}(\partial_y^{-1}f)(y)\|_{L^\infty_y(\mathbb{R}_+)}
\le \frac{1}{\widetilde{\lambda}}\|{\langle y\rangle}^{1-\widetilde{\lambda}}f(y)\|_{L^\infty_y(\mathbb{R}_+)}
\end{equation}
and then, for $l \in \mathbb{R}$, an integer $m \ge 3$,
and any $\alpha=(\beta, k)\in \mathbb{N}^3,
\widetilde{\beta}=(\widetilde{\beta}_1, \widetilde{\beta}_2)\in \mathbb{N}^2$
with $|\alpha|+|\widetilde{\beta}|\le m$,
\begin{equation}\label{e6}
\|(D^\alpha g\cdot \partial_\tau^{\widetilde{\beta}}\partial_y^{-1}h)(t,\cdot)\|_{L^2_{l+k}(\Omega)}
\le C\|g(t)\|_{\mathcal{H}^m_{l+\lambda}}\|h(t)\|_{\mathcal{H}^m_{1-\lambda}}.
\end{equation}
In particular, for $\lambda=1$,
\begin{equation}\label{e7}
\|{\langle y\rangle}^{-1}(\partial_y^{-1}f)(y)\|_{L^2_y(\mathbb{R}_+)}
\le 2\|f\|_{L^2_y(\mathbb{R}_+)},
\end{equation}
and
\begin{equation}\label{e8}
\|(D^\alpha g\cdot \partial_\tau^{\widetilde{\beta}}\partial_y^{-1}h)(t,\cdot)\|_{L^2_{l+k}(\Omega)}
\le C\|g(t)\|_{\mathcal{H}^m_{l+1}}\|h(t)\|_{\mathcal{H}^m_{0}}.
\end{equation}
\end{lemma}

\section{Almost Equivalence of Weighted Norms}\label{appendixB}

\quad In this subsection, we give the almost equivalence in $L_l^2-$norm
between $\partial_\tau^\beta(u, \theta, h)$ and the quantity
$(u_\beta, \theta_\beta, h_\beta)$ defined in \eqref{function3}.
\begin{lemma}\label{equivalent}
If the smooth function $(u, \theta, h)$ satisfies the nonlinear problem \eqref{eq4}
in $[0, T]$, and the assumption condition
\eqref{p1} holds on, then for any $t \in [0, T]$, any real number $l \ge 0$, an integer
$m \ge 3$ and the quantity $(u_\beta, \theta_\beta, h_\beta)$
with $|\beta|=m$ defined by \eqref{function3}, we have the following relations
\begin{equation}\label{d1}
M(t)^{-1}\|\partial_\tau^\beta(u, \theta, h)(t)\|_{L^2_l(\Omega)}
\le \|(u_\beta, \theta_\beta, h_\beta)(t)\|_{L^2_l(\Omega)}
\le M(t)\|\partial_\tau^\beta(u, \theta, h)(t)\|_{L^2_l(\Omega)}
\end{equation}
and
\begin{equation}\label{d2}
\|\partial_y \partial_\tau^\beta (u, \theta, h)(t)\|_{L^2_l(\Omega)}
\le \|\partial_y(u_\beta, \theta_\beta, h_\beta)(t)\|_{L^2_l(\Omega)}
+M(t)\|h_\beta(t)\|_{L^2_l(\Omega)},
\end{equation}
where the function $M(t)$ is defined by \eqref{mt}.
\end{lemma}
\begin{proof}
Firstly, by the definition of \eqref{function3}, we find
\begin{equation}\label{d3}
\begin{aligned}
\|\theta_\beta\|_{L^2_l(\Omega)}
&\le \|\partial_\tau^\beta \theta\|_{L^2_l(\Omega)}
     +\|\langle y\rangle^{l+1}\eta_2\|_{L^\infty(\Omega)}
     \|\langle y\rangle^{-1}\partial_\tau^\beta \psi\|_{L^2(\Omega)}\\
&\le \|\partial_\tau^\beta \theta\|_{L^2_l(\Omega)}
     +2\delta_0^{-1}
     (\|\langle y\rangle^{l+1}\partial_y \theta\|_{L^\infty(\Omega)}
     +C\|\Theta\|_{L^\infty(\mathbb{T}_x)})
     \|\partial_\tau^\beta h\|_{L^2(\Omega)}\\
&\le M(t)\|\partial_\tau^\beta(\theta, h)\|_{L^2_l(\Omega)}.
\end{aligned}
\end{equation}
Similarly, we obtain the following estimate(or see Liu et al.\cite{Liu-Xie-Yang})
\begin{equation}\label{d4}
\|(u_\beta, h_\beta)\|_{L^2_l(\Omega)}
\le M(t)\|\partial_\tau^\beta(u, h)\|_{L^2_l(\Omega)}.
\end{equation}
On the other hand, in view of the definition of $h_\beta$ in \eqref{function3},
one attains
\begin{equation*}
h_\beta=\partial_\tau^\beta h-\eta_3 \partial_\tau^\beta \psi
=(h+H\phi')\partial_y \left\{\frac{\partial_\tau^\beta \psi}{h+H\phi'}\right\},
\end{equation*}
which, together with the boundary condition $\psi|_{y=0}=0$, implies directly
\begin{equation}\label{d5}
\partial_\tau^\beta \psi=(h+H\phi')\int_0^y \frac{h_\beta}{h+H\phi'}d\widetilde{y}.
\end{equation}
Substituting \eqref{d5} into \eqref{function3}, one arrives at
\begin{equation}\label{d6}
\partial_\tau^\beta \theta
=\theta_\beta+(\partial_y \theta+\Theta \phi'')
\int_0^y \frac{h_\beta}{h+H\phi'}d\widetilde{y}.
\end{equation}
Then, it is easy to check that
\begin{equation}\label{d7}
\begin{aligned}
\|\partial_\tau^\beta \theta\|_{L^2_l(\Omega)}
&\le \|\theta_\beta\|_{L^2_l(\Omega)}
     +\|\langle y\rangle^{l+1}(\partial_y \theta+\Theta \phi'')\|_{L^\infty(\Omega)}
     \|\langle y\rangle^{-1}\int_0^y \frac{h_\beta}{h+H\phi'}\|_{L^2(\Omega)}\\
&\le \|\theta_\beta\|_{L^2_l(\Omega)}
     +2\delta_0^{-1}
     (\|\langle y\rangle^{l+1}\partial_y \theta\|_{L^\infty(\Omega)}
      +C\|\Theta\|_{L^\infty(\mathbb{T}_x)})
     \|h_\beta\|_{L^2(\Omega)}\\
&\le M(t)\|(\theta_\beta, h_\beta)\|_{L^2_l(\Omega)}.
\end{aligned}
\end{equation}
Furthermore, taking $y$ derivative to both handside of representation \eqref{d6}, we find
\begin{equation}\label{d8}
\partial_y \partial_\tau^\beta \theta
=\partial_y \theta_\beta
+(\partial_y^2 \theta+\Theta \phi^{(3)})\int_0^y \frac{h_\beta}{h+H\phi'}d\widetilde{y}
+\frac{h_\beta(\partial_y \theta+\Theta \phi'')}{h+H\phi'}.
\end{equation}
Hence, it is easy to deduce that
\begin{equation}\label{b9}
\begin{aligned}
\|\partial_y \partial_\tau^\beta \theta\|_{L^2_l(\Omega)}
&\le \|\partial_y \theta_\beta\|_{L^2_l(\Omega)}
+\|\langle y\rangle^{l+1}(\partial^2_y \theta+\Theta \phi^{(3)})\|_{L^\infty(\Omega)}
     \|\langle y\rangle^{-1}\int_0^y \frac{h_\beta}{h+H\phi'}d\tau\|_{L^2(\Omega)}\\
&\quad +\delta_0^{-1}\|\partial_y \theta+\Theta \phi''\|_{L^\infty(\Omega)}
       \|h_\beta\|_{L_l^2(\Omega)}\\
&\le \|\partial_y \theta_\beta\|_{L^2_l(\Omega)}
     +M(t)\|h_\beta\|_{L^2_l(\Omega)}.
\end{aligned}
\end{equation}
Similarly, we obtain the following estimates(or see \cite{Liu-Xie-Yang})
\begin{equation}\label{b10}
\|\partial_\tau^\beta(u, h)\|_{L^2_l(\Omega)}
\le M(t)\|(u_\beta, h_\beta)\|_{L^2_l(\Omega)},
\end{equation}
and
\begin{equation}\label{b11}
\|\partial_y \partial_\tau^\beta (u, h)\|_{L^2_l(\Omega)}
\le \|\partial_y (u_\beta, h_\beta)\|_{L^2_l(\Omega)}
     +M(t)\|h_\beta\|_{L^2_l(\Omega)}.
\end{equation}
Therefore, we  complete  the proof of the Lemma \ref{equivalent}.
\end{proof}

\section*{Acknowledgements}
The author Jincheng Gao would like to thank Chengjie Liu for fruitful discussions.
Jincheng Gao's research was partially supported by
Guangdong Natural Science Foundation (Grant No.2014A030313161),
China Postdoctoral Science Foundation Project(Grant No.2016M600064),
and NNSF of China(Grant No.11571380). Daiwen Huang's research was partially
supported by the NNSF of China(Grants No.11631008) and National Basic
Research Program of China 973 Program(Grants No. 2007CB814800).


\end{document}